\documentclass[11pt, a4paper, reqno, intlimits]{amsart}

\usepackage{amsmath,amsthm}
\usepackage[margin=3cm]{geometry}
\usepackage[active]{srcltx}
\usepackage{amssymb}
\usepackage{latexsym}
\usepackage{nicefrac}
\usepackage[all,cmtip,arrow,2cell]{xy}
\usepackage{enumitem}
\usepackage{cancel}
\usepackage{eucal}
\usepackage{cases}
\usepackage{comment}

\usepackage{tikz}
\usepackage{tikz-cd}
\usepackage{tikz-3dplot}
\usepackage{pgfplots}
\pgfplotsset{compat=1.16}
\usetikzlibrary{intersections}
\usetikzlibrary{positioning}
\usepgfplotslibrary{fillbetween}
\usetikzlibrary{arrows}
\usepackage{subcaption}
\usepackage[dvipsnames]{xcolor}
\usepackage{cite}
\usepackage{hyperref}
\usepackage{bbm}
\usepackage{mathtools}

\theoremstyle{plain}
\newtheorem{Theorem}{Theorem}[subsection]
\newtheorem*{Theorem*}{Theorem}
\newtheorem*{TheoremA*}{Theorem~A}
\newtheorem*{TheoremB*}{Theorem~B}
\newtheorem{Proposition}[Theorem]{Proposition}
\newtheorem{Lemma}[Theorem]{Lemma}
\newtheorem{Corollary}[Theorem]{Corollary}

\theoremstyle{definition}
\newtheorem{Remark}[Theorem]{Remark}
\newtheorem{Definition}[Theorem]{Definition}

\newtheorem{Example}[Theorem]{Example}

\newtheorem{Terminology}[Theorem]{Terminology}
\newtheorem{Warning}[Theorem]{Warning}
\newtheorem{Notation}[Theorem]{Notation}

\newtheorem*{Acknowledgement*}{Acknowledgement}
\newtheorem*{Outline*}{Outline of the paper}

\DeclareMathOperator{\Mat}{Mat}

\DeclareMathOperator{\Span}{Span}

\DeclareMathOperator{\Hom}{Hom}

\DeclareMathOperator{\Aut}{Aut}

\DeclareMathOperator{\Empty}{{\_\!\_\,}}

\DeclareMathOperator*{\Lim}{lim}
\DeclareMathOperator*{\Colim}{colim}

\DeclareMathOperator{\intHom}{\underline{\Hom}}

\DeclareMathOperator{\Ass}{asc}

\newcommand{\bbC}{\mathbb{C}}
\newcommand{\bbN}{\mathbb{N}}

\newcommand{\bbR}{\mathbb{R}}
\newcommand{\bbT}{\mathbb{T}}
\newcommand{\bbZ}{\mathbb{Z}}
\newcommand{\bbK}{\mathbb{K}}

\newcommand{\calA}{\mathcal{A}}

\newcommand{\calC}{\mathcal{C}}
\newcommand{\calD}{\mathcal{D}}

\newcommand{\calI}{\mathcal{I}}

\newcommand{\calK}{\mathcal{K}}

\newcommand{\calU}{\mathcal{U}}

\newcommand*{\longhookrightarrow}{\ensuremath{\lhook\joinrel\relbar\joinrel\rightarrow}}

\newcommand{\op}{{\mathrm{op}}}

\newcommand{\id}{\mathrm{id}}
\newcommand{\Id}{\mathrm{1}}
\newcommand{\pr}{\mathrm{pr}}
\newcommand{\Eval}{\mathrm{ev}}

\newcommand{\Set}{{\mathcal{S}\mathrm{et}}}

\newcommand{\Vect}{{\mathcal{V}\mathrm{ec}}}
\newcommand{\Mfld}{{\mathcal{M}\mathrm{fld}}}

\newcommand{\Alg}{{\mathcal{A}\mathrm{lg}}}
\newcommand{\Grpd}{{\mathcal{G}\mathrm{rpd}}}
\newcommand{\Mod}{{\mathcal{M}\mathrm{od}}}
\newcommand{\Morita}{{\mathcal{M}\mathrm{rt}}}
\newcommand{\AlgMrt}{{\Alg\Morita}}
\newcommand{\GrpdMrt}{{\Grpd\Morita}}

\newcommand{\R}{\mathbb{R}}

\newcommand{\N}{\mathbb{N}}

\newcommand{\im}{\text{im}}
\DeclareMathOperator{\Supp}{supp}

\DeclareMathOperator{\coker}{coker}

\newcommand{\textdef}[1]{\textbf{#1}}

\newcommand{\Born}{{\mathcal{B}\mathrm{orn}}}
\newcommand{\sBorn}{{\mathrm{s}\mathcal{B}\mathrm{orn}}}
\newcommand{\SBorn}{{\sBorn}}
\newcommand{\cBorn}{{\mathrm{c}\mathcal{B}\mathrm{orn}}} 
\newcommand{\CBorn}{{\cBorn}} 

\DeclareMathOperator{\intBorn}{\underline{\Born}}
\DeclareMathOperator{\intsBorn}{\underline{\sBorn}}
\DeclareMathOperator{\intcBorn}{\underline{\cBorn}}

\newcommand{\lcTVS}{\mathrm{lcTVS}}

\newcommand{\Disk}{{\mathcal{D}\mathrm{isk}}} 

\newcommand{\GrpdBar}{{\mathrm{bar}}}

\DeclareMathOperator{\Sep}{Sep} 
\DeclareMathOperator{\Comp}{Comp} 

\DeclareMathOperator{\Conv}{Conv}

\DeclareMathOperator{\Botimes}{\otimes^\mathrm{b}}
\DeclareMathOperator{\Sotimes}{\otimes^\mathrm{s}}
\DeclareMathOperator{\Cotimes}{\hat{\otimes}}

\DeclareMathOperator{\Asc}{asc}
\DeclareMathOperator{\Unitl}{lun}
\DeclareMathOperator{\Unitr}{run}
\DeclareMathOperator{\Ev}{Ev}
\DeclareMathOperator{\Coev}{Coev}

\DeclareMathOperator{\vN}{vN} 

\DeclareMathOperator{\Ext}{Ext}

\newcommand{\di}{\mathrm{d}}

\newcommand{\RHom}{\R \Hom}

\renewcommand{\phi}{\varphi}
\renewcommand{\epsilon}{\varepsilon}

\begin{document}

\title[Functoriality of bornological groupoid convolution]{Functoriality of bornological\\ groupoid convolution}

\author[D.~Aretz]{David Aretz}
\address{Max-Planck-Institut f\"ur Mathematik, Vivatsgasse 7, 53111 Bonn, Germany}
\email{aretz@mpim-bonn.mpg.de}

\author[C.~Blohmann]{Christian Blohmann}
\address{Max-Planck-Institut f\"ur Mathematik, Vivatsgasse 7, 53111 Bonn, Germany}
\email{blohmann@mpim-bonn.mpg.de}




\begin{abstract}
We show that the complete bornological convolution algebras of Lie groupoids and convolution bimodules of groupoid bibundles define a monoidal functor from the 2-category of differentiable stacks to the Morita 2-category of complete bornological algebras. The convolution algebras are generally non-unital, but are shown to possess one-sided approximate units such that the multiplication operators Mackey converge in the functional bornology of endomorphisms. This implies that the convolution algebras are self-induced and the convolution modules are smooth in the sense of R.~Meyer. We also show that Lie groupoid actions that are submersive, proper, and transitive have projective convolution modules. This implies that all convolution algebras are quasi-unital. We provide a long list of examples and applications, such as to bornological noncommutative tori, which are Hopf monoids in the Morita category.
\end{abstract}
\maketitle


\tableofcontents

\section{Introduction}

It is often suggested that the convolution algebra of a Lie groupoid encodes the noncommutative geometry of its generally singular space of orbits \cite{Connes1994}. The orbit space is described as the differentiable stack presented by the groupoid. Convolution should then be a functor from the 2-category of differentiable stacks to the Morita 2-category of algebras, bimodules, and bilinear maps. Two issues arise:

1.~Functoriality requires functional analytic completion. Mapping a Lie groupoid $G = (G_1 \rightrightarrows G_0)$ to the space of compactly supported smooth functions $A := C_\mathrm{c}^\infty(G_1)$ with the convolution product on the \emph{algebraic} tensor product is not functorial. $A$ is generally non-unital and may fail to represent the identity bimodule in the Morita category, $A \otimes_A M \not\cong M$ for an $A$-$B$ bimodule $M$. This issue is fixed partially by $C^*$-completion. The work of many authors \cite{Renault1980, MuhlyRenaultWillams1987, Mrcun:1999, Landsman2000, Landsman:2001, CrainicMoerdijk:2001} has yielded a 2-functor, however only from a \emph{subcategory} of differentiable stacks to the 2-category of $C^*$-algebras, Hilbert bimodules, and intertwiners \cite[Theorem~3.1.3]{Nuiten:2013}. The 1-morphisms of the subcategory are required to be given by groupoid bibundles that are proper and submersive on their codomain. This excludes such basic morphisms as the points of a stack, which are given by the inclusion of gauge groupoids over the orbits, and the stacky group structures of irrational torus foliations \cite{Blohmann2008}.

2.~The $C^*$-completion of the convolution algebra forgets the smooth structure of the groupoid and, therefore, the differentiable structure of the stack. For example, given a smooth manifold $X$, the $C^*$-completion of $C_\mathrm{c}^\infty(X)$ is the algebra $C_0(X)$ of continuous functions vanishing at $\infty$. While the derivations of $C_\mathrm{c}^\infty(X)$ are the vector fields of $X$, $C_0(X)$ has no derivations other than $0$. More generally, since derivations correspond to degree 1 Hochschild cocycles, the loss of the differentiable structure manifests itself in the vanishing of Hochschild cohomology of most $C^*$-algebras (when they are nuclear or do not admit a bounded trace) \cite{Connes:1978,ChristensenSinclair:1989}.

We will solve both issues by bornological completion, which makes convolution functorial and monoidal, yet does not change the underlying vector spaces of smooth functions. Our main result is:

\begin{TheoremA*}[Theorem~\ref{thm:ConvFunc}]
The bornological completion of the convolution algebras of Lie groupoids and of the convolution bimodules of right principal bibundles, together with additional functoriality data, defines a symmetric monoidal $2$-functor 
\begin{equation*}
    \GrpdMrt \longrightarrow \AlgMrt
\end{equation*}
from the geometric Morita $2$-category of Lie groupoids to the Morita $2$-category of non-unital complete bornological algebras. In particular, Morita equivalent Lie groupoids are mapped to Morita equivalent bornological algebras. 
\end{TheoremA*}

Since $\GrpdMrt$ is equivalent to the 2-category of differentiable stacks \cite[Theorem~2.18]{Blohmann2008}, the complete bornological convolution algebra of a Lie groupoid can be viewed as the smooth noncommutative geometry of the stack presented by the groupoid. As an immediate corollary, any additional algebraic structure on differentiable stacks given by a finitary multi-sorted Lawvere theory \cite{AdamekRosickyVitale:Algebraic} is translated into the corresponding structure on the convolution algebras. In particular, a differentiable group stack, presented by a stacky Lie group \cite{Blohmann2008}, is mapped to a Hopf monoid in $\AlgMrt$ (Section~\ref{sec:HopfMonoid}). This endows for example bornological noncommutative tori and their categories of modules with additional structure that does not exist in the $C^*$-algebraic setting. From a larger perspective, we posit that the symmetric monoidal functor of Theorem~A is a conceptual basis for the noncommutative geometry of stacks that can be used to establish connections to a variety of recent developments in mathematics, such as six functor formalisms or condensed mathematics.

The proof of Theorem~A is quite involved. By the way, we prove the following properties of the convolution algebra and convolution bimodules that are of independent interest: 

\begin{TheoremB*}[Theorem~\ref{thm:ApproxUnitExists}, Proposition~\ref{prop:ConvModProjective}, Corollary~\ref{cor:ConvMoritaProjective}, Corollary~\ref{cor:quasiunitality}]
The complete bornological convolution algebra of a Lie groupoid admits a strong left/right approximate unit and is quasi-unital. The complete bornological convolution module of a biprincipal groupoid bibundle is left and right projective.
\end{TheoremB*}

\subsection{Content}

We begin in Section~\ref{sec:bornology} by reviewing the category of bornological vector spaces $\Born$ as well as its reflective subcategories $\sBorn$ of separated and $\cBorn$ of complete bornological vector spaces,
\begin{equation*}
\begin{tikzcd}
  (\cBorn,\Cotimes) 
  \ar[r, shift right=1.2, hook] 
  \ar[r, phantom, "\scriptscriptstyle\boldsymbol{\bot}"]
  &
  (\sBorn, \Sotimes)
  \ar[l,shift right=1.2, "\Comp"']
  \ar[r, shift right=1.2, hook]
  \ar[r, phantom, "\scriptscriptstyle\boldsymbol{\bot}"]
  & 
  (\Born, \Botimes) 
  \ar[l,shift right=1.2, "\Sep"']
\end{tikzcd}
\,.
\end{equation*}
As a consequence of Day's reflection Theorem~\ref{thm:DaysReflection}, the tensor product $\Cotimes$ on $\cBorn$ is obtained by completing the algebraic tensor product $\Botimes$ with its natural bornology. Special attention is given to the bornology of mapping spaces and Mackey convergence, which will be used throughout the paper.

Section~\ref{sec:TestFuncBorn} contains a review of the bornology of the space $C_\mathrm{c}^\infty(X)$ of compactly supported smooth functions on a manifold $X$. The completed bornological tensor product is monoidal, $C_{\mathrm{c}}^\infty(X)\Cotimes C_{\mathrm{c}}^\infty(Y) \cong C_{\mathrm{c}}^\infty(X\times Y)$, which will imply that the convolution functor is monoidal. The section concludes with auxiliary lemmas for pullbacks and pushforwards needed for the proof of the main theorem.

In Section~\ref{sec:MoritaCategory}, we introduce the Morita $2$-category $\AlgMrt(\calC)$ of non-unital algebras in $\calC$. The algebras must be \emph{self-induced}, $A \otimes_A A \cong A$, and the $A$-$B$ bimodules \emph{smooth}, $A \otimes_A M \cong M$ and $M \otimes_B B \cong M$, so that the $A$-$A$ bimodule $A$ is the identity 1-morphism. We give useful criteria that guarantee these properties. In the case $\calC = \cBorn$, we show that they follow from the existence of strong approximate units.

In Section~\ref{sec:ConvolutionIngredients} we review the notion of smooth right principal groupoid bibundles and their composition, which are the 1-morphisms in the geometric Morita 2-category of Lie groupoids. We set up the integral formulas that define the convolution algebras and convolution bimodules, as well as the functoriality constraint in $\Born$. It is a straightforward consequence of our categorical setup that all these structures can be completed to structures on $\cBorn$, as summarized in Proposition~\ref{prop:CompletedAlgsMods}.

In Section~\ref{sec:functor} we state and prove Theorems~A and B. The longest part of the proof is to show that the bornological convolution algebra has a strong approximate left/right unit, which will make use of most of the concepts, results, and technical lemmas of the previous sections. 

The last Section~\ref{sec:Applications} contains a number of examples and applications of the bornological convolution functor. We show that bornological noncommutative tori are simple algebras by adapting the proof for $C^*$-algebras to the bornological setting. While most computations and results in this section are more or less immediate consequences of our main Theorem~\ref{thm:ConvFunc}, the analogous statements in the uncompleted smooth or the $C^*$-algebraic setting are often false or difficult to show.

\subsection{Related Work}

Bornological vector spaces were introduced by George Mackey and further developed in \cite{HogbeNlend1977}. R.~Meyer has demonstrated the convenience of bornology in homological algebra, for example in \cite{Meyer2004,Meyer2007, Meyer2011}. The present paper owes its motivation and many insights to his work. More recently, bornological methods have been developed in the context of derived analytic geometry \cite{benbassat2024}.

For an overview of Lie groupoids with a view towards noncommutative geometry, we refer to the text of Cannas da Silva and Weinstein \cite{CannasdaSilva1999}. Few papers have studied smooth groupoid convolution algebras. Crainic and Moerdijk prove the $H$-unitality of the smooth convolution algebra in \cite[Proposition~2]{CrainicMoerdijk:2001}, but their paper is agnostic of the functional analytic setting. The authors of \cite{Neumaier2006, OrbifoldCupProducts2011} study the Hochschild cohomology of the convolution algebra of orbifolds, that is, proper \'{e}tale Lie groupoids. This was extended in \cite{Pflaum2020} to proper Lie groupoids, using bornological language. In \cite{smoothGroupRep}, the bornological convolution algebra of Lie groups and their representation theory is studied, including a discussion of approximate units and smooth modules. In \cite{Posthuma23}, the authors construct a cochain map from Lie groupoid deformation cohomology to the Hochschild cohomology of the smooth convolution algebra.

This work builds on the Master's thesis~\cite{Aretz23} of D.A. advised by C.B.

\subsubsection{Prerequisites, notation}

We assume the reader to be familiar with basic category theory on the level of \cite{MacLane:Working} and basic notions from $2$-categories as in \cite{Benabou1967}. We will use the terms bicategory, $2$-category, and weak $2$-category interchangeably. All manifolds are assumed to be second countable and smooth. The coproduct of a family $\{V_i ~|~i \in I\}$ of bornological vector spaces will be denoted by $\bigoplus_{i \in I} V_i$. Since the category of (complete, separated) bornological vector spaces is additive, finite coproducts are biproducts, which motivates the notation.

\subsubsection*{Acknowledgements}

We are indebted to Ralf Meyer, David Miyamoto, Hessel Post\-huma, and Alan Weinstein for invaluable discussions and advice. D.A.~was supported by a scholarship of the Studienstiftung des deutschen Volkes. We thank the referee for helpful suggestions.

\section{Categories of bornological vector spaces}
\label{sec:bornology}
\subsection{Bornological vector spaces}

We recall the basic notions.

\begin{Definition}
A \textdef{bornology} on a set $X$ is a collection of subsets, called \textdef{bound\-ed}, such that the following holds:
\begin{itemize}

\item[(i)] Every subset of a bounded set is bounded.

\item[(ii)] The union of two bounded sets is bounded.

\item[(iii)] Every singleton $\{x\}$, $x \in X$ is bounded.

\end{itemize}
\end{Definition}
A set $X$ together with a bornology is called a \textdef{bornological space}. A map of sets with a bornology is called \textdef{bounded} if it sends bounded sets to bounded sets. The bounded maps are the morphisms of the category of bornological spaces. It can be shown that bornological spaces are a quasitopos \cite{AdamekHerrlich:1985}, which implies a number of good categorical properties.

\begin{Definition}
A \textdef{basis} of a bornological space $X$ is a collection of bounded sets $\calC$ such that every bounded set is the subset of a set in $\calC$.
\end{Definition}

\begin{Remark}
    Any collection of subsets of a set $X$ generates a bornology on $X$.
    A collection of subsets $\calC$ is a basis of the bornology it generates if and only if it satisfies the following properties:
    \begin{enumerate}
        \item[(ii)'] For $A,B\in \calC$ there is $C\in \calC$ such that $A\cup B\subset C$.
        \item[(iii)'] For all $x\in X$ there is a $C\in \calC$ with $x\in C$.
    \end{enumerate}
\end{Remark}

\begin{Example}
\label{ex:BanachBornology}
The eponymous example for a bornology is given by the bounded sets of a semi-normed, normed, or Banach space $(X,\lVert \Empty\rVert)$, which is called its \textdef{von Neumann bornology}.
Here, a subset $B\subset X$ is bounded if $\sup_{b\in B }\lVert b\rVert < \infty$.
\end{Example}

\begin{Definition}
A real or complex \textdef{bornological vector space} is a vector space internal to the category of bornological spaces, where the field is equipped with its natural bornology.     
\end{Definition}

The von Neumann bornology on a normed or Banach space has additional good properties. To describe these, we first recall some basic notions for real or complex vector spaces. Let $V$ be a $\bbK$-vector space, where $\bbK \in \{\bbR, \bbC\}$. For subsets $A, B \subset V$ and $\Lambda \subset \bbK$, we denote
\begin{align*}
  A + B 
  &:= \{ a + b ~|~ a \in A,~b \in B\}
  \\
  \Lambda\cdot A 
  &:= \{ \lambda a ~|~ \lambda \in \Lambda,~a \in A \}
  \,.
\end{align*}
For $\Lambda = \{\lambda\}$ a singleton, we write $\{\lambda\}\cdot A \equiv \lambda A$. Let
\begin{equation*}
  \bar{B}_1(\bbK) = \{\lambda \in \bbK ~|~ |\lambda| \leq 1\}
\end{equation*}
denote the closed unit ball in $\bbK$.

\begin{Remark}
\label{rmk:VecBornExplicit}
A bornology on the set underlying a vector space $V$ is a vector space bornology if and only if for all bounded $A$, $B$ and all scalars $\lambda \in \bbK$ the sets $A+B$, $\lambda A$, and $\bar{B}_1(\bbK) \cdot A$ are bounded \cite[Section~1:1.2, p.~19 and Exercise 2.E.1, p.~123]{HogbeNlend1977}.    
\end{Remark}

\begin{Definition}
\label{def:Disks1}
A subset $D$ of a real or complex vector space is called \textdef{circled} if $\bar{B}_1(\bbK) \cdot D \subset D$, and a \textdef{disk} if it is circled and convex.
\end{Definition}

\begin{Terminology}
Circled subsets are also called \textdef{balanced}. The set $\bar{B}_1(\bbK) \cdot D = \bigcup_{|\lambda| \leq 1} \lambda D$ is the smallest circled subset containing $D$. It is called the \textdef{circled hull} of $D$. A subset $D$ is a disk if and only if for $\mu, \nu \in \bbK$ such that $|\mu| + |\nu| \leq 1$, we have $\mu D + \nu D \subset D$. The set
\begin{equation*}
  D^\diamond := \bigl\{ \lambda_1 d_1 + \ldots + \lambda_n d_n ~|~
  n \geq 0\,; 
  d_1, \ldots, d_n \in D \,;
  {\textstyle\sum_{i=1}^n} |\lambda_i|  \leq 1
  \bigr\}
\end{equation*}
is the smallest disk containing $D$. It is called the \textdef{disked hull} of $D$.
\end{Terminology}

Recall that the \textdef{Minkowski functional} of an arbitrary subset $D \subset V$  is a function on $V$ defined by\begin{equation*}
  \| v \|_D 
  := \inf \{r \in \bbR~|~ r > 0\,;~v \in r D\}
  \,,
\end{equation*}
where the infimum of the empty set is defined to be $\infty$. Let $V_D$ denote the vector subspace generated by $D$.

\begin{Proposition}
If $D$ is a disk in a real or complex vector space, then the restriction of $\|\Empty\|_D$ to $V_D$ is a seminorm.
\end{Proposition}

\begin{Definition}
\label{def:Disks2}
A disk $D$ in a real or complex vector space $V$ is called \textdef{norming} if $\| \Empty \|_D$ is a norm on $V_D$. A norming disk $D$ is called \textdef{completant} if $V_D$ is complete.
\end{Definition}

\begin{Remark}
\label{rmk:NormingDiskClosed}
A norming disk $D \subset V$ is called \textdef{internally closed} if $D \subset V_D$ is closed in the norm topology. If the disk is internally closed, it is the closed unit ball in $V_D$. The internal closure of a norming disk is still a norming disk, so that we can assume without loss of generality that all norming disks are internally closed \cite{Meyer2007}.
\end{Remark}

\begin{Definition}
\label{def:Disks3}
A bornological vector space is called 
\begin{itemize}

\item[(i)] \textdef{convex} if every bounded set is contained in a bounded disk,

\item[(ii)] \textdef{separated} if every bounded set is contained in a bounded norming disk,

\item[(iii)] \textdef{complete} if every bounded set is contained in a bounded completant disk.

\end{itemize}
\end{Definition}

\begin{Remark}
By Remark~\ref{rmk:VecBornExplicit}, the circled hull of a bounded set is bounded. It follows that condition (i) of Definition~\ref{def:Disks3} is equivalent to the following: A bornological vector space is called
\begin{itemize}

\item[(i)'] \textdef{convex} if the convex hull of a bounded set is bounded.

\end{itemize}
\end{Remark}

\begin{Example}
The von Neumann bornology on a semi-normed space is separated if and only if the space is normed. The von Neumann bornology on a normed space is complete if and only if the space is norm complete, that is, Banach.
\end{Example}

\begin{Definition}
\label{def:Absorb}
A subset $A$ of a real or complex vector space \textdef{absorbs} another subset $S$ if there is a real number $r > 0$ such that $S \subset \lambda A$ for all scalars $\lambda$ satisfying $|\lambda| \geq r$.    
\end{Definition}

We can generalize Example~\ref{ex:BanachBornology} to arbitrary topological vector spaces.

\begin{Example}
\label{ex:vonNeumanBorn}
Let $V$ be a topological vector space. The \textdef{von Neumann bornology} is given by the sets that are bounded with respect to the Minkowski functionals $\| \Empty \|_U$ for all open neighborhoods $U$ of $0$. Equivalently, a subset $B$ is bounded if it is absorbed by every open neighborhood $U$ of $0$, that is, $B \subset rU$ for some $r < \infty$. Every singleton $\{v\}$ is bounded because every open neighborhood of $0$ is absorbing, so that $v \in rU$ for some $r > 0$. If $V$ is locally convex and $(p_i)_{i\in I}$ is a family of seminorms defining the topology, then a subset $B$ is bounded if and only if $p_i(B)$ is bounded in $\R$. 
\end{Example}

\begin{Definition}
\label{def:BornivorousTop}
A subset $A$ of a bornological vector space $V$ is called \textdef{bornivorous} if it absorbs every bounded subset. The \textdef{bornivorous topology} is the smallest vector space topology such that the balls $B_1(0) = \{ v \in V~|~ \|v\|_A < 1\}$ are open for all bornivorous $A$.
\end{Definition}

The bornivorous topology is the finest vector space topology such that the bornivorous sets form a neighborhood basis of $0$. A sequence $(v_n)_{n \in \bbN}$ converges in the bornivorous topology to $w \in V$ if and only if $\|v_n - w\|_A$ converges to zero for all bornivorous $A$. The following stronger notion of convergence is generally more useful and will be used throughout.


\begin{Definition}
\label{def:Mackey-Cauchy}
A sequence $(v_n)_{n \in \bbN}$ in a bornological vector space $V$ is \textdef{Mackey convergent (Mackey-Cauchy)} if there is a bounded disk $D \subset V$ such that the sequence is convergent (Cauchy) in the seminormed space $(V_D, \|\Empty\|_D)$.
\end{Definition}

Let us spell out Definition~\ref{def:Mackey-Cauchy} in more detail: $(v_n)_{n \in \bbN}$ Mackey converges to $v$ if for some bounded disk $D$ the following holds: $v \in V_D$ and for all $\epsilon > 0$ there exists an $N \in \bbN$ such that $v_n - v \in \epsilon D$ for all $n \geq N$. It is Mackey-Cauchy if for some bounded disk $D$ the following holds: For all $\epsilon > 0$ there exists an $N \in \bbN$ such that $v_n, v_m \in D$ and $v_n - v_m \in \epsilon D$ for all $n,m \geq N$.

A bornological vector space is separated if and only if the limit of each Mackey convergent sequence is unique. It is complete if and only if every bounded set is contained in a bounded disk $D$ for which every Mackey-Cauchy sequence with respect to $\|\Empty\|_D$ is also Mackey convergent in $\|\Empty\|_D$.
If a sequence is Mackey convergent, then it is also convergent in the bornivorous topology. The converse is wrong in general.

\begin{Proposition}
\label{prop:BoundedContinuous}
Bounded linear maps preserve Mackey convergence and the Mackey-Cauchy property of sequences.
\end{Proposition}
\begin{proof}
A bounded linear map $f:V \to W$ maps a bounded disk $D \subset V$ to the bounded disk $f(D)$ and $\epsilon D$ to $\epsilon f(D)$. It follows that, if $v \in \epsilon D$, then $f(v) \in \epsilon f(D)$. This shows that if the sequence $n \mapsto v_n \in V$ Mackey converges to $v$, then $n \mapsto f(v_n)$ Mackey converges to $f(v)$. It also shows that if the sequence is Mackey-Cauchy, then so is its image in $W$.
\end{proof}

\begin{Example}
In a Fr\'{e}chet space equipped with the von Neumann bornology, Mackey convergence is equivalent to convergence in the Fr\'{e}chet topology \cite[Theorem~1.36]{Meyer2007}.    
\end{Example}

\subsection{Limits and colimits of bornological vector spaces}

\begin{Definition}
We denote by $\Born$ the category of convex bornological vector spaces and bounded linear maps, by $\sBorn$ the full subcategory of separated convex bornological vector spaces, and by $\cBorn$ the full subcategory of complete convex bornological vector spaces.     
\end{Definition}

Every $\bbK$-vector space $V$ has a smallest and a largest convex bornology. In the largest convex bornology, called the \textdef{coarse bornology}, all subsets of $V$ are bounded. In the smallest convex bornology, called the \textdef{fine bornology}, a subset is bounded if it is contained in the circled hull of a finite subset of $V$. A subset is bounded in the fine bornology if and only if it is contained in a finite dimensional subspace $W \subset V$ and bounded in $W$ in the usual sense. This implies that the fine bornology is complete.

\begin{Proposition}
\label{prop:ForgetfulFunctors}
The forgetful functor $\Born \to \Vect$ is faithful and has a left and a right adjoint. The forgetful functors $\sBorn \to \Vect$ and $\cBorn \to \Vect$ are faithful and have left adjoints.
\end{Proposition}
\begin{proof}
By definition, a morphism in $\Born$, $\sBorn$, or $\cBorn$ is a linear map that is bounded. If two bounded maps are equal, then the linear maps are equal. This shows that the forgetful functors are faithful.

Every linear map $V \to W$ is bounded if we equip $W$ with the coarse bornology. We conclude that equipping a vector space with the coarse bornology is the right adjoint to the forgetful functor $\Born \to \Vect$.

Every linear map $f: V \to W$ maps the disked hull of $\{v_1, \ldots, v_p\}$ to the disked hull of $\{f(v_1), \ldots, f(v_p)\}$. Since every bornology on $W$ contains the fine bornology, it follows that $f$ is bounded if we endow $V$ with the fine bornology. Since the fine bornology is complete, we conclude that equipping a vector space $V$ with the fine bornology is the left adjoint of the forgetful functor $\cBorn \to \Vect$. Composing the left adjoint with the full and faithful inclusion $\cBorn \hookrightarrow \sBorn$ is the left adjoint to the forgetful functor $\sBorn \to \Vect$. Composing further with the full and faithful inclusion $\sBorn \hookrightarrow \Born$ yields the left adjoint to the forgetful functor $\Born \to \Vect$.
\end{proof}

Since the forgetful functor $\Born \to \Vect$ is a right adjoint, it preserves all limits. This means that, in order to compute the limit of a diagram $V: \calI \to \Born$, $i \mapsto V_i$, we first compute the limit in $\Vect$ and then equip it with the appropriate bornology: Let $\tau_i: \Lim_{i \in \calI}V_i \to V_i$ denote the maps of the limit cone. A subset $A \subset \Lim_i V_i$ is bounded in the \textdef{limit bornology} if and only if $\tau_i(A) \subset V_i$ is bounded for all $i \in \calI$. It is the largest bornology such that all $\tau_i$ are bounded.

If the bornology of every $V_i$ is separated or complete, then so is the limit bornology. Since the forgetful functors $\sBorn \to \Vect$ and $\cBorn \to \Vect$, too, are right adjoints, limits in $\sBorn$ and $\cBorn$ work in the same way: We compute a limit in $\Vect$ and then equip it with the limit bornology.

Since the forgetful functor $\Born \to \Vect$ is also a left adjoint, it preserves all colimits. In order to compute the colimit of a diagram $V: \calI \to \Born$, we first compute the colimit in $\Vect$ and then equip it with the appropriate bornology: Let $\sigma_i: V_i \to \Colim_{i \in \calI} V_i$ denote the maps of the colimit cocone. The \textdef{colimit bornology} is the convex vector space bornology generated by the images $\sigma_i(A_i)$ of all bounded subsets $A_i \subset V_i$ for all $i$. It is the smallest bornology such that all $\sigma_i$ are bounded.

Every limit can be obtained by computing a product followed by an equalizer \cite[Section V.2]{MacLane:Working} and dually for colimits. In an additive category such as $\Born$, we can replace the equalizer of two morphisms by the kernel of their difference. Analogously, their coequalizer is the difference cokernel. It is helpful to spell out the limit bornology of products, coproducts, quotients, and subspaces.

\begin{Proposition}
\label{prop:ProdCoprodBorn}
Let $(V_i)_{i \in I}$ be a family of convex bornological vector spaces. 
\begin{itemize}

\item[(i)] A subset $A \subset \prod_i V_i$ is bounded in the product bornology if and only if it is a subset of a product $\prod_i B_i$ of bounded subsets $B_i \subset V_i$.

\item[(ii)] A subset $A \subset \bigoplus_i V_i$ is bounded in the coproduct bornology if and only if it is a subset of the sum $\sigma_{i_1}(B_{i_1}) + \ldots + \sigma_{i_n}(B_{i_n})$ of the images of a finite number of bounded subsets $B_{i_1} \subset V_{i_1},\ldots, B_{i_n} \subset V_{i_n}$.

\item[(iii)] Let $i:V \to W$ be an injective linear map to $W \in \Born$. A subset $A \subset V$ is bounded in the subspace bornology if and only if $i(A)$ is bounded.

\item[(iv)] Let $p:V \to W$ be a surjective linear map from $V \in \Born$. A subset $A \subset W$ is bounded in the quotient bornology if and only if $A = p(B)$ is the image of a bounded subset $B \subset V$.

\end{itemize}
\end{Proposition}
\begin{proof}
See \cite[Chapter~1.3.1 and 1.3.5]{Meyer2007}.
\end{proof}

\begin{Lemma}
\label{lem:CoProdComplete}
Let $(V_i)_{i \in I}$ be a family of convex bornological vector spaces. If all $V_i$ are separated or complete, then so is the product bornology on $\prod_i V_i$ and the coproduct bornology on $\bigoplus_i V_i$.
\end{Lemma}
\begin{proof}
     See Section~1.3.5 in \cite{Meyer2007}.
\end{proof}

Lemma~\ref{lem:CoProdComplete} implies that the product of a family in $\sBorn$ or $\cBorn$ is given by the product of the vector spaces equipped with the product bornology. The analogous statement holds for coproducts. The statement for products also follows formally using the adjunctions of Section~\ref{sec:separationcompletion}.

\begin{Lemma}
\label{lem:SubBornSep}
Let $V \hookrightarrow W$ be an injective linear map to $W \in \Born$. If $W$ is separated, then so is the subspace bornology on $V$.
\end{Lemma}
\begin{proof}
See Corollary~(a) in Section~2:10 of \cite{HogbeNlend1977}.
\end{proof}

Lemma~\ref{lem:CoProdComplete} and Lemma~\ref{lem:SubBornSep} together imply that the limit of a diagram in $\sBorn$ is given by the limit of the vector spaces equipped with the limit bornology. Quotients of separated bornological vector spaces are generally not separated.

\begin{Lemma}
\label{lem:separatedcompletequotients}
Let $V \subset W$ be a bornological vector subspace of the bornological vector space $W$. Then $V$ is
separated if $W$ is separated, and $W/V$ is separated if and only if $V$ is closed in $W$. If $W$ is complete, then $V$ is complete if and only if it is closed. In this case, $W/V$ is complete as well.
\end{Lemma}
\begin{proof}
See Lemma~1.6.1 and also Sections~1.3.2 and 1.3.3 in \cite{Meyer2007}.
\end{proof}

\begin{Corollary}
\label{cor:CompleteCoker}
The cokernel of a morphism $f: V \to W$ of complete bornological vector spaces is given by the canonical epimorphism $W \to W/\overline{f(V)}$.
\end{Corollary}

We will need the following statement about convergence in the product bornology.

\begin{Lemma}
\label{lem:InfSeqConverge}
Let $v_i \in V_i$, $i \in \bbN$ be elements of a countable family of convex bornological vector spaces. Then the sequence $v_{\leq n} = (v_1, \ldots, v_n, 0, 0, \ldots) \in \prod_{i=1}^\infty V_i$ Mackey converges to $v = (v_1, v_2, \ldots)$.
\end{Lemma}
\begin{proof}
For all $i \in \bbN$, let $D_i$ be a bounded disk containing $v_i$. By Proposition~\ref{prop:ProdCoprodBorn}~(i), $D = \prod_{i=1}^\infty iD_i$ is a bounded disk in $V = \prod_{i=1}^\infty V_i$. It contains $v$ and $v_{\leq n}$ for all $n \in \bbN$. For the difference we obtain
\begin{equation*}
\begin{split}
  v_{\leq n} - v 
  &= (0, \ldots, 0, -v_{n+1}, -v_{n+2}, \ldots)
  \\
  &\in \frac{1}{n}D_1 \times \ldots 
  \times \frac{n}{n} D_n
  \times \frac{n+1}{n} D_{n+1} 
  \times \frac{n+2}{n} D_{n+2} 
  \times
  \ldots
  = \frac{1}{n}D
  \,. \qedhere
\end{split}
\end{equation*}
\end{proof}

Lemma~\ref{lem:InfSeqConverge} illustrates the dependence of Mackey convergence on the disk. While $v_{\leq n} - v \in D' = \prod_{i=1}^\infty D_i$, the sequence does not converge in the norm $\lVert \Empty \rVert_{D'}$. In the case $V_i = \bbR$, Lemma~\ref{lem:InfSeqConverge} shows that Mackey convergence in the product bornology of $\prod_{i=1}^\infty \bbR$ is the convergence of formal power series in $\bbR[[x]]$.

\subsection{Separation and completion}
\label{sec:separationcompletion}
The inclusion $I:\SBorn \hookrightarrow \Born$ has a left adjoint,
\begin{equation*}
\begin{tikzcd}
\sBorn
  \ar[r, shift right=1.2, hook, "I"']
  \ar[r, phantom, "\scriptscriptstyle\boldsymbol{\bot}"]
& 
\Born
  \ar[l,shift right=1.2, "\Sep"']
\,,
\end{tikzcd}
\end{equation*}
called the separation functor. In order to describe the functor explicitly, we note that a bornological vector space is separated if and only if $\{0\}$ is the only bounded vector subspace. This suggests that the separation of $V$ is given by the quotient of $V$ by the largest bounded vector subspace of $V$.

\begin{Proposition}
\label{prop:SepDivZero}
The separation functor is given by $\Sep(V) = V/\overline{\{0\}}$, where the quotient is taken in the category of bornological vector spaces. 
\end{Proposition}
\begin{proof}
See Proposition~(2) in Section~2:12 of \cite{HogbeNlend1977}.
\end{proof}

Any colimit in $\sBorn$ is obtained by computing the colimit in $\Born$ and then applying the separation functor. Since every completant disk is norming, every complete bornological vector space is separated. The inclusion $J:\CBorn \hookrightarrow \SBorn$ has a left adjoint,
\begin{equation}
\label{eq:CompJAdjoint}
\begin{tikzcd}
\cBorn
  \ar[r, shift right=1.2, hook, "J"']
  \ar[r, phantom, "\scriptscriptstyle\boldsymbol{\bot}"]
& 
\sBorn
  \ar[l,shift right=1.2, "\Comp"']
\,,
\end{tikzcd}
\end{equation}
called the completion functor. It can be constructed by formally adding limits of Mackey-Cauchy sequences.
More concisely, it can be defined in terms of the completion functor $(\Empty)^c$ from normed spaces to Banach spaces.
By definition, every separated bornological vector space $V$ may be written as a colimit of normed spaces $V = \Colim_{D\in \Disk(V)} V_D$, where $\Disk(V)$ denotes the filtered category of disks with absorption (Definition~\ref{def:Absorb}) as morphisms. The cocontinuous extension of the Banach completion functor then provides a completion functor \cite[Chapter 1.5]{Meyer2007}:
\begin{equation*}
  \Comp(V) = \Colim_{D \in \Disk(V)}
  (V_D)^c
  \,.
\end{equation*}

\begin{Remark}
The previous construction implicitly uses that $\cBorn$ is equivalent to the category of strict ind-Banach spaces, that is, the full subcategory of ind-Banach spaces consisting of objects represented by filtered diagrams of monomorphisms \cite[Chapter~1.5]{Meyer2007}.
\end{Remark}

\begin{Proposition}[Proposition~1.126 in \cite{Meyer2007}]
The categories $\Born$, $\sBorn$, and $\cBorn$ have all limits and colimits.
\end{Proposition}


\begin{Remark}
\label{rmk:vonNeumann}
The construction of the von Neumann bornology in Example~\ref{ex:vonNeumanBorn} defines a functor $\vN:\lcTVS\to \Born$. Equipping a bornological vector space with its bornivorous topology (Definition~\ref{def:BornivorousTop}) is the right adjoint $\Born\to \lcTVS$. The adjunction restricts to the subcategories of separated bornological spaces on the left and Hausdorff topological spaces on the right. $\vN$ maps complete locally convex vector spaces to complete bornological vector spaces.
\end{Remark}

\subsection{Tensor products}

The goal of this section is to define a tensor product on the category of complete bornological spaces. In a first step, we recall that the tensor product in the category $\Born$ of convex bornological vector spaces is defined by the usual universal property.

\begin{Definition}
\label{def:TensorBorn}
Let $V$ and $W$ be convex bornological vector spaces. Their tensor product, if it exists, is a convex bornological vector space $V \Botimes W$ together with a bounded bilinear map $i: V \times W \to V \Botimes W$ such that for all $U \in \Born$ and all bounded bilinear maps $\phi: V \times W \to U$ there is a unique bounded linear map $f: V \Botimes W \to U$ such that the diagram of the underlying sets
\begin{equation*}
\begin{tikzcd}
V \times W
\ar[r, "\phi"] \ar[d, "i"']
&
U
\\
V \Botimes W
\ar[ur, "\exists!f"', dashed]
\end{tikzcd}
\end{equation*}
commutes.
\end{Definition}

\begin{Proposition}
\label{prop:TensorBorn}
The tensor product $V\Botimes W$ of convex bornological vector spaces $V$ and $W$ is given by the algebraic tensor product $V \otimes W$ together with the bornology in which a subset is bounded if it is contained in the convex hull $\Conv(i(A\times B))$ for some bounded subsets $A \subset V$ and $B \subset W$.
\end{Proposition}
\begin{proof}
The proof is straightforward. See also Section~1.3.6 and Equation~(1.82) in \cite{Meyer2007}.
\end{proof}

\begin{Proposition}
The tensor product of Definition~\ref{def:TensorBorn} equips $\Born$ with a symmetric monoidal structure.
\end{Proposition}
\begin{proof}
The associativity and symmetry of the tensor product follows from the universal property of Definition~\ref{def:TensorBorn}.
\end{proof}

\begin{Proposition}
\label{prop:TensorConvergence}
Let $V$ and $W$ be bornological vector spaces. If a sequence $v_n \in V$ Mackey converges to $v$ and a sequence $w_n \in W$ Mackey converges to $w$, then $v_n\otimes w_n$ Mackey converges to $v \otimes w$ in $V \Botimes W$. 
\end{Proposition}
\begin{proof}
The proof follows by applying Lemma~\ref{lem:BilinearContinuous} to the bilinear map $i:V \times W \to V \Botimes W$.
\end{proof}

\begin{Lemma}
\label{lem:BilinearContinuous}
Let $U$, $V$, and $W$ be bornological vector spaces; let $f: V\times W \to U$ be a bounded bilinear map. If a sequence $v_n \in V$ Mackey converges to $v$ and a sequence $w_n \in W$ Mackey converges to $w$, then $f(v_n, w_n)$ Mackey converges to $f(v, w) \in U$.
\end{Lemma}
\begin{proof}
By definition there are bounded disks $D\subset V$ and $D'\subset W$ such that $v, v_n \in D$ and $w, w_n \in D'$, and for every $\epsilon > 0$ there is an $N \in \N$ such that $v_n-v\in \epsilon D$ and $w_n-w\in \epsilon D'$ for all $n\geq N$. By the boundedness of $f$, the set $f(D\times D')$ is bounded and hence lies in a bounded disk $B\subset U$. It follows that
\begin{equation*}
  f(v_n,w_n) - f(v,w)
  = f(v_n-v,w_n) + f(v,w_n-w)
\end{equation*}
lies in
\begin{equation*}
      f(\epsilon D \times D')
    + f(D\times \epsilon D')
    = 2\epsilon\, f(D \times D') 
    \subset 2\epsilon B
\end{equation*}
for all $n \geq N$. We conclude that $f(v_n,w_n)$ Mackey converges to $f(v,w)$.
\end{proof}

The tensor product on convex bornological vector spaces restricts to separated bornological vector spaces, as the next proposition shows. 

\begin{Proposition}
\label{prop:TensSepIsSep}
If two convex bornological vector spaces $V$ and $W$ are separated, then so is their tensor product $V \Botimes W$.
\end{Proposition}
\begin{proof}
We give a proof in Section~\ref{sec:TensSepIsSep} of the appendix.
\end{proof}

\begin{Notation}
We will denote by $\Sotimes: \sBorn \times \sBorn \to \sBorn$, $(V,W) \mapsto V \Sotimes W$ the restriction of $\Botimes$ to the subcategory of separated bornological vector spaces.    
\end{Notation}

\begin{Corollary}
$(\sBorn, \Sotimes)$ is a symmetric monoidal subcategory of $(\Born, \Botimes)$. The universal property of Definition~\ref{def:TensorBorn} restricts to the analogous universal property for $(\sBorn, \Sotimes)$.    
\end{Corollary}

It is not true that the tensor product $V \Sotimes W$ of complete bornological vector spaces is complete. Defining the tensor product of complete bornological vector spaces will be deferred to Section~\ref{sec:CompTensorProd}.

\subsection{Mapping spaces}

By definition, the \textdef{bornological mapping spaces} or \textdef{inner homs} $\intBorn(V,W)$ are an enrichment of the sets of morphisms $\Born(V,W)$ in $\Born$, such that there is an adjunction
\begin{equation}
\label{eq:TensorInnHomAdj}
\begin{tikzcd}
\Empty \Botimes V : \Born 
\ar[r, shift left=0.6ex]
& 
\Born 
\ar[l, shift left=.6ex] 
:\intBorn(V,\Empty)
\end{tikzcd}
\end{equation}
for all $V \in \Born$. If the inner homs exist, they are determined by this adjunction up to unique isomorphism. If a symmetric monoidal category has inner homs, it is called \textdef{closed}. The mapping space bornology is given explicitly as follows.

\begin{Proposition}
\label{prop:UnifBoundBorn}
Let $V$ and $W$ be convex bornological vector spaces. For every subset $F \subset \Born(V,W)$ and subset $A \subset V$, let
\begin{equation}
\label{eq:UnifBound}
  F(A) 
  := \{ f(a) \in W ~|~ f \in F \,,~ a \in A \}
  \,.
\end{equation}
There is a convex vector space bornology on $\Born(V,W)$ in which a subset $F$ is bounded if and only if for every bounded $A \subset V$ the subset $F(A) \subset W$ is bounded.
\end{Proposition}
\begin{proof}
For every bounded linear map $f:V \to W$ and any bounded subset $A \subset V$, $F(\{f\})(A) = f(A)$ is bounded, which shows that every singleton is bounded. Since $(F \cup F')(A) = F(A) \cup F'(A)$ it follows that the union of bounded subsets is bounded. We conclude that the bounded sets constitute a bornology on  the set $\Born(V, W)$.

The vector space structure on $\Born(V,W)$ is defined by $(f+g)(v) = f(v) + g(v)$ and $(\lambda f)(v) = \lambda( f(v))$ for $f, g \in \Born(V,W)$ and $\lambda \in \bbK$. It follows that
\begin{align*}
  (\lambda F)(A) 
  &= \lambda\bigl( F(A) \bigr)
  \\
  \bigl(\bar{B}_1(\bbK)\cdot F\bigr)(A) 
  &= \bar{B}_1(\bbK)\cdot \bigl( F(A) \bigr)
  \\
  (F + F')(A)
  &=
  F(A) + F'(A)
  \\
  \bigl( \Conv(F) \bigr)(A)
  &=
  \Conv\bigl( F(A) \bigr)
  \,,
\end{align*}
which shows that we have a convex vector space bornology.
\end{proof}

\begin{Terminology}
\label{term:equibounded}
A bounded set of maps in the functional bornology of Proposition~\ref{prop:UnifBoundBorn} is also called \textdef{equibounded}. The terminology comes from normed spaces, where a set of linear maps is equibounded if their operator norms have a common upper bound.
\end{Terminology}

\begin{Proposition}
\label{prop:MappinBornInnerHom}
The vector spaces $\Born(V,W)$ equipped with the bornology of Proposition~\ref{prop:UnifBoundBorn} are the inner homs $\intBorn(V,W)$ of the symmetric monoidal category $(\Born, \Botimes)$.
\end{Proposition}
\begin{proof}
We give a proof in Section~\ref{sec:MappinBornInnerHom} of the appendix.
\end{proof}

\begin{Proposition}
\label{prop:InnHomSep}
Let $V$ and $W$ be convex bornological vector spaces. If $W$ is separated, then so is $\intBorn(V,W)$.
\end{Proposition}

\begin{proof}
Let $W$ be separated. Assume that $\intBorn(V,W)$ is not separated, so that there is a non-zero bounded vector subspace $F \subset \intBorn(V,W)$. Since $F$ is non-zero, there is an $f \in F$ and a $v \in V$ such that $f(v) \neq 0$. Then the vector subspace $F(\{v\})$ as defined in Proposition~\ref{prop:UnifBoundBorn} is non-zero. Since $F$ is bounded and since the singleton $\{v\}$ is bounded, $F(\{v\}) \subset W$ is a bounded vector subspace, which is a contradiction to the assumption that $W$ is separated. We conclude that $\intBorn(V,W)$ is separated.
\end{proof}

Proposition~\ref{prop:InnHomSep} states that the inner hom functor $\intBorn: \Born^\op \times \Born \to \Born$ restricts to a functor $\intsBorn: \sBorn^\op \times \sBorn \to \sBorn$ defined by
\begin{equation*}
  I\bigl(\intsBorn(V,W)\bigr) = \intBorn(IV, IW)
  \,,
\end{equation*}
where $I: \sBorn \hookrightarrow \Born$ is the inclusion. 

\begin{Proposition}
\label{prop:InnsBorn}
$\intsBorn(V,W)$ is the inner hom of the symmetric monoidal category $(\sBorn, \Sotimes)$.
\end{Proposition}
\begin{proof}
We have the natural bijections
\begin{equation*}
\begin{split}
  \sBorn(U \Sotimes V, W)
  &\cong
  \Born\bigl( I(U \Sotimes V) , IW \bigr)
  \\
  &\cong
  \Born\bigl( IU \Botimes IV , IW \bigr)
  \\
  &\cong
  \Born\bigl( IU, \intBorn(IV, IW) \bigr)
  \\
  &\cong
  \Born\bigl( IU, I(\intsBorn(V, W)) \bigr)
  \\
  &\cong
  \sBorn\bigl( U, \intsBorn(V, W) \bigr)
  \,,
\end{split}
\end{equation*}
where we have used that $I$ is full and faithful, the definition of $\Sotimes$, the adjunction~\eqref{eq:TensorInnHomAdj}, the definition of $\intsBorn(V,W)$, and again that $I$ is full and faithful. 
\end{proof}
We record the following useful results about convergence in the internal Hom.
\begin{Proposition}
\label{prop:convergenceComposition}
Let $U$, $V$, and $W$ be bornological vector spaces. If a sequence $f_n \in \intHom(U,V)$ Mackey converges to $f$ and a sequence $g_n \in \intHom(V,W)$ Mackey converges to $g$, then the sequence $g_n \circ f_n \in \intHom(U,W)$ Mackey converges to $g\circ f$.
\end{Proposition}
\begin{proof}
The proof follows from applying Proposition~\ref{prop:BoundedContinuous} to the bounded linear map of composition $\circ: \intHom(V,W) \Botimes \intHom(U,V) \to \intHom(U,W)$.
\end{proof}

Let $V := \bigoplus_{i=1}^\infty V_i$ be the coproduct of a countable family of bornological vector spaces. Let $\sigma_i: V_i \to V$ denote the inclusion. Since the inner hom sends colimits in the first argument to limits, we have
\begin{equation}
\label{eq:EndSumVi}
  \intHom(V,V) \cong \prod_{i = 1}^\infty \intHom(V_i, V)
  \,.
\end{equation}
On the right side of~\eqref{eq:EndSumVi}, the identity is given by the family
\begin{equation*}
  \id_V = (\sigma_1, \sigma_2, \ldots) 
  \in \prod_{i = 1}^\infty \intHom(V_i, V)
  \,.
\end{equation*}
Let us denote by $I_n \in \intHom(V,V)$ the partial identity map given by
\begin{equation*}
  I_n = (\sigma_1, \ldots, \sigma_n, 0, 0, \ldots)
  \in \prod_{i = 1}^\infty \intHom(V_i, V)
  \,.
\end{equation*}
If $V_i = \bbR$, then $I_n$ is the $n\times n$ identity matrix sitting inside the space of endomorphism of $\bbR^\infty = \bigoplus_{i=1}^\infty \bbR$. 

    
\begin{Lemma}
\label{lem:convergenceFiniteDiagonals}
The partial identities $I_n \in \intHom(V,V)$ Mackey converge to $\id_V$.
\end{Lemma}
\begin{proof}
The proof follows from Lemma~\ref{lem:InfSeqConverge}.   
\end{proof}

\subsection{The complete tensor product and mapping space}
\label{sec:CompTensorProd}

\begin{Definition}
\label{def:TensorcBorn}
Let $J: \cBorn \hookrightarrow \sBorn$ be the inclusion of the full subcategory and $\Comp: \sBorn \to \cBorn$ its left adjoint, the completion functor. The tensor product of complete bornological vector spaces is defined by the commutative diagram
\begin{equation*}
\begin{tikzcd}
\cBorn \times \cBorn
\ar[r, "\Cotimes"]
\ar[d, "J \times J"']
&
\cBorn
\\
\sBorn \times \sBorn
\ar[r, "\Sotimes"']
&
\sBorn
\ar[u, "\Comp"']
\end{tikzcd}
\end{equation*}
of categories.
\end{Definition}

We have to show that $\Cotimes$ is associative. Since the inclusion $J: \cBorn \to \sBorn$ is full and faithful, the counit $\epsilon: \Comp J \to \Id$ of the adjunction~\eqref{eq:CompJAdjoint} is an isomorphism. In other words, $\cBorn \hookrightarrow \sBorn$ is a \textdef{reflective subcategory}. Let $\eta: \Id \to J\Comp$ denote the unit of the adjunction. A sufficient condition for the tensor product to be associative is that 
\begin{equation}
\label{eq:DayConvCond}
  \Comp(JU \Sotimes JV)
  \xrightarrow{~\Comp(\eta_U \Sotimes \id_V)~}
  \Comp( J\Comp U \Sotimes V)
\end{equation}
is an isomorphism. If this is the case, then we have the natural isomorphisms
\begin{equation*}
\begin{split}
  (U \Cotimes V) \Cotimes W
  &=
  \Comp \bigl( J\Comp(JU \Sotimes JV) \Sotimes JW \bigr)
  \\
  &\cong
  \Comp \bigl( (JU \Sotimes JV) \Sotimes JW \bigr)
  \\
  &\cong
  \Comp \bigl( JU \Sotimes (JV \Sotimes JW) \bigr)
  \\
  &\cong
  \Comp \bigl( JU \Sotimes J\Comp(JV \Sotimes JW) \bigr)
  \\
  &\cong
  U \Cotimes (V \Cotimes W)
  \,,
\end{split}
\end{equation*}
where we have used that $\Sotimes$ is associative and symmetric. Day's reflection theorem states four equivalent conditions for~\eqref{eq:DayConvCond} to be an isomorphism.

\begin{Theorem}[Theorem~1.2 in \cite{Day:1972}]
\label{thm:DaysReflection}
Let $\calD$ be a cartesian closed symmetric monoidal category; let $R: \calC \to \calD$ be a full and faithful functor with a left adjoint $L: \calD \to \calC$; let $\eta: \Id \to RL$ denote the unit of the adjunction. The following are equivalent: For all objects $c \in \calC$ and $d, d' \in \calD$
\begin{align}
\eta_{\intHom(d,Rc)}: \intHom(d, Rc) &\longrightarrow 
RL\intHom(d, Rc) \quad\text{is an isomorphism;}
\tag{i}
\\
\intHom(\eta_D,Rc): \intHom(RL d, Rc) &\longrightarrow 
\intHom(d, Rc) \quad\text{is an isomorphism;}
\tag{ii}
\\
L(\eta_d \otimes \id_{d'}): L (d \otimes d') &\longrightarrow
L(RLd \otimes d') \quad\text{is an isomorphism;}
\tag{iii}
\\
L(\eta_d \otimes \eta_{d'}): L (d \otimes d') &\longrightarrow
L(RLd \otimes RLd') \quad\text{is an isomorphism.}
\tag{iv}
\end{align}
\end{Theorem}

We want to apply Theorem~\ref{thm:DaysReflection} to the adjunction~\eqref{eq:CompJAdjoint}. The inclusion $J: \cBorn \to \sBorn$ is full and faithful and $\Comp: \sBorn \to \cBorn$ its left adjoint. The next proposition states that Condition~(i) of Theorem~\ref{thm:DaysReflection} is satisfied.

\begin{Proposition}
\label{prop:InnHomComp}
Let $V$ and $V'$ be separated convex bornological vector spaces. If $V'$ is complete, then so is $\intBorn(V,V')$.
\end{Proposition}
\begin{proof}
We give a proof in Section~\ref{sec:InnHomComp} of the appendix.
\end{proof}

Proposition~\ref{prop:InnHomComp} states that the inner hom functor $\intsBorn: \sBorn^\op \times \sBorn \to \sBorn$ restricts to a functor $\intcBorn: \cBorn^\op \times \cBorn \to \cBorn$ defined by
\begin{equation}
\label{eq:intcBorn}
  J\bigl(\intcBorn(V,W)\bigr) = \intsBorn(JV, JW)
  \,,
\end{equation}
where $J: \cBorn \hookrightarrow \sBorn$ is the inclusion. 

\begin{Corollary}
The complete tensor product $\Cotimes$ of Definition~\ref{def:TensorcBorn} and the inner hom $\intcBorn(\Empty, \Empty)$ of~\eqref{eq:intcBorn} equip the category $\cBorn$ of complete bornological vector spaces with a closed symmetric monoidal structure.   
\end{Corollary}
\begin{proof}
This follows from Day's reflection Theorem~\ref{thm:DaysReflection}~(i).
\end{proof}

From now on, we will write $\intHom$ for any of the internal hom functors of $\Born$, $\sBorn$, and $\cBorn$. We have shown this to be unambiguous.

\section{The bornology of test functions}
\label{sec:TestFuncBorn}

\subsection{The bornological vector space of test functions}

Our main example of a complete bornological vector space is the vector space $C^\infty_\mathrm{c}(X)$ of compactly supported smooth functions on a manifold $X$, commonly called \textdef{test functions}, which we will review now. 
We start by reviewing its well-known structure as a locally convex LF space and discuss the von Neumann bornology.

\begin{Remark}
The functor $\vN:\lcTVS \to \Born$ from Remark~\ref{rmk:vonNeumann} restricts to an equivalence on the subcategory of LF spaces. For a detailed comparison, we refer to \cite[Chapter 1.1.6]{Meyer2007}.
\end{Remark}

Let $K \subset \bbR^n$ be a compact subset. Let the space 
\begin{equation*}
  C^\infty_K(\bbR^n)
  := \{ f \in C^\infty(\bbR^n) ~|~ \Supp f \subset K \}
\end{equation*}
denote the vector space of smooth functions with support in $K$. We have the family of seminorms
\begin{equation}
\label{eq:TestFuncSemiNorms}
  p_{K,k}(f)
  := 
   \sum_{|\alpha| \leq k} \frac{1}{\alpha!} \sup_{x \in K} \left| \partial^\alpha f \right|
\end{equation}
indexed by a multi-index $\alpha \in \bbN^n$.
These seminorms are submultiplicative, that is,
\begin{equation*}
  p_{K,k}(fg) 
  \leq p_{K,k}(f)\,p_{K,k}(g)
\end{equation*}
for the pointwise product $fg$ of any two functions $f$ and $g$ with support in $K$.

As locally convex topological vector space, the seminorms equip $C^\infty_K(\R^n)$ with a Fr\'echet topology. Using the von Neumann bornology, we can regard it as a complete bornological vector space. When $K \subset L$ for another compact subset $L$, we have a bounded injection
\begin{equation*}
  C^\infty_K(\bbR^n) \hookrightarrow  
  C^\infty_{L}(\bbR^n)
\end{equation*}
of bornological vector spaces. The bornological space of all compactly supported functions is obtained as the colimit
\begin{equation*}
  C_\mathrm{c}^\infty(\bbR^n)
  = \Colim_{
    \substack{
      K\subset \bbR^n\\
      K\text{ compact}
    }} 
    C^\infty_K(\bbR^n)
\end{equation*}
in $\Born$, where the indexing category has compact sets as objects and inclusions as morphisms.

It follows from Proposition~\ref{prop:ProdCoprodBorn}~(ii) and (iii), that a bounded set in the colimit bornology is a subset of a sum $A_1 + \ldots + A_k$, where each $A_i$ is a bounded subset of $C_{K_i}^\infty(\bbR^n)$ for $K_i$ compact. In particular, all functions in $A$ have support in the compact subset $K = \bigcup_{1 \leq i \leq k} K_i$. The upshot is that $A\subset C_\mathrm{c}^\infty(\R^n)$ is bounded if and only if
\begin{itemize}

\item[(i)] all $a \in A$ have support in some fixed compact $K \subset \bbR^n$, and

\item[(ii)] all seminorms~\eqref{eq:TestFuncSemiNorms} are bounded on $A$.

\end{itemize}
This implies that a sequence $(f_i)_{i \in \bbN}$ converges in $C_\mathrm{c}^\infty(\bbR^n)$ if and only if all $f_i$ have support in some fixed compact $K$ and $f_i$ converges in $C_K^\infty(\bbR^n)$. Since $C_K^\infty(\bbR^n)$ is complete, it follows that $C_\mathrm{c}^\infty(\bbR^n)$ is complete.
In particular, the colimit in $\Born$ coincides with the colimit in $\cBorn$.

This generalizes to a manifold $X$ by taking the colimit over the compact subsets subordinate to the charts of an atlas. Since any coordinate transformation and its inverse are bounded, the bornology does not depend on the atlas. The upshot is that the bornological space of compactly supported smooth functions on $X$ is given by the colimit
\begin{equation}
\label{eq:CompFuncMColim}
  C_\mathrm{c}^\infty(X)
  = \Colim_{
    \substack{
      K\subset X\\
      K\text{ compact}
    }} 
    C^\infty_K(X) \,,
\end{equation}
taken in any of the categories $\cBorn$, $\sBorn$, or $\Born$.
The following property holds only for the complete bornological tensor product.

\begin{Proposition}[Example~1.95 in \cite{Meyer2007}]
\label{prop:TestTensProd}
Let $X$ and $Y$ be smooth manifolds. Then there is a natural isomorphism of complete bornological vector spaces
\begin{equation*}
  C_\mathrm{c}^\infty(X) \Cotimes C_\mathrm{c}^\infty(Y)
  \cong
  C_\mathrm{c}^\infty(X\times Y)      
\end{equation*}
induced by $f \otimes g \mapsto f(x)\,g(y)$.    
\end{Proposition}

\begin{Remark}
\label{rmk:lcTVSpathology}
In the category $\lcTVS$ of locally convex topological vector spaces, the analogue of Proposition~\ref{prop:TestTensProd} holds only for the complete inductive tensor product. However, this does \emph{not} equip $\lcTVS$ with a symmetric monoidal structure.    
\end{Remark}

\subsection{Pullback and fiber integration}
\label{sec:PullbackFibInt}

\begin{Proposition}
\label{prop:TestFuncPullback}
Let $\phi: X \to Y$ be a smooth map of manifolds. If $\phi$ is proper, then the pullback $\phi^*: C_\mathrm{c}^\infty(Y) \to C_\mathrm{c}^\infty(X)$, $f \mapsto f \circ \phi$ is bounded.  
\end{Proposition}
\begin{proof}
Let $K \subset Y$ be compact. By the assumption that $\phi$ is proper, $\phi^{-1}(K) \subset X$ is compact. The pullback restricts to a map $C_K^\infty(Y) \to C_{\phi^{-1}(K)}^\infty(X)$. Since $K$ is compact, all derivatives of $\phi$ are bounded on $K$. It follows from the chain rule that the derivatives of $f \circ \phi$ are bounded for all $f \in C_K^\infty(Y)$. By the colimit formula~\eqref{eq:CompFuncMColim}, a map on $C_\mathrm{c}^\infty(Y)$ is bounded if its restrictions to $C_K^\infty(Y)$ are bounded for all compact $K$. This shows that $\phi^*$ is bounded.
\end{proof}

Proposition~\ref{prop:TestTensProd} and Proposition~\ref{prop:TestFuncPullback} can be summarized by saying that $C_\mathrm{c}^\infty$ is a monoidal contravariant functor from the category of smooth manifolds and proper maps to the category of complete bornological spaces.    

\begin{Example}
\label{ex:TestEval}
Let $x: * \to X$ be a point in a smooth manifold. The pullback $x^*: C_\mathrm{c}^\infty(X) \to \bbR$, $f \mapsto f(x)$ is the evaluation at $x$. Let $f \in C_\mathrm{c}^\infty(X)$ be such that $f(x) = 1$. Then the map $\bbR \to C_\mathrm{c}^\infty(X)$, $r \mapsto rf$ is a section of $x^*$, which shows that $x^*$ is a split epimorphism of bornological vector spaces.
\end{Example}

\begin{Proposition}
\label{prop:TestSubmfgPullback}
Let $\phi: X \to Y$ be a closed embedded submanifold. Then the pullback $\phi^*: C_\mathrm{c}^\infty(Y) \to C_\mathrm{c}^\infty(X)$ is a split epimorphism of bornological vector spaces.
\end{Proposition}
\begin{proof}
Since $\phi$ is a closed embedding, it is proper. Since $\phi$ is an embedding, there is a tubular neighborhood $U \subset Y$ of $\phi(X)$ with projection $\pi:U \to X$, $\pi\circ \phi = \id_X$. We can choose $U$ such that $\pi^{-1}(K)$ is precompact in $Y$ for all compact $K\subset X$. We can find a function $\chi \in C^\infty(Y)$ with support in $U$ such that $\phi^*\chi = 1$.

For every $f \in C_K^\infty(X)$, the pullback $\pi^* f: U \to \bbR$ has support in $\pi^{-1}(K)$, which is precompact in $Y$. For every smooth $g: U \to \bbR$ with support in $\pi^{-1}(K)$, multiplication by $\chi$ yields a smooth function on $U$ with support contained in the compact set $\pi^{-1}(K) \cap \Supp\chi$. Composing the two operations, we obtain a sequence of bounded linear maps
\begin{equation*}
\begin{tikzcd}
C_K^\infty(X) \ar[r, "\pi^*"]
& 
C^\infty_{\pi^{-1}(K)}(U) \ar[r,"\cdot \chi"]
&         
C_\mathrm{c}^\infty(U) \ar[r, hook]
&
C_\mathrm{c}^\infty(Y)
\,.
\end{tikzcd}
\end{equation*}
Taking the colimit in $\Born$ over all compact $K \subset X$, we obtain the bounded linear map $\sigma: C_\mathrm{c}^\infty(X) \to C_\mathrm{c}^\infty(Y)$, $f \mapsto (\pi^* f)\chi$. Since $\phi^* \bigl((\pi^* f) \chi\bigr) = (\phi^* \pi^* f)(\phi^* \chi) = f$, $\sigma$ is a section of $\phi^*$. We conclude that $\phi^*$ is a split epimorphism in $\Born$.
\end{proof}

\begin{Example}
\label{ex:PointwiseMult}
Let $\Delta_X: X \to X \times X$, $x \mapsto (x,x)$  be the diagonal, which is a closed embedded submanifold. By Proposition~\ref{prop:TestFuncPullback}, the pullback
\begin{equation*}
  \Delta_X^*:
  C_\mathrm{c}^\infty(X) \Cotimes C_\mathrm{c}^\infty(X)
  \longrightarrow
  C_\mathrm{c}^\infty(X)      
\end{equation*}
is bounded. It is the bornological completion of the pointwise multiplication of functions. It follows from Proposition~\ref{prop:TestSubmfgPullback} that it is a split epimorphism.
\end{Example}

\begin{Proposition}
\label{prop:TestIntegr}
The integration over a manifold $X$ with a Radon measure $\lambda$,
\begin{equation*}
\begin{aligned}
  C_\mathrm{c}^\infty(X)      
  &\longrightarrow \bbR
  \\
  f &\longmapsto \int\nolimits_X f\, \di \lambda
\end{aligned}
\end{equation*}
is bounded.
\end{Proposition}
\begin{proof}
For every $f$ with support in a compact subset $K \subset X$, we have $| \int_X f\, \di\lambda| \leq \sup_K |f| \, \lambda(K)$, where the measure $\lambda(K)$ is finite since $K$ is compact.
\end{proof}

\begin{Proposition}
\label{prop:TestFiberIntegr}
Let $\phi: X \to Y$ be a surjective submersion. Let $\{\lambda_y \}_{y \in Y}$ be a family of strictly positive Radon measures on the fibers $\phi^{-1}(y)$ such that the integration over the fibers sends smooth compactly supported maps to smooth (necessarily compactly supported) maps. Then the fiber integration
\begin{equation*}
\begin{aligned}
  \phi_*: C_\mathrm{c}^\infty(X)      
  &\longrightarrow C_\mathrm{c}^\infty(Y)
  \\
  (\phi_* f)(y) 
  &:= \int\nolimits_{\phi^{-1}(y)} f\, \di \lambda_y
\end{aligned}
\end{equation*}
is a regular epimorphism of bornological vector spaces.
\end{Proposition}
\begin{proof}
Locally, every submersion is isomorphic to a projection $\pr: \bbR^p \times \bbR^q \to \bbR^p$, where $p+q = \dim X$ and $p = \dim Y$. That is, there are neighborhoods $U \subset \bbR^p$ and $V \subset \bbR^q$ such that $U \times V$ is a chart of $X$, $U$ a chart of $Y$, and $\pr: U \times V \to U$ the projection to the first factor. Moreover, since the measures are strictly positive, the coordinates can be chosen such that the measure $\lambda_y$ does not depend on $y \in \bbR^p$, that is, $\lambda_y = \lambda$ is a measure on $V$. If we consider functions with support in $U \times V$, the fiber integration is obtained by tensoring the identity on $C_\mathrm{c}^\infty(U)$ with the integration of Proposition~\ref{prop:TestIntegr},
\begin{equation*}
  \phi_*: 
  C_\mathrm{c}^\infty(U \times V)
  \cong 
  C_\mathrm{c}^\infty(U) \Cotimes
  C_\mathrm{c}^\infty(V)
  \xrightarrow{~\id \Cotimes \int_V(\Empty)\di\lambda~}
  C_\mathrm{c}^\infty(U) \Cotimes \bbR
  \cong
  C_\mathrm{c}^\infty(U)
  \,.
\end{equation*}
It follows from the functoriality of the tensor product that this map is a split epimorphism of bornological vector spaces, so a fortiori regular. We conclude that if the compact subset $K \subset X$ is contained in this chart, then $\phi_*: C_K^\infty(X) \to C_{\phi(K)}^\infty(Y)$ is a regular epimorphism. By taking the colimit over all such $K$, which cover $X$, we obtain the morphism of bornological vector spaces
\begin{equation}
\label{eq:RegEpi1}
   \phi_*:
   C_\mathrm{c}^\infty(X) \longrightarrow 
   C_\mathrm{c}^\infty\bigl(\phi(X)\bigr)
   \,,
\end{equation}
which shows that $\phi_*$ is bounded. Since regular epimorphisms are by definition coequalizers, they are preserved by colimits. It follows that~\eqref{eq:RegEpi1} is a regular epimorphism. Since $\phi$ is assumed to be surjective, $\phi(X) = Y$, which concludes the proof.
\end{proof}

\section{Morita 2-category of bornological algebras}
\label{sec:MoritaCategory}

\subsection{Internal algebras and modules}

For the convenience of the reader, we recall some basic notions of algebra internal to a symmetric monoidal category $\calC$ with tensor product $\otimes$, associator $\Asc$, left and right unit isomorphisms $\Unitl$ and $\Unitr$. An \textdef{algebra} in $\calC$ is an object $A$ together with a morphism $\mu: A \otimes A \to A$ such that the diagram of associativity
\begin{equation}
\begin{tikzcd}[column sep=3.5em]
(A\otimes A) \otimes A
\ar[r,"\Asc_{A,A,A}"] 
\ar[d,"\mu \otimes \id_A"']
& 
A\otimes (A\otimes A) 
\ar[r,"\id_A \otimes \mu"] 
&
A\otimes A 
\ar[d,"\mu"]
\\
A\otimes A 
\ar[rr,"\mu"'] 
&& 
A
\end{tikzcd}
\end{equation}
is commutative. A \textdef{unit} of the algebra is a morphism $\eta: \mathbbm{1} \to A$ such that the diagram
\begin{equation}
\begin{tikzcd}[column sep=3em]
\mathbbm{1} \otimes A 
\ar[r, "\eta \otimes \id_A"]
\ar[rd, "\Unitl_A"'] 
& 
A\otimes A 
\ar[d,"\mu"] 
& 
A\otimes \mathbbm{1} 
\ar[ld,"\Unitr_A"] 
\ar[l,"\id_A\otimes \eta"']
\\
&
A
&
\end{tikzcd}
\end{equation}
is commutative. Unless stated explicitly, we will not assume algebras to have units.

A \textdef{left $A$-module} is an object $M$ in $\calC$ with a morphism $\lambda : A \otimes M \to M$, such that the diagram 
\begin{equation}
\label{diag:ModAssoc}
\begin{tikzcd}[column sep=3.5em]
(A \otimes A) \otimes M 
\ar[r,"\Asc_{A,A,M}"]
\ar[d, "\mu \otimes \id_M"']
& 
A\otimes (A\otimes M) 
\ar[r,"\id_A \otimes \lambda"] 
&
A\otimes M \ar[d, "\lambda"]
\\
A\otimes M \ar[rr, "\lambda"']
&& 
M
\end{tikzcd}
\end{equation}
commutes. If $A$ has a unit, then the module is called \textdef{unital} if the diagram
\begin{equation*}
\begin{tikzcd}
\mathbbm{1} \otimes M
\ar[r, "\eta \otimes \id_M"]
\ar[dr, "\Unitl_M"']
&
A \otimes M
\ar[d, "\lambda"]
\\
&
M
\end{tikzcd}
\end{equation*}
commutes. A morphism $f: M \to M'$ in $\calC$ is a \textdef{morphism of $A$-modules} if $f \circ \lambda = \lambda' \circ (\id_A \otimes f)$.

Let $B$ be another algebra in $\calC$. A right $B$-module is an object $M$ with a morphism $\rho: M \otimes B \to M$ such that the analogous diagram of associativity commutes. An \textdef{$A$-$B$-bimodule} is an object $M$ that is a left $A$-module and a right $B$-module such that the diagram
\begin{equation}
\begin{tikzcd}[column sep=3em]
(A \otimes M) \otimes B 
\ar[r,"\Asc_{A,M,B}"]
\ar[d, "\lambda \otimes \id_B"']
& 
A\otimes (M\otimes B) 
\ar[r, "\id_A \otimes \rho"] 
&
A \otimes M \ar[d, "\lambda"]
\\
M \otimes B \ar[rr, "\rho"']
&& 
M
\end{tikzcd}
\end{equation}
commutes. The set of $A$-$B$ bimodules will be denoted by $\Mod(A,B)$. Let $M'$ be another $A$-$B$ bimodule. A morphism $M \to M'$ in $\calC$ is a morphism of $A$-$B$-bimodules if it is a morphism of left $A$-modules and a morphism of right $B$-modules.

\subsection{Tensor product of modules}

Let $M$ be a right $A$-module and $N$ a left $A$-module. A morphism $f:M \otimes N \to P$ is \textdef{$A$-tensorial} if the diagram
\begin{equation}
\label{eq:TensProdModCoeq}
\begin{tikzcd}[column sep=6em]
(M \otimes A) \otimes N
\ar[r, shift left=1.2, "\rho \otimes \id_N"]
\ar[r, shift left=-1.2, "(\id_M \otimes \lambda) \circ \Asc_{M,A,N}"']
&
M \otimes N
\ar[r, "f"]
&
P
\end{tikzcd}
\end{equation}
is commutative, where $\rho$ is the right $A$-action on $M$ and $\lambda$ the left $A$-action on $N$. The coequalizer of the parallel arrows in $\calC$ is the \textdef{tensor product over $A$}, denoted by $M \otimes_A N$.
It has the universal property that any $A$-tensorial morphism $f:M\otimes N\to P$ induces a unique map $M\otimes_A N\to P$.

Let $M \in \Mod(A,B)$ and $N \in \Mod(B,C)$. The tensor product $M \otimes_B N$ is an $A$-$C$-bimodule. Let $P \in \Mod(C,D)$. Then the associator $\Asc_{M,N,P}$ descends to a natural isomorphism of $A$-$D$-bimodules
\begin{equation*}
  (M \otimes_B N) \otimes_C P
  \cong
  M \otimes_B (N \otimes_C P)
  \,.
\end{equation*}
In this sense, the tensor product of bimodules is associative. 

We would like to view the map
\begin{equation*}
\begin{aligned}
  \Mod(A,B) \times \Mod(B,C)
  &\longrightarrow \Mod(A,C)
  \\
  (M,N) &\longmapsto M \otimes_B N
\end{aligned}
\end{equation*}
as the composition morphisms of a category. However, since the algebras are not assumed to be unital, tensoring with $A$ over $A$ is generally not an isomorphism. The commutativity of Diagram~\eqref{diag:ModAssoc} means that the action $\lambda: A \otimes M \to M$ is $A$-tensorial. It follows that it descends to a unique morphism $\tilde{\lambda}$ on the tensor product over $A$, such that the diagram
\begin{equation}
\label{diag:AotimesAMtoM}
\begin{tikzcd}
A \otimes M 
\ar[dr, "\lambda"]
\ar[d]
&
\\
A \otimes_A M
\ar[r, "\tilde{\lambda}"']
&
M
\end{tikzcd}
\end{equation}
commutes. However, $\tilde{\lambda}$ is generally not an isomorphism.

\begin{Definition}[Definition~3.2 in \cite{Meyer2011}]
Let $A$ be an algebra in the monoidal category $\calC$. A left $A$-module $M$ is called \textdef{smooth} if the morphism $A \otimes_A M \to M$ from~\eqref{diag:AotimesAMtoM} is an isomorphism. A right $B$-module $M$ is smooth if the analogous morphism $M \otimes_B B \to M$ is an isomorphism. A bimodule is smooth if it is smooth as a left and as a right module. An algebra $A$ that is smooth as a left (and therefore right) module over itself is called \textdef{self-induced}.
\end{Definition}

Let $A$ be a self-induced algebra and $M$ a left $A$-module. Then
\begin{equation*}
  A \otimes_A (A \otimes_A M)
  \cong
  (A \otimes_A A) \otimes_A M
  \cong
  A \otimes_A M
  \,,
\end{equation*}
which shows that $A \otimes_A M$ is smooth. Moreover, the operator $S := A\otimes_A \Empty$ is idempotent up to isomorphism, $S^2M \cong SM$. We conclude that $M$ is smooth if and only if it is in the essential image of $S$. The operator $S$ is the right adjoint of the inclusion of smooth $A$-modules into all $A$-modules.


\begin{Proposition}
\label{prop:MorCat}
Let $\calC$ be a monoidal category. Then there is a 2-category that has self-induced algebras as objects, smooth bimodules as 1-morphisms with the tensor product as composition, and morphisms of bimodules as 2-morphisms.
\end{Proposition}

\begin{Terminology}
The 2-category of Proposition~\ref{prop:MorCat} will be called the \textdef{Morita 2-category} of algebras in $\calC$ and denoted by $\AlgMrt(\calC)$. A weak isomorphism in this 2-category is called a \textdef{Morita equivalence} of algebras. 
\end{Terminology}

\begin{Remark}
The tensor product of algebras, bimodules, and biequivariant maps equips $\AlgMrt(\calC)$ with a monoidal structure. 
\end{Remark}

In the remainder of the paper, we will invoke MacLane's coherence theorem and omit the associator of the tensor product in formulas and commutative diagrams, which will improve the legibility.

\subsection{Separable algebras and nondegenerate modules}

\begin{Proposition}
\label{prop:UnitalSelfInduced}
If an algebra $A$ in $\calC$ is unital, then $A$ is self-induced and every unital $A$-module is smooth.    
\end{Proposition}
\begin{proof}
Let $M$ be a left $A$-module with action $\lambda: A \otimes M \to M$; let $\pi: A \otimes M \to A \otimes_A M$ be the coequalizer and $\tilde{\lambda}: A \otimes_A M \to M$ the morphism from Diagram~\eqref{diag:AotimesAMtoM}; let $\eta: \mathbbm{1} \to A$ be the unit of $A$. Assume that $M$ is unital, that is, $\lambda \circ (\eta \otimes \id_M) = \Unitl_M$. This means that $\sigma := (\eta \otimes \id_M) \circ \Unitl_M^{-1}$ is a section of the action $\lambda$,
\begin{equation*}
  \lambda \circ \sigma = \id_M
  \,.
\end{equation*}
It follows that
\begin{equation*}
\begin{split}
  \tilde{\lambda} \circ (\pi \circ \sigma)
  &= 
  (\tilde{\lambda} \circ \pi) \circ \sigma
  =
  \lambda \circ \sigma
  \\
  &=
  \id_M
  \,,
\end{split}
\end{equation*}
which shows that $\pi \circ \sigma$ is the right inverse of $\tilde{\lambda}$. It follows from the universal property of the coequalizer $\pi$ that $(\pi \circ \sigma) \circ \tilde{\lambda} =  \id_{A \otimes_A M}$, so that $\pi \circ \sigma$ is also the left inverse of $\tilde{\lambda}$. We conclude that $\tilde{\lambda}$ is an isomorphism, that is, $M$ is smooth. For $M = A$, viewed as a left $A$-module, this shows that $A$ is self-induced.
\end{proof}

Non-unital algebras are generally not self-induced. For example, in a category with a zero object, we can equip any object $A$ with the zero multiplication. Then $A \otimes_A A = 0$. If $A$ is non-unital, the action $\lambda: A \otimes M \to M$ does generally not have a section. For example, if $\lambda$ is the zero action, then $A \otimes_A M = 0$. There are weaker conditions that ensure that algebras are self-induced and modules are smooth.

\begin{Definition}
\label{def:Separable}
An algebra $A$ is called \textdef{left (right) separable} if the multiplication $A \otimes A \to A$ has a left (right) $A$-linear section. $A$ is called \textdef{separable} if there is a section that is both left and right $A$-linear.
\end{Definition}

\begin{Example}
Every unital algebra is left and right separable, with sections $a \mapsto a \otimes 1$ and $a \mapsto 1 \otimes a$, but generally not separable.    
\end{Example}

\begin{Example}
The commutative algebra of test functions $C_\mathrm{c}^\infty(X)$ on a non-compact manifold $X$ is non-unital because it does not contain the constant function $1$. When viewed as an algebra in $\cBorn$, we have seen in Example~\ref{ex:PointwiseMult} that the pointwise multiplication is the pullback along the diagonal in $X \to X \times X$, which is a  closed embedded submanifold. It follows from Proposition~\ref{prop:TestSubmfgPullback} and its proof that this pullback has a section that is left and right $C_\mathrm{c}^\infty(X)$-bilinear. We conclude that $C_\mathrm{c}^\infty(X)$ is separable. 
\end{Example}

\begin{Definition}
\label{def:NondegenerateMod}
Let $A$ be a (not necessarily unital) algebra in $\calC$. A left $A$-module with action $\lambda: A \otimes M \to M$ will be called
\begin{itemize}

\item[(i)] \textdef{nondegenerate} if $\lambda$ is an epimorphism;

\item[(ii)] \textdef{strongly nondegenerate} if $\lambda$ is a strong epimorphism;

\item[(iii)] \textdef{split nondegenerate} if $\lambda$ is a split epimorphism;

\end{itemize}
The notions for right modules are analogous. See Appendix~\ref{sec:monosandEpis} for the different notions of epimorphisms.
\end{Definition}

\begin{Proposition}
\label{prop:QuasiUnitalAction}
Every split nondegenerate module is smooth. Every algebra that is split nondegenerate as left or right module over itself is self-induced. In particular, every left or right separable algebra is self-induced.
\end{Proposition}
\begin{proof}
The proof is identical to Proposition~\ref{prop:UnitalSelfInduced} upon replacing $\eta \otimes \id$ by the section of the multiplication.
\end{proof}

\subsection{Projective modules}

The general definition for an object $P$ in some category to be projective is to require that every morphism $P \to N$ factors through every epimorphism $M \twoheadrightarrow N$,
\begin{equation*}
\begin{tikzcd}
&
M \ar[d, two heads]
\\
P \ar[r] \ar[ur, dashed, "\exists"]
& N
\end{tikzcd}   
\end{equation*}
For the category of $A$-modules internal to some monoidal category, this notion is too strict and not very useful. Instead, we will use the following notion:

\begin{Definition}
\label{def:ProjectiveMod}
Let $A$ be an algebra in a monoidal category $\calC$. A morphism $M \to N$ of $A$-modules is called a \textdef{strict epimorphism} if the underlying morphism in $\calC$ is a split epimorphism. An $A$-module $P$ will be called \textdef{projective} if every morphism of $A$-modules $P \to N$ factors through every strict epimorphism $M \twoheadrightarrow N$ of $A$-modules.
\end{Definition}

Let $A_+ = A \oplus \bbC 1$ denote the unital algebra obtained by the free adjunction of a unit. The projection $\epsilon = \pr_2: A_+ \to \bbC$ is an augmentation of $A_+$ with ideal $\ker \epsilon = A$. Since $A$ is a subalgebra of $A_+$, $A_+$ is a left and right $A$-module. A left $A$-module is free if it is of the form $A_+ \otimes V \cong (A \otimes V) \oplus V$ for some vector space $V$. When $\calC$ is concrete, the $A$-action is given by $b \cdot \bigl( (a \otimes v) \oplus v'\bigr) = ba \otimes v + b \otimes v'$. 

\begin{Proposition}[Lemma~5.54 in \cite{Aretz23}]
\label{prop:ProjModCriteria}
Let $P$ be an $A$-module in $\calC$. The following are equivalent:
\begin{itemize}

\item[(i)] $P$ is projective.

\item[(ii)] Every strict epimorphism  of $A$-modules $M \to P$ splits.

\item[(iii)] $P$ is the direct summand of a free module.

\end{itemize}
\end{Proposition}

\begin{Proposition}
\label{prop:ModuleSectionProj}
Let $A$ be a left separable algebra and $P$ a left $A$-module. If the action $A \otimes P \to P$ has an $A$-linear section, then $P$ is projective. In particular, $A$ is projective as left $A$-module.
\end{Proposition}
\begin{proof}
Since $A$ is an ideal of $A_+$, the multiplication of $A_+$ restricts to a left $A$-linear map $\tilde{\mu}: A_+ \otimes A \to A$. $A$ is left separable, which means that there is a left $A$-linear section $\sigma: A\to A\otimes A$ of the multiplication $\mu: A \otimes A \to A$. Composing $\sigma$ with the inclusion of the left $A$-submodule $A \otimes A \hookrightarrow A_+ \otimes A$ is a section of $\tilde{\mu}$, which shows that $A$ is a direct summand of the free left $A$-module $A_+ \otimes A$.

A left $A$-linear splitting of $A \otimes P \to A$ makes $P$ a direct summand of $A \otimes P$ which is itself a direct summand of the free $A$-module $A_+ \otimes A \otimes P$. It follows from Proposition~\ref{prop:ProjModCriteria}~(iii) that $P$ is projective.
\end{proof}

\subsection{Approximate units}

In bornological algebras we have the notion of approximate units, which is analogous to $C^*$-algebras.

\begin{Definition}
\label{def:ApproxUnit}
Let $A$ be an algebra in $\Born$; for every $b \in A$, let $r_b: A \to A$, $a \mapsto ab$ denote the right multiplication by $b$. A \textdef{countable strong approximate right unit} of $A$ is a sequence $(e_n)_{n \in \bbN}$ of elements in $A$ such that the sequence $r_{e_n}\in \intHom(A,A)$ Mackey converges to the identity, $\lim_{n\to \infty} r_{e_n} = \id_A$.
\end{Definition}

All approximate units of this paper will be countable, so we will omit the adjective ``countable''. There is an obvious notion of strong approximate left unit. If a strong approximate right unit is also a strong approximate left unit, then it is called a strong approximate unit. 

\begin{Proposition}
\label{prop:StrongPointwiseApprUnit}
Let $(e_n)_{n \in \bbN}$ be a strong approximate right unit of a bornological algebra $A$. Then
\begin{equation}
\label{eq:PointwiseApproxUnit}
  \lim_{n\to \infty} a\,e_n = a
\end{equation}
for all $a\in A$.
\end{Proposition}
\begin{proof}
Let $A$ be an algebra in $\Born$; let $(e_n)_{n \in \bbN}$ be a strong approximate unit of $A$. It follows from Proposition~\ref{prop:TensorConvergence} for every $a \in A$, that $r_{e_n} \Botimes a \in \intHom(A,A) \Botimes A$ Mackey converges to $\id_A \Botimes a$. Since $\Born$ is closed monoidal, we have an evaluation map $\Eval:\intHom(A,A)\Botimes A\to A$, which maps $r_{e_n}\Botimes a \mapsto a\,e_n$. Since $\Eval$ is bounded linear, it preserves Mackey convergence by Proposition~\ref{prop:BoundedContinuous}. This shows that
\begin{equation*}
  \lim_{n\to \infty} a\, e_n
  = \lim_{n\to \infty }\Eval(r_{e_n} \Botimes a) 
  = \Eval\bigl(\lim_{n\to \infty}( r_{e_n} \Botimes a)\bigr) 
  = \Eval(\id\Botimes a) = a
  \qedhere
\end{equation*}
\end{proof}

\begin{Terminology}
\label{term:QuasiUnital}
A sequence $(e_n)_{n \in \bbN}$ with the pointwise convergence~\eqref{eq:PointwiseApproxUnit} is called an \textdef{approximate unit} without the adjective ``strong''. A bornological algebra with such an approximate unit that is left and right separable (Definition~\ref{def:Separable}) is called \textdef{quasi-unital} in \cite[Definition~14]{Meyer2004}.
\end{Terminology}

\begin{Example}
Let $M$ be a smooth manifold. The commutative bornological algebra $C_\mathrm{c}^\infty(M)$ of test functions has a strong approximate unit given by any sequence of compactly supported functions $\xi_n$ with $\xi_n|_{K_n} = 1$ on an exhaustion $K_0 \subset K_1 \subset \ldots \subset M$ by compact sets. This is a consequence of Lemma~\ref{lem:fiberdirac} (see Remark~\ref{rmk:approximateunitcompactsupport}).
\end{Example}

\begin{Proposition}
\label{prop:MAAtoAinjective}
Let $A$ be an algebra and $M$ a right $A$-module in $\cBorn$. If $A$ has a strong approximate right unit, then the natural map $M \Cotimes_A A \to M$ is a monomorphism.
\end{Proposition}

\begin{proof}
We have the commutative diagram
\begin{equation}
\label{diag:SmoothModule}
\begin{tikzcd}
M \Cotimes A
\ar[dr, "\rho"]
\ar[d, "\pi"']
&
\\
M \Cotimes_A A
\ar[r, "\tilde{\rho}"']
&
M
\end{tikzcd}
\end{equation}
where $\rho$ is the right $A$-action on $M$, $\pi$ is the coequalizer~\eqref{eq:TensProdModCoeq} that defines the tensor product over $A$, and $\tilde{\rho}$ the map given by the universal property of the coequalizer. Assume that $\tilde{\zeta}$ is an element of $M \Cotimes_A A$ lying in the kernel $\tilde{\rho}(\tilde{\zeta}) = 0$. Since $\pi$ is a regular epimorphism and regular epimorphisms in $\cBorn$ are surjective (cf. Lemma~\ref{lem:separatedcompletequotients}), there is a preimage $\zeta \in M \Cotimes A$ such that $\pi(\zeta) = \tilde{\zeta}$. It follows that $\rho(\zeta) = 0$.

Let $e_n\in A$ be a strong approximate right unit. The tensor product of morphisms that sends $f: M \to M$ and $g:A \to A$ to $f \Cotimes g: M \Cotimes A \to M \Cotimes A$ is a bounded bilinear map
\begin{equation*}
  \intHom(M,M) \times \intHom(A,A)
  \xrightarrow{~\Cotimes~} 
  \intHom(M\Cotimes A,M\Cotimes A)
  \,.
\end{equation*}
Since $(\id_M, r_{e_n})$ Mackey converges component-wise to $(\id_M, \id_A)$, it follows that the tensor product $\id_M \Cotimes r_{e_n}$ Mackey converges to $\id_M \Cotimes \id_A = \id_{M \Cotimes A}$. Using the coequalizer property~\eqref{eq:TensProdModCoeq} of the tensor product over $A$, we obtain 
\begin{equation*}
\begin{split}
  \pi\bigl( (\id_M\Cotimes r_{e_n})(\zeta) \bigr)
  &= \bigl( \pi \circ (\id_M \Cotimes \mu)\bigr) 
    (\zeta\Cotimes e_n)
  \\
  &= \bigl( \pi \circ (\rho \Cotimes \id_A) \bigr)
    (\zeta\Cotimes e_n)
  \\
  &= \pi(\rho(\zeta)\Cotimes e_n)
  \\
  &= 0 \,.    
\end{split}
\end{equation*}
In the limit $n\to \infty$ we obtain $\pi(\zeta) = \tilde{\zeta} = 0$, which shows that $\tilde{\rho}$ is injective. Since the forgetful functor $\cBorn \to \Vect$ is faithful (Proposition~\ref{prop:ForgetfulFunctors}), an injective map in $\cBorn$ is a monomorphism.
\end{proof}

\begin{Remark}
Meyer proves the same proposition in \cite[Lemma~4.4]{smoothGroupRep} using a weaker notion of approximate unit. The proof needs a better understanding of generic elements in the tensor product; they can be expressed as infinite sums of elementary tensors.
\end{Remark}

\begin{Proposition}
\label{prop:ApproxUnitSmoothMod}
Let $\calA$ be an algebra in $\cBorn$ with a strong approximate right unit; let $M$ be a right $A$-module. The following are equivalent:
\begin{itemize}

\item[(i)] $M$ is smooth.

\item[(ii)] The action $M \Cotimes A \to M$ is a strong epimorphism (Definition~\ref{def:Epis}).

\end{itemize}
\end{Proposition}
\begin{proof}
Assume $M$ is smooth, that is, $\tilde{\rho}$ of Diagram~\eqref{diag:SmoothModule} is an isomorphism. Since $\pi$ is a regular epimorphism, $\rho = \tilde{\rho} \circ \pi$ is a regular, so a fortiori a strong epimorphism.

Conversely, assume that $\rho = \tilde{\rho} \circ \pi$ is a strong epimorphism, which implies that $\tilde{\rho}$ is a strong epimorphism. By Proposition~\ref{prop:MAAtoAinjective}, $\tilde{\rho}$ is a monomorphism. A strong epimorphism that is also a monomorphism is an isomorphism.
\end{proof}

In the terminology of Definition~\ref{def:NondegenerateMod}, Proposition~\ref{prop:ApproxUnitSmoothMod} can be stated as follows: If an algebra has a strong approximate right unit, then its modules are smooth if and only if they are strongly nondegenerate.

\section{Bornological convolution of Lie groupoids}
\label{sec:ConvolutionIngredients}

\subsection{The geometric Morita 2-category of Lie groupoids}
\label{sec:GrpBibu}

\begin{Definition}
\label{def:Gbun}
Let $G$ be a Lie groupoid. A \textdef{left $G$-bundle} is a smooth map $l: P \to G_0$ together with a smooth map
\begin{equation*}
\begin{aligned}
  G_1 \times^{s,l}_{G_0} P 
  &\longrightarrow P
  \\
  (g, p) 
  &\longmapsto g \cdot p
  \,,
\end{aligned}
\end{equation*}
called the \textdef{left $G$-action}, such that $1_{l(p)} \cdot p = p$ and $g \cdot (g' \cdot p) = gg' \cdot p$, whenever defined.
\end{Definition}

\begin{Lemma}
\label{lem:ActionSubmersion}
The action map of a left $G$-bundle is a surjective submersion.  \end{Lemma}
\begin{proof}
Let the action be denoted by $\alpha: G_1 \times_{G_0}^{s,l} P \to P$. Let $(g, p_0')$ be in the $\alpha$-fiber over $p_0 \in P$, that is, $l(p_0') = s(g)$ and $g \cdot p_0' = p_0$. Then $p_0' = g^{-1} \cdot p_0$. As is the case of the target of any Lie groupoid, $t: G_1 \to G_0$ is a submersion. Therefore, there is a local section of $t$ through $g$, that is, we have an open neighborhood $V$ of $t(g) = l(p_0)$ with a smooth map $\sigma: V \to G_1$ satisfying $\sigma(l(p_0)) = g$ and $t(\sigma(v)) = v$ for all $v \in V$. Let $U := l^{-1}(V)$. Since $l$ is continuous, $U$ is an open set. Since $l(p_0) \in V$ it follows that $p_0 \in U$. Consider the map
\begin{equation*}
\begin{aligned}
  \tau:
  U &\longrightarrow 
  G_1 \times_{G_0}^{s,l} P
  \\
  p &\longmapsto 
  \bigl(\sigma(lp), \sigma(lp)^{-1} \cdot p \bigr)
  \,.
\end{aligned} 
\end{equation*}
where we have used shorthand $lp \equiv l(p)$.

We have $\alpha\bigl( \tau(p) \bigr) = \sigma(lp) \cdot \bigl(\sigma(lp)^{-1} \cdot p\bigr) = p$, which shows that $\tau$ is a local section of $\alpha$. Moreover, $\tau(p_0) = \bigl( \sigma(lp_0), \sigma(lp_0)^{-1} \cdot p_0 \bigr) = (g, g^{-1} \cdot p_0) = (g, p_0')$. We conclude that $\alpha$ has a local section through any given point $(g,p_0')$ in the fiber over $p_0$, which shows that $\alpha$ is a submersion. Since $1_{lp} \cdot p = p$, $\alpha$ is surjective.
\end{proof}

\begin{Definition}
Let $l_P: P \to G_0$ and $l_Q: Q \to G_0$ be left $G$-bundles. A smooth map $\phi: P \to Q$ is called \textdef{$G$-equivariant} if $l_Q(\phi(p)) = l_P(p)$ and $\phi(g \cdot p) = g \cdot \phi(p)$, whenever defined.
\end{Definition}

There is an analogous definition of right groupoid bundles and right equivariant maps.

\begin{Definition}
Let $G$ and $H$ be Lie groupoids. A \textdef{$G$-$H$-bibundle} is a span $G_0 \xleftarrow{l} P \xrightarrow{r} H_0$ with a left $G$-action and a right $H$-action that commute, that is, $(g \cdot p) \cdot h = g \cdot (p \cdot h)$, whenever $s(g) = l(p)$ and $r(p) = t(h)$.
\end{Definition}

We may depict a bibundle by the following diagram:
\begin{equation*}
\begin{tikzcd}
G_1
\ar[d, "t_G"', shift left=-1.2, near start]
\ar[d, "s_G", shift left=1.2, near start]
\ar[r, phantom, "\circlearrowright"]
& 
P 
\ar[ld, "l"]
\ar[rd, "r"']
& 
H_1 
\ar[d, "t_H"', shift left=-1.2, near start]
\ar[d, "s_H", shift left=1.2, near start] 
\ar[l, phantom, "\circlearrowleft"]
\\
G_0 &[-1em] &[-1em] H_0
\end{tikzcd}
\end{equation*}

\begin{Definition}
Let $G_0 \xleftarrow{l_P} P \xrightarrow{r_P} H_0$ and $G_0 \xleftarrow{l_Q} Q \xrightarrow{r_Q} H_0$ be left $G$-bundles. A smooth map $\phi: P \to Q$ is called \textdef{biequivariant} if it is $G$-equivariant and $H$-equivariant.
\end{Definition}

Let $r_P: P \to H_0$ be a right $H$-module and $l_Q: Q \to H_0$ a left $H$-module. A map $\phi: P \times_{H_0} Q \to M$ will be called \textdef{$H$-compositional} if 
\begin{equation*}
  \phi(p \cdot h, q) = \phi(p, h \cdot q)
  \,,
\end{equation*}
whenever defined. That is, the diagram
\begin{equation}
\label{diag:BibundCompCoeq}
\begin{tikzcd}[column sep=3em]
P \times_{H_0} H \times_{H_0} Q
\ar[r, shift left=1.2, "{(\beta, \id_Q)}"]
\ar[r, shift left=-1.2, "{(\id_P, \alpha)}"']
&
P \times_{H_0} Q
\ar[r, "\phi"]
&
M
\end{tikzcd}
\end{equation}
is commutative, where $\beta$ is the right $H$-action on $P$ and $\alpha$ the left $H$-action on $Q$. The coequalizer of the parallel arrows of~\eqref{diag:BibundCompCoeq} is the \textdef{composition} of groupoid bundles, denoted by 
\begin{equation*}
  P \circ_H Q = (P \times_{H_0} Q)/H
  \,,
\end{equation*}
where the quotient is by the diagonal $H$-action $(p,q) \cdot h = (p \cdot h, h^{-1} \cdot q)$. We can depict the construction by the following diagram.
\begin{equation}
\begin{tikzcd}
&[-1ex] &[-2em]
P \circ_H Q
\ar[llddd, "l_{P\circ_H Q}"']
\ar[rrddd, "r_{P\circ_H Q}"]
&[-2em] &
\\
&&
P \times_{H_0} Q 
\ar[u, two heads]
\ar[ld]
\ar[rd]
&&
\\
G_1 
\ar[d,"s_G", shift left=1.2, near start] 
\ar[d,"t_G"', shift left=-1.2, near start] 
& 
P 
\ar[ld, "l_P"]
\ar[rd, "r_P"']
& 
H_1 
\ar[d,"s_H", shift left=1.2, near start] 
\ar[d,"t_H"', shift left=-1.2, near start] 
& 
Q \ar[ld, "l_Q"]
\ar[rd, "r_Q"'] 
& 
K_1 
\ar[d,"s_K", shift left=1.2, near start] 
\ar[d,"t_K"', shift left=-1.2, near start] 
\\
G_0 
&& 
H_0 
&& 
K_0
\end{tikzcd}
\end{equation}
Since pullbacks and coequalizers do not generally exist in smooth manifolds, the composition is a priori only a topological space. If $P$ is a $G$-$H$-bibundle and $Q$ a $H$-$K$-bibundle, then $P \circ_H Q$ is a topological $G$-$K$-bibundle. We want the composition to be a smooth manifold and, moreover, that $H \circ_H Q \cong Q$ and $P \circ_H H \cong P$. For this, we have to impose the following conditions.

\begin{Definition}
\label{def:principal}
Let $G$ and $H$ be Lie groupoids. A $G$-$H$ bibundle $G_0 \xleftarrow{l} P \xrightarrow{r} H_0$ is \textdef{right principal} if the following three conditions hold: 
\begin{itemize}
 \item[(P1)] $l$ is a surjective submersion.
 \item[(P2)] The $H$-action is free.
 \item[(P3)] The $H$-action is transitive on the $l$-fibers.
\end{itemize}
\end{Definition}

Assume that $l_P$ is a submersion, not necessarily surjective. Then we have the smooth map
\begin{equation}
\label{eq:ActionCharMap}
\begin{aligned}
  P \times_{H_0} H &\longrightarrow 
  P \times_{H_0}^{l,l} P
  \\
  (p,h) & \longmapsto (p, p\cdot h)
  \,.
\end{aligned}
\end{equation}
The $H$-action is free if and only if~\eqref{eq:ActionCharMap} is injective. It is transitive on the $l$-fibers if and only if~\eqref{eq:ActionCharMap} is surjective. It can be shown that if~\eqref{eq:ActionCharMap} is a bijection, then it is a diffeomorphism \cite{Blohmann2008}. It follows that the action is right principal if and only if~\eqref{eq:ActionCharMap} is a diffeomorphism. A groupoid action is called \textdef{proper} if~\eqref{eq:ActionCharMap} is a proper map. Principal groupoid actions are always proper.

\begin{Proposition}[Proposition~2.11 in \cite{Blohmann2008}]
\label{prop:GrpdBibEquiv2Cat}
Let $G$, $H$, $K$ be Lie groupoids, $P$ a smooth $G$-$H$ bibundle, and $Q$ a smooth $H$-$K$ bibundle. If $P$ and $Q$ are right principal then the composition $P \circ_H Q$ is a smooth right principal $G$-$K$-bibundle.
\end{Proposition}

The triple composition $P\circ_H Q \circ_K R$ of bibundles is given by the colimit of the diagram
\begin{equation*}
\begin{tikzcd}[column sep=4em, row sep=6ex]
P \times_{H_0} H_1 \times_{H_0} Q \times_{K_0} K_1 \times_{K_0} R
\ar[d, "\beta_P \times \id_Q \times \id"', shift right=1.2] 
\ar[d, "\id_P \times \alpha_Q \times \id", shift right=-1.2] 
\ar[r, "\id \times \beta_Q \times \id"', shift right=1.2] 
\ar[r, "\id \times \id_Q \times \alpha_R", shift right=-1.2] 
&
P \times_{H_0} H_1 \times_{H_0} Q \times_{K_0} R
\ar[d, "\beta_P \times \id_Q \times \id"', shift right=1.2] 
\ar[d, "\id_P \times \alpha_Q \times \id", shift right=-1.2] 
\\
P \times_{H_0} Q \times_{K_0} K_1 \times_{K_0} R
\ar[r, "\id \times \beta_Q \times \id"', shift right=1.2] 
\ar[r, "\id \times \id_Q \times \alpha_R", shift right=-1.2] 
& 
P \times_{H_0} Q \times_{K_0} R 
\end{tikzcd}
\end{equation*}
Since the indexing category is the square of the category $\{0 \rightrightarrows 1\}$ of two parallel arrows, the colimit can be computed in two steps. Computing first the colimit of the columns and then of the resulting row, we obtain $(P \circ_H Q) \circ_K R$. Computing the colimits in the other order, we obtain $P \circ_H (Q \circ_K R)$. The unique natural isomorphism $(P \circ_H Q) \circ_K R \cong P \circ_H Q \circ_K R \cong P \circ_H (Q \circ_K R)$ is the associator.

\begin{Proposition}[Proposition~2.12 in \cite{Blohmann2008}]
\label{prop:GrpdBibuCat}
There is a weak 2-category that has Lie groupoids as objects, smooth right principal bibundles as 1-morphisms, and smooth biequivariant maps of bibundles as 2-morphisms.
\end{Proposition}

\begin{Terminology}
The 2-category of Proposition~\ref{prop:GrpdBibuCat} will be called the \textdef{Morita 2-category} of Lie groupoids and denoted by $\GrpdMrt$. A weak isomorphism in this 2-category is called a \textdef{Morita equivalence} of Lie groupoids.
\end{Terminology}

\begin{Remark}
The cartesian product of Lie groupoids, right principal bibundles, and biequivariant maps is the categorical product of $\GrpdMrt$ \cite[Section~3.4]{Blohmann2008}. In particular, it is a weak symmetric monoidal structure. 
\end{Remark}

\begin{Remark}
A smooth biequivariant map of right principal bibundles is always a diffeomorphism. This shows that $\GrpdMrt$ is a $(2,1)$-category.
\end{Remark}

\begin{Theorem}[Theorem~2.18 in \cite{Blohmann2008}]
The Morita 2-category of Lie groupoids is equivalent to the 2-category of differentiable stacks.
\end{Theorem}

\subsection{Convolution of Lie groupoids and bibundles}

\begin{Definition}
A \textdef{right Haar system} on a Lie groupoid $G_1 \rightrightarrows G_0$ is a collection of strictly positive Radon measures $\lambda_x$ on the source fiber $s^{-1}(x)$ for every $x \in G_0$, satisfying the following conditions:
\begin{itemize}

\item[(i)] Integration is smooth: For every $\phi \in C^\infty_\mathrm{c}(G_1)$, the pushforward function $s_* \phi: G_0 \to \bbC$ obtained by fiber integration
\begin{equation*}
  (s_* \phi)(x)
  := 
  \int_{\mathclap{s^{-1}(x)}} \phi(h) \,\di \lambda_x(h) 
\end{equation*}
is smooth.
		
\item[(ii)] Integration is right $G$-invariant: For every $\phi \in C^\infty_\mathrm{c}(G_1)$ and $g \in G_1$, we have
\begin{equation*}
  \int_{\mathclap{s^{-1}(s(g))}} \phi(h) \,\di \lambda_{s(g)}(h) 
  = 
  \int_{\mathclap{s^{-1}(t(g))}} \phi(hg) \,\di \lambda_{t(g)}(h)
  \,.
\end{equation*}
\end{itemize}
\end{Definition}

A right Haar system on a groupoid can be obtained by choosing a nowhere vanishing smooth section of density bundle of the Lie algebroid and then extending it to a density on the $s$-fibers by right $G$-translation. Haar systems are not unique. The convolution product and action defined below both depend, up to isomorphism, on the choice of the Haar system. These choices are absorbed in the 2-categorical level of the convolution functor, as explained in Remark~\ref{rmk:HaarSystemChoice}.

\begin{Definition}
\label{def:ConvProd}
Let $G$ be a Lie groupoid with a right Haar system $\{\lambda_x\}_{x\in G_0}$. The \textdef{convolution product} of two functions $a, b \in C^\infty_\mathrm{c}(G_1)$ is the function
\begin{equation}
\label{eq:ConvProd}
  (a * b)(g)
  =
  \int_{\mathclap{s^{-1}(s(g))}}
  a(gh^{-1})\, b(h) \, \di\lambda_{s(g)}(h)
\end{equation}
on $G_1$.
\end{Definition}

The convolution product is linear in each variable. Let $K$ be the compact support of $a$ and $L$ the compact support of $b$. Then the support of $a * b$ is contained in $KL$, the set of all products of composable arrows, which is compact. This shows that the convolution product defines a bilinear map
\begin{equation*}
  *: 
  C^\infty_\mathrm{c}(G_1) \times C^\infty_\mathrm{c}(G_1)
  \longrightarrow
  C^\infty_\mathrm{c}(G_1)
\end{equation*}
It follows from the right invariance of the Haar system and the Fubini theorem that the product is associative. The upshot is that the convolution product defines a non-unital algebra on the vector space $A(G) := C^\infty_\mathrm{c}(G_1)$.

\begin{Definition}
Let $G$ be a Lie groupoid and $l: P \to G_0$ a left $G$-bundle. The \textdef{left convolution action} of a function $a \in C^\infty_\mathrm{c}(G_1)$ on a function $m \in C^\infty_\mathrm{c}(P)$ is the function
\begin{equation}
\label{eq:ConvActLeft}
  (a \cdot m)(p)
  := \int_{\mathclap{s^{-1}(l(p))}} 
  a(h^{-1})\, m(h \cdot p)\, \di \lambda_{l(p)}(h)
\end{equation}
on $P$.
\end{Definition}

As for the convolution product, it follows from the compactness of the support of $a$ and $m$ that the support of $a\cdot m$ is compact, and from the right invariance of the Haar system and the Fubini theorem that the convolution action is associative. In other words, the vector space $M(P) := C^\infty_\mathrm{c}(P)$ with the left convolution action is a left $A(G)$-module.

Given a right $H$ action on $r: P \to H_0$, the \textdef{right convolution action} of $b \in A(H)$ on $m \in M(P)$ is the function
\begin{equation}
\label{eq:ConvActRight}
  (m \cdot b)(p)
  := \int_{\mathclap{s^{-1}(r(p))}} 
  m(p \cdot h^{-1})\, b(h)\, \di \lambda_{r(p)}(h)
  \,,
\end{equation}
which equips $M(P)$ with a right $A(H)$-action. If $P$ is an $G$-$H$ bibundle, it follows from the commutativity of the left and right actions that the left $A(G)$-action and the right $A(H)$-action on $M$ commute. 

If $P'$ is another $G$-$H$ bibundle, and $\phi: P \to P'$ a biequivariant isomorphism, then the pullback
\begin{equation}
\label{eq:ConvPullback}
  \phi^*: C_\mathrm{c}^\infty(P') 
  \longrightarrow 
  C_\mathrm{c}^\infty(P)
\end{equation}
is $A(G)$-$A(H)$-bilinear. We summarize:

\begin{Proposition}
Let $G$ and $H$ be Lie groupoids; let $P$ and $P'$ be $G$-$H$-bibundles; let $\phi: P \to P'$ be a biequivariant isomorphism. Then:
\begin{itemize}

\item[(i)] The convolution product~\eqref{eq:ConvProd} equips the vector space $A(G) = C^\infty_\mathrm{c}(G_1)$ with the structure of a non-unital algebra.

\item[(ii)] The convolution actions~\eqref{eq:ConvActLeft} and \eqref{eq:ConvActRight} equip the vector space $M(P) = C^\infty_\mathrm{c}(P)$ with the structure of an $A(G)$-$A(H)$-bimodule.

\item[(iii)] The pullback~\eqref{eq:ConvPullback} is an isomorphism $\phi^*: M(P') \to M(P)$ of bimodules.

\end{itemize}
\end{Proposition}

\begin{Warning}
Let $\phi: G \to H$ be a proper homomorphism of Lie groupoids. The pullback $\phi^*: C_\mathrm{c}^\infty(H_1) \to C_\mathrm{c}^\infty(G_1)$ is generally not a homomorphism of convolution algebras. As example, consider the inclusion $i_y: * \to H$, $* \mapsto 1_y$ of a point $y\in H_0$. The pullback is the evaluation of a function on $H_1$ at $1_y$. However, $(a * b)(1_y) = \int_{s^{-1}(y)} a(h^{-1})\, b(h)\, \di\lambda_y(h) \neq a(1_y)\, b(1_y)$.
\end{Warning}

\subsection{The functoriality constraint}

Let $P$ be a right principal $G$-$H$-bibundle and $Q$ a right principal $H$-$K$-bibundle. Let $m \in C^\infty_\mathrm{c}(P)$ and $n \in C^\infty_\mathrm{c}(Q)$. By viewing the function $m \Botimes n \in M(P) \Cotimes M(Q)$ as function in $C^\infty_\mathrm{c}(P \times Q)$, further restricting it to $P \times_{H_0} Q$, followed by fiber integration over the quotient map $\pi: P \times_{H_0} Q \to P \circ_H Q$, we obtain a function in $P \circ_H Q$ defined by
\begin{equation*}
    \tau'(m, n)[p,q]=\int_{s^{-1}(r(p))} m(p\cdot h^{-1})\,n(h\cdot q) \,\di \lambda_{r(p)}(h)
    \, .
\end{equation*}
It follows from the right invariance of the Haar system that the map
\begin{equation*}
  \tau': M(P) \Botimes M(Q) 
  \longrightarrow 
  C^\infty_\mathrm{c}(P \circ_H Q)
\end{equation*}
is left $A(G)$-linear, right $A(K)$-linear, and $A(H)$-tensorial. By the universal property of the tensor product of bimodules, it descends to a morphism of bimodules
\begin{equation*}
  \tau_{P,Q}: M(P) \Botimes_{A(H)} M(Q) 
  \longrightarrow
  M(P \circ_H Q)
  \,.
\end{equation*}
The situation can be depicted by the following commutative diagram: 
\begin{equation}
\label{eq:taudiagram}
\begin{tikzcd}[column sep=2.5em]
M(P) \Botimes
A(H) \Botimes
M(Q)
\ar[r, shift left=1.2, "{\rho\otimes  \id}"]
\ar[r, shift left=-1.2, "{\id \otimes \lambda}"']
\ar[d]
&
M(P) \Botimes 
M(Q)
\ar[r]
\ar[d]
&[-1em]
M(P) \Botimes_{A(H)}
M(Q)
\ar[dd, "\tau_{P,Q}"]
\\
C_\mathrm{c}^\infty(P \times H_1 \times Q)
\ar[d, "i^*"]
&
C_\mathrm{c}^\infty(P \times Q)
\ar[dr, "\tau'"]
\ar[d, "j^*"]
&
\\
C_\mathrm{c}^\infty(
P \times_{H_0} H_1 \times_{H_0} Q
)
\ar[r, shift left=1.2, "{(\beta \times \id_Q)_*}"]
\ar[r, shift left=-1.2, "{(\id_P \times \alpha)_*}"']
&
C_\mathrm{c}^\infty(
P \times_{H_0} Q
)
\ar[r, "\pi_*"']
&
C_\mathrm{c}^\infty(
P \circ_H Q
)
\,.
\end{tikzcd}
\end{equation}
Here, $\rho$ is the right $A(H)$-action on $M(P)$ and $\lambda$ the left $A(H)$-action on $M(Q)$. The top row is the coequalizer defining the tensor product over $A(H)$. The vertical maps $i^*$ and $j^*$ are the pullbacks along the embedding of the fiber products and $\alpha_*$, $\beta_*$, and $\pi_*$ the maps obtained by fiber integration. The left square is serially commutative, so that we obtain the morphism $\tau_{P,Q}$ between the coequalizers of each row. 

It follows from the naturality of this construction that $\tau_{P,Q}$ is natural in $P$ and $Q$. That is, if $\phi:P\to P'$ and $\psi:Q\to Q'$ are biequivariant diffeomorphisms, then the diagram
\begin{equation}
\label{diag:tauNatural}
\begin{tikzcd}[column sep=3.5em]
  M(P)\Botimes_{A(H)} M(Q) 
  \ar[r,"\phi_*\Botimes \psi_*"]
  \ar[d,"\tau_{P,Q}"']
  & 
  M(P')\Botimes_{A(H)} M(Q') \ar[d,"\tau_{P',Q'}"]
  \\
  M(P\circ_H Q) \ar[r,"(\phi\circ \psi)_*"']&
  M(P'\circ_H Q')
\end{tikzcd}
\end{equation}
commutes. The convolution functor maps the composition of bibundles naturally to the tensor product of the convolution bimodules if and only if $\tau_{P,Q}$ is a natural isomorphism. We will call $\tau$ the \textdef{functoriality constraint} of convolution (see Appendix~\ref{sec:bicategories} for conventions on $2$-functors).

\subsection{Bornological completion of convolution product and action}
\label{sec:ConvCompletion}

In the category $\Born$ of convex bornological vector spaces, the algebra $(A(G), *)$ is generally not self-induced, the $A(G)$-$A(H)$-bimodule $M(P)$ generally not smooth, and the functoriality constraint $\tau_{P,Q}$ generally not an isomorphism. This means that convolution does not map groupoids and bibundles to algebras and bimodules in the Morita category, and that the composition of bibundles is not mapped to the composition of bibundles. We will show in our main Theorem~\ref{thm:ConvFunc} that these issues are solved by the bornological completion of the convolution product, the convolution actions, and the functoriality constraint.

By the universal property of the bornological tensor product, (Definition~\ref{def:TensorBorn}), the convolution product can be viewed as a bounded linear map
\begin{equation*}
  * : A(G) \Botimes A(G) \longrightarrow A(G)
  \,.
\end{equation*}
By applying the monoidal functors $\Sep: \Born \to \sBorn$ and $\Comp: \sBorn \to \cBorn$, we obtain the complete convolution product, which we will also denote by
\begin{equation*}
  * := (\Comp \Sep)(*): 
  A(G) \Cotimes A(G) \longrightarrow A(G)
  \,.
\end{equation*}
It will be clear from the context, when $*$ denotes the completed convolution product. Similarly, we can complete the right and left convolution action, which yields maps
\begin{equation*}
\begin{aligned}
  A(G) \Cotimes M(P) 
  &\longrightarrow M(P)
  \\
  M(P) \Cotimes A(H) 
  &\longrightarrow M(P)
  \,.
\end{aligned}
\end{equation*}

To describe the completed convolution product and module structures explicitly, we recall from Section~\ref{sec:TestFuncBorn} that $A(G) = C_\mathrm{c}^\infty(G_1)$ is already complete as a bornological vector space. (It follows from Proposition~\ref{prop:TensSepIsSep} that $\Sep( A(G) \Botimes A(G)) = A(G) \Sotimes A(G)$ as bornological vector space.) The map from the algebraic tensor product with the tensor product bornology (Definition~\ref{def:TensorBorn}) into the complete tensor product, $A(G) \Botimes A(G) \to A(G) \Cotimes A(G)$, is bounded. The convolution product~\eqref{eq:ConvProd} on $A(G) \Botimes A(G) \to A(G)$ can be obtained by composing the following morphisms of bornological vector spaces
\begin{equation}
\label{eq:ConvAlgSeq}
\begin{split}
  C_\mathrm{c}^\infty(G_1) \Botimes
  C_\mathrm{c}^\infty(G_1)
  &\xrightarrow{~JI~} 
  C_\mathrm{c}^\infty(G_1) \Cotimes
  C_\mathrm{c}^\infty(G_1)
  \\
  &\xrightarrow{~\cong~}
  C_\mathrm{c}^\infty(G_1 \times G_1)
  \\
  &\xrightarrow{~i^*~}
  C_\mathrm{c}^\infty (G_1\times_{G_0}^{s,t} G_1)
  \\
  &\xrightarrow{~m_*~}
  C_\mathrm{c}^\infty(G_1)
  \,,
\end{split}
\end{equation}
where $i^*$ is the pullback along the inclusion $i: G_1 \times_{G_0} G_1 \to G_1 \times G_1$ and $m_*$ the integration over the fibers of groupoid multiplication $m: G_1 \times_{G_0}^{s,t} G_1 \to G_1$. 

More precisely, we use the bar isomorphism \cite[Chapter 1.5]{Brown1982} for groupoids,
\begin{equation*}
\begin{tikzcd}[row sep=0.0ex]
\mathllap{G_2 =~} G_1 \times_{G_0}^{s,t} G_1
\ar[r, "\GrpdBar", shift left=1]
&  
G_1 \times_{G_0}^{s,s} G_1 
\mathrlap{~= \bar{G}_2}
\ar[l, "\GrpdBar^{-1}", shift left=1]
\\
(g_1, g_2)
\ar[r, mapsto]
&
(g_1 g_2, g_2)
\\
(gh^{-1}, h)
&
(g,h)
\ar[l, mapsto]
\end{tikzcd}
\end{equation*}
and pull back along
\begin{align*}
  \bar{\imath} := \GrpdBar^{-1} \circ i: \bar{G}_2 
  &\longrightarrow G_1 \times G_1\\
\intertext{followed by the fiber integration of}
  \bar{m} := m \circ \GrpdBar^{-1}: \bar{G}_2 &\longrightarrow G_1
  \,.
\end{align*}
The tensor product of two functions $a, b \in C_\mathrm{c}^\infty(G_1)$ is the function $a \Cotimes b \in C_\mathrm{c}^\infty(G_1 \times G_1)$ given by the pointwise product $(a \Cotimes b)(g_1, g_2) = a(g_1) b(g_2)$. Its pullback along $\bar{\imath}$ is the function given by
\begin{equation*}
  \bigl( \bar{\imath}^{\,*}(a \Cotimes b) \bigr)(g,h)
  = a(gh^{-1})b(h)
\end{equation*}
for all $(g,h) \in \bar{G}_2$, which is the product of the functions $\bar{a}(g,h) = a(gh^{-1})$ and $\bar{b}(g,h) = b(h)$. Since $\bar{m} = \pr_1$, the $g$-fiber of $\bar{m}$ is given by
\begin{equation*}
  \bar{m}^{-1}(g)
  = \{g\} \times_{G_0}^{s,s} G_1
  = s^{-1}(s(g))
  \,,
\end{equation*}
which is equipped with the measure $\lambda_{s(g)}$ of the  Haar system. We conclude that the fiber integration of $\bar{\imath}^{\,*}(a \Cotimes b)$ over the fibers of $\bar{m}^{-1}(g)$ yields the formula~\eqref{eq:ConvProd}.


Similarly, the right module action of $A(H)$ on $M(P)$ can be obtained by the sequence of bornological vector spaces
\begin{equation}
\label{eq:ConvModSeq}
\begin{split}
  C_\mathrm{c}^\infty(P) \Botimes
  C_\mathrm{c}^\infty(H_1)
  &\xrightarrow{~JI~} 
  C_\mathrm{c}^\infty(P) \Cotimes
  C_\mathrm{c}^\infty(H_1)
  \\
  &\xrightarrow{~\cong~}
  C_\mathrm{c}^\infty(P \times H_1)
  \\
  &\xrightarrow{~j^*~}C_\mathrm{c}^\infty (P\times_{H_0} H_1)
  \\
  &\xrightarrow{~\beta_*~}
  C_\mathrm{c}^\infty(P)
  \,,
\end{split}
\end{equation}
where $j^*$ is the pullback along the inclusion $j: P \times_{H_0} H_1 \to P \times H_1$ and $\beta_*$ the integration over the fibers of the right groupoid action $\beta: P \times_{H_0} H_1 \to P$. The fiber of $\beta$ over $p$ is given by
\begin{equation*}
  \beta^{-1}(p) 
  = \{ (p \cdot h^{-1}, h) \} 
  \cong s^{-1}(r(p))
  \,,
\end{equation*}
which is equipped with a measure from the Haar system. There is an analogous sequence of bounded maps for the left action. This leads us to the following statements.

\begin{Proposition}
\label{prop:CompletedAlgsMods}
Let $G$, $H$, $K$, $L$ be Lie groupoids; let $P$ and $P'$ be right principal $G$-$H$-bibundles; let $Q$ be a right principal $H$-$K$-bibundle; let $R$ be a right principal $K$-$L$ bibundle. Then:
\begin{itemize}

\item[(i)] The completion of the convolution product equips $A(G)$ with the structure of an associative algebra in $\cBorn$.

\item[(ii)] The completions of the left and right convolution actions equip $M(P)$ with the structure of an $A(G)$-$A(H)$-bimodule in $\cBorn$.

\item[(iii)] For every biequivariant diffeomorphism of bibundles $\phi: P \to P'$, the pushforward $\phi_*: M(P') \to M(P)$ is an isomorphism of bimodules in $\cBorn$.

\item[(iv)] The completion of the functoriality constraint $\tau_{P,Q}$ yields a morphism 
\begin{equation*}
  \hat{\tau}_{P,Q}: M(P)\Cotimes_{A(H)} M(Q) 
  \longrightarrow M(P\circ_H Q)   
\end{equation*}
of bimodules in $\cBorn$ that is natural in $(P,Q)$.

\item[(v)] The completion of the convolution algebra of a product of groupoids is isomorphic to the tensor algebra, $A(G \times H) \cong A(G) \Cotimes A(H)$. The convolution of a product of bibundles $P \times R$ viewed as $(G \times K)$-$(H \times L)$ bibundle is isomorphic to the tensor product $M(P \times R) \cong M(P) \Cotimes M(R)$ viewed as $A(G) \Cotimes A(K)$-$A(H) \Cotimes A(L)$ bimodule.

\end{itemize}
Moreover, bornological completion commutes with the composition of bimodules,
\begin{equation}
\label{eq:CompletedAlgsTens}
  (\Comp\Sep)\bigl( M(P) \Botimes_{A(H)} M(Q) \bigr)
  \cong
  M(P) \Cotimes_{A(H)} M(Q)
  \,.
\end{equation}
\end{Proposition}
\begin{proof}
Since $\Comp$ and $\Sep$ are symmetric monoidal functors, $\Comp\Sep$ maps algebras to algebras, modules to modules, and morphisms of modules to morphisms of modules, which proves (i), (ii), and (iii). Since $\Comp$ and $\Sep$ are left adjoints, $\Comp\Sep$ preserves colimits, in particular the coequalizer~\eqref{eq:TensProdModCoeq} defining the tensor product of modules. This implies~\eqref{eq:CompletedAlgsTens}. Now, when we apply $\Comp\Sep$ to Diagram~\eqref{diag:tauNatural}, we obtain the commutative square defining the naturality of $\hat{\tau}_{P,Q}$. Finally, $C_{\mathrm{c}}^\infty(X)\Cotimes C_{\mathrm{c}}^\infty(Y) \cong C_{\mathrm{c}}^\infty(X\times Y)$ implies (v).
\end{proof}

\section{Functoriality of bornological groupoid convolution}
\label{sec:functor}
\subsection{The main theorem}

We are now ready to state our main results.

\begin{Theorem}
\label{thm:ConvFunc}
The assignments of the complete bornological convolution algebras $G \mapsto A(G)$, bimodules $P \mapsto M(P)$, and pushforwards of biequivariant diffeomorphisms $\phi \mapsto \phi_*$, together with the completed functoriality constraint $(P,Q) \mapsto \hat{\tau}_{P,Q}$ define a weak monoidal 2-functor
 \begin{equation*}
  \GrpdMrt \longrightarrow \AlgMrt(\cBorn)
\end{equation*}
from the geometric Morita 2-category of Lie groupoids, right principal bibundles, and biequivariant diffeomorphisms (Proposition~\ref{prop:GrpdBibuCat}) to the Morita 2-category of self-induced algebras, smooth bimodules, and bilinear morphisms in the category of complete bornological vector spaces (Proposition~\ref{prop:MorCat}).
\end{Theorem}

\begin{Corollary}
If two Lie groupoids are Morita equivalent then so are their completed bornological convolution algebras.
\end{Corollary}

\begin{Remark}
Since $\GrpdMrt$ is equivalent to the 2-category of differentiable stacks, the convolution functor of Theorem~\ref{thm:ConvFunc} can be viewed as a 2-functor on differentiable stacks. The convolution algebra of a differentiable stack can be interpreted as its smooth noncommutative geometry. We are not aware of an alternative construction that does not pass through the presentation of the stacks by Lie groupoids.
\end{Remark}

\begin{Remark}
\label{rmk:HaarSystemChoice}
The definition of the convolution algebra depends on a choice of Haar system. There are several ways to manage this dependence. A right Haar system corresponds bijectively to a nowhere vanishing smooth section of the density bundle $\mathrm{Dens}(A)$ of the Lie algebroid $A \to G_0$, via right translation. It follows that two Haar systems are related by multiplication by a positive smooth function on $G_0$, which gives rise to an explicit isomorphism between the associated convolution algebras.
    
Haar systems can be avoided by defining convolution in terms of compactly supported half-densities \cite[Chapter~2.5]{Connes1994} or densities along the source fibers (e.g.~in \cite{Posthuma23}). This makes the choice of a measure part of every element of the convolution algebra, which makes the underlying vector space larger and studying the bornology more cumbersome. For the sake of readability and concreteness, we have decided to work with compactly supported functions.

To make this precise, we can either fix a Haar system on each Lie groupoid and define a 2-category of pairs $(G, \lambda)$. The forgetful functor from this 2-category to the 2-category of Lie groupoids is an equivalence. Or we can consider the choice of a Haar system as part of the construction of the functor. Since we are working in the setting of \emph{weak} 2-categories, this has no bearing on Theorem~\ref{thm:ConvFunc}.
\end{Remark}

\begin{Proposition}
\label{prop:ConvModProjective}
Let $H_1 \rightrightarrows H_0$ be a Lie groupoid; let $X \stackrel{l}{\leftarrow }P \stackrel{r}{\rightarrow} H_0$ be submersions with a right groupoid action that is proper and transitive on the $l$-fibers. Then the convolution $A(H)$-module $M(P)$ is projective in the sense of Definition~\ref{def:ProjectiveMod}.
\end{Proposition}

\begin{Corollary}
\label{cor:ConvMoritaProjective}
The convolution module of a Morita bibundle is left and right projective.
\end{Corollary}

\begin{Corollary}
\label{cor:quasiunitality}
The convolution algebra of a Lie groupoid is quasi-unital (Terminology~\ref{term:QuasiUnital}).
\end{Corollary}

In Proposition~4.12 of \cite{smoothGroupRep}, a special case of Proposition~\ref{prop:ConvModProjective} was proved for the regular action of a subgroup $H \subset G$ on $G$, that is, for the bibundle $G \leftarrow G \rightarrow *$ with the right action by $H \rightrightarrows *$. The following example shows that the assumption for the right bundle map to be a submersion is necessary.

\begin{Example}
Consider the span $* \leftarrow * \stackrel{x}{\rightarrow} \bbR$ for some $x \in \bbR$, with the trivial right action of the groupoid $\bbR \rightrightarrows \bbR$. The real convolution algebra is the commutative algebra of functions $A(\bbR\rightrightarrows \bbR) = C_\mathrm{c}^\infty(\bbR)$. The convolution module is $M(*) = \bbR$ with the action $\bbR \Cotimes C_\mathrm{c}^\infty(\bbR) \to \bbR$, $\nu \Cotimes a \mapsto \nu\, a(x)$. An $\bbR$-linear section $\sigma$ is determined by a single function $\sigma(1) = 1 \Cotimes f$, such that $f(1) = 1$. The condition for $\sigma$ to be $C_\mathrm{c}^\infty(\bbR)$-linear is
\begin{equation*}
  1 \Cotimes f\, a(x)  
  =
  \sigma\bigl( a(x) \bigr)
  =
  \sigma(1 \cdot a)
  \stackrel{!}{=}
  \sigma(1)\, a
  =
  1 \Cotimes fa
  \,,
\end{equation*}
where in the first step we have used the $\bbR$-linearity of $\sigma$. This condition holds for all $a$ if and only if $f = 1$, which does not have compact support. We conclude that the action does not have a $C_\mathrm{c}^\infty(\bbR)$-linear section.
\end{Example}

The remainder of this section will be devoted to the proofs of Theorem~\ref{thm:ConvFunc} and Proposition~\ref{prop:ConvModProjective}. In the proof of Theorem~\ref{thm:ConvFunc}, we have to show that the completed convolution algebras $A(G)$ are self-induced and that the completed bimodules $M(P)$ are smooth. It follows from the functoriality of the pullback of functions that the pushforward $\phi_* = (\phi^{-1})^*$ of an isomorphism is an isomorphism. For the functoriality we have to show that the completed functoriality constraint $\hat{\tau}_{P,Q}$ is an isomorphism. The identity constraint is given by the equality $M(\id_G) = A(G) = \id_{A(G)}$ of bimodules, so tautological. Finally, we have to check the coherence conditions of a 2-functor, which is straightforward. For Proposition~\ref{prop:ConvModProjective}, we have to show that the action $M(P) \Cotimes A(H) \to M(P)$ has a right $A(H)$-linear section. The projectivity then follows from Proposition~\ref{prop:ModuleSectionProj}.

\subsection{Proof of existence of strong approximate units}
\label{sec:ApproxUnitExists}

\begin{Theorem}
\label{thm:ApproxUnitExists}
The completed bornological convolution algebra $A(G)$ has a strong approximate right (left) unit.
\end{Theorem}

For the case of a Lie group $G$, Meyer has proved the existence of pointwise approximate units on $A(G)$ \cite[Proposition~4.3]{smoothGroupRep}. We generalize his result to groupoids and strengthen it to Mackey convergence in the operator bornology. The idea of the proof is to construct a sequence of functions $e_n \in C_\mathrm{c}^\infty(G_1)$ such that the integration over every source fiber converges to the Dirac distribution. The difficulty is that since $e_n$ must have compact support, it cannot be constant along the base direction of $s: G_1 \to G_0$, yet $a * e_n$ is required to converge uniformly in all derivatives of $a$ in the direction of $G_0$. In the first step, we solve in Lemma~\ref{lem:fiberdirac} this technical problem for the fiber integration of a vector bundle by constructing $e_n$ in local bundle trivializations and then using a partition of unity argument. In the second step, we show in Lemma~\ref{lem:fiberdirac} that the convergence property of $e_n$ is stable under pullbacks of the vector bundle, which will be needed for the convolution product. In the third step, we find a suitable tubular neighborhood $U \supset 1(G_0)$ of the identity bisection that is isomorphic, as a bundle $s:U \to G_0$ to the Lie algebroid $A \to G_0$, to which we can then apply the two lemmas. First, we recall the following well-known result.

\begin{Lemma}
\label{lem:DiracApproximation}
Let $\epsilon$ be a nonnegative compactly supported smooth function on $\bbR^d$ with $\int_{\bbR^d}\epsilon(y) \, \di y = 1$, where $y = (y^1, \ldots, y^d)$ and $\di y = dy^1\cdots dy^d$ the Lebesgue measure; let $\epsilon_n= n^d\epsilon(ny)$ for $n \in \bbN$. Then the sequence of functionals $\int\Empty \epsilon_n: C_\mathrm{c}^\infty(\bbR^d) \to \bbR$, $f \mapsto \int_{\bbR^d}f\epsilon_n \,\di y$ Mackey converges in $\intHom(C_\mathrm{c}^\infty(\bbR^d), \bbR)$ to the Dirac delta functional $\delta_0: f \mapsto f(0)$.
\end{Lemma}
\begin{proof}
Let $K := \Supp \epsilon \subset \bbR^d$, which is compact by assumption. By standard estimates, there exists a constant $C < \infty$ such that for all $f \in C_\mathrm{c}^\infty(\bbR^d)$ we have
\begin{equation}
\label{eq:DiracApproximation01}
  \Bigl| \int_{\bbR^d} f(y)\epsilon_n(y)\,\di y-f(0)\Bigr|
  \leq 
  \frac{1}{n}
  C\sup_{y\in K} \lVert (\nabla f)(y)\rVert \,.
\end{equation}
Let $B \subset C_\mathrm{c}^\infty(\bbR^d)$ be bounded. Then $p_1(B) := \sup_{f \in B} \sup_{y \in K} \lVert (\nabla f)(y)\rVert$ is finite. Let $D$ be the equibounded (Terminology~\ref{term:equibounded}) set of functionals that map any bounded subset $B\subset C_\mathrm{c}^\infty(X)$ to the interval $[-Cp_1(B),Cp_1(B)]$. Inequality~\eqref{eq:DiracApproximation01} shows that $\int \Empty\epsilon_n -\delta_0 \in \frac{1}{n}D$. We conclude that the sequence converges in $\lVert \Empty \rVert_D$.
\end{proof}

Let $\pi: Z \to X$ be a smooth map of manifolds. We denote
\begin{equation*}
    C_{/X}^\infty(Z)
    := \Colim_{
    \substack{
      K\subset X\\
      K\text{ compact}
    }} C^\infty_{\pi^{-1}(K)}(Z) 
    \,.
\end{equation*}
As vector space, this is the space of all smooth functions $f: Z \to \bbR$ such that $\pi(\Supp f)$ is precompact. A sequence Mackey converges in $C_{/X}^\infty(Z)$ if the sequence Mackey converges in $C^\infty_{\pi^{-1}(K)}(Z)$ for some fixed compact $K$.

Assume that $\Supp g \cap \pi^{-1}(K)$ is compact for every compact $K \subset X$. If $\Supp f \subset \pi^{-1}(K)$ for some compact $K$, then $\Supp(fg)$ is compact. This shows that we have a map
\begin{equation*}
\begin{aligned}
  r_g:
  C_{/X}^\infty(Z)
  &\longrightarrow
  C_\mathrm{c}^\infty(Z)
  \\
  f
  &\longmapsto
  fg
  \,,
\end{aligned}
\end{equation*}
which is linear and bounded. 

We recall from Proposition~\ref{prop:TestFiberIntegr} that, if $\pi: Z \to X$ is a surjective submersion and the fibers are equipped with a strictly smooth positive density along the fibers, the fiber integration,
\begin{equation*}
  \pi_*: 
  C_\mathrm{c}^\infty(Z) \longrightarrow
  C_\mathrm{c}^\infty(X)
\end{equation*}
is a bounded linear map.

\begin{Lemma}
\label{lem:fiberdirac}
Let $\pi: A\to X$ be a vector bundle; let $0_A: X \to A$ denote its zero section; let $\rho$ be a smooth, strictly positive density along the fibers of $\pi$. Then there exists a sequence $(e_n)_{n \in \bbN}$ of functions in $C_\mathrm{c}^\infty(A)$ such that
\begin{equation*}
    \lim_{n\to \infty} \pi_*\circ r_{e_n} = 0_A^* \,,
\end{equation*}
in $\intHom(C^\infty_{/X}(A), C_\mathrm{c}^\infty(X))$.
\end{Lemma}

\begin{proof}
By assumption, the manifold $X$ is second countable, so that we have a countable, locally finite open cover $\{U_i~|~i \in \bbN\}$ of $X$ by open balls and a partition of unity $\{\chi_i: X \to [0,1]\}$ subordinate to the cover, $\Supp \chi_i \subset U_i$. Since $U_i$ is contractible, $A$ has a trivialization over $U_i$ by bundle coordinates $(x,y) = (x^1, \ldots, x^p, y^1, \ldots, y^q) : A|_{U_i} \to \bbR^p \times \bbR^q$, where $p$ is the dimension of $X$ and $q$ the rank of $A$. 

The fiber density $\rho$ is given in local coordinates by $\di\rho = \rho_i'(x,y)\,\di y$, where $\di y = |dy^1 \wedge \ldots \wedge dy^q|$ denotes the canonical top density of the fiber coordinates and where, by assumption on $\rho$, the Radon-Nikodym derivative $\rho_i': A|_{U_i} \to \bbR$ is smooth and strictly positive. 
    
Let $\epsilon_n: \bbR^q \to \bbR$ be a sequence of smooth compactly supported functions that approximate the Dirac delta distribution as constructed in Lemma~\ref{lem:DiracApproximation}. Let the functions $e_{n,i}: A|_{U_i} \to \bbR$ be defined in local coordinates by
\begin{equation*}
  e_{n,i}(x,y) 
  := \frac{\chi_i(x)\epsilon_n(y)}{\rho_i'(x,y)}
  \,.
\end{equation*}
Since $\chi_i$ is compactly supported in $U_i$ and $\epsilon_n$ compactly supported in $\bbR^q$, $e_{n,i}$ is compactly supported in $A|_{U_i}$. For any smooth function $f: A \to \bbR$, $f e_{n,i}$ lies in  $C_\mathrm{c}^\infty(A|_{U_i})$ and its fiber integration is given by
\begin{equation*}
  \bigl( \pi_*(f e_{n,i}) \bigr)(x)
  =
  \int_{\bbR^q} f(x,y) \chi_i(x) \epsilon_n(y) \, \di y
  \,.
\end{equation*}
Let $e_n := \sum_{i=1}^n e_{n,i}$. The sequence $e_n$ has compact support that simultaneously grows along the base direction and shrinks in the fiber direction to the zero section. The fiber integration $f\mapsto \pi_*(fe_n) = \sum_{i=1}^n\pi_* (fe_{n,i})$ can be written as a composition of bounded linear maps as follows.

In the first step, we observe that for $K \subset X$ compact, there is a finite subset $I_K \subset \bbN$ such that $K \cap U_i = \empty$ for all $i \notin I_K$. We have the morphism
\begin{equation*}
\begin{aligned}
  C_K^\infty(A)
  &\longrightarrow
  \prod_{i \in I_K} C_\mathrm{c}^\infty(A|_{U_i}) \cong
  \bigoplus_{i \in I_K} C_\mathrm{c}^\infty(A|_{U_i})
  \\
  f &\longmapsto (f\,\pi^*\!\chi_i)_{i \in I_K}
  \,,
\end{aligned}  
\end{equation*}
where we have used that finite products and coproducts in $\cBorn$ are the same. Taking the colimit over all compact $K \subset X$, we obtain the morphism
\begin{equation}
\label{eq:fiberdirac01}
  C^\infty_{/X}(A) 
  \xrightarrow{~(r_{\pi^*\!\chi_i})_{i \in \bbN} ~}
  \bigoplus_i C_\mathrm{c}^\infty(A|_{U_i})
\end{equation}
In the second step, we use that over each of the contractible $U_i$ we have a trivialization $A|_{U_i} \cong U_i \times \bbR^q$, which induces an isomorphism
\begin{equation}
\label{eq:fiberdirac02}
  \bigoplus_{i} C_{/U_i}^\infty(A|_{U_i})
  \cong~
  \bigoplus_i 
  \bigl( C_\mathrm{c}^\infty(U_i)\Cotimes C^\infty(\bbR^q)
  \bigr)
  \,,
\end{equation}
In the third step, we use that the operator of integration $\int \Empty \varepsilon_n: C_\mathrm{c}^\infty(\bbR^q) \to \bbR$ induces a morphism 
\begin{equation}
\label{eq:fiberdirac03}
  \bigoplus_i 
  \bigl(
  C_\mathrm{c}^\infty(U_i)
  \Cotimes C^\infty(\bbR^q)
  \bigr) 
  \xrightarrow{\,\bigoplus_i \id_{C_\mathrm{c}^\infty(U_i)} \Cotimes \int\! \Empty \varepsilon_n\,}
  \bigoplus_i 
  \bigl( C_\mathrm{c}^\infty(U_i) \Cotimes \bbR \bigr)
  \cong
  \bigoplus_i C_\mathrm{c}^\infty(U_i)
  \,.
\end{equation}
In the fourth step, we implement the truncation of the coproduct to $i \leq n$, which is given by the partial identities \begin{equation}
\label{eq:fiberdirac04}
  \bigoplus_i C_\mathrm{c}^\infty(U_i)
  \xrightarrow{~I_n~}
  \bigoplus_i C_\mathrm{c}^\infty(U_i) 
  \longrightarrow
  C_\mathrm{c}^\infty(X)
\end{equation}
of Lemma~\ref{lem:convergenceFiniteDiagonals} that map $\oplus_{i=1}^\infty g_i \mapsto \oplus_{i=1}^n g_i$. In the fifth and last step, we add the components of the coproduct,
\begin{equation}
\label{eq:fiberdirac05}
  \bigoplus_i C_\mathrm{c}^\infty(U_i) 
  \xrightarrow{~\sum_i~}
  C_\mathrm{c}^\infty(X) 
  \,,    
\end{equation}
which takes values in compactly supported functions since every elements of coproduct has only a finite number of non-zero components. The composition of \eqref{eq:fiberdirac01}, \eqref{eq:fiberdirac02}, \eqref{eq:fiberdirac03}, \eqref{eq:fiberdirac04}, and~\eqref{eq:fiberdirac05} is the map
\begin{equation*}
\begin{tikzcd}[row sep=2.6ex, column sep=1.5em]
  C^\infty_{/X}(A) \ar[d]
  &  
  f \ar[d,mapsto]
  \\
  \bigoplus_i C_{/U_i}^\infty(A|_{U_i}) \ar[d, "\cong"]
  & 
  ( f \pi^*\chi_i )
  \ar[dd, mapsto]
  \\
  \bigoplus_i C_\mathrm{c}^\infty(U_i)\Cotimes C^\infty(\bbR^q) \ar[d,"\bigoplus_i \id\Cotimes \int \Empty \epsilon_n"] 
  &
  \\
  \bigoplus_i C_\mathrm{c}^\infty(U_i) \ar[d,"I_n"]
  &
  \bigl( \pi_* (f e_{n,i}) \bigr) \ar[d,mapsto]
  \\
  \bigoplus_i C_\mathrm{c}^\infty(U_i) \ar[d]
  &
  \bigl(\pi_* (f e_{n,1}), \ldots, \pi_* (f e_{n,n}) \bigr) \ar[d,mapsto]
  \\
  C_\mathrm{c}^\infty(X) 
  &
  \sum_{i=1}^n \pi_* (f e_{n,i})
\end{tikzcd}
\end{equation*}
which is $\pi_* \circ r_{e_n}$.

By Lemma~\ref{lem:DiracApproximation}, the fiber integrations Mackey converge to $\lim_n \id \Cotimes \int\Empty \epsilon_n = \id \Cotimes \delta_0$, which is the restriction to the zero section in every trivialization of the vector bundle. By Lemma~\ref{lem:convergenceFiniteDiagonals}, the partial identities $I_n$ Mackey converge to the identity. By Proposition~\ref{prop:convergenceComposition}, the composition of convergent sequences of morphisms of bornological spaces converges to the composition of their limits. We conclude that $\pi_* \circ r_{e_n}$ Mackey converges to $0_A^*$.
\end{proof}

\begin{Remark}
As explained in the proof of~Proposition~\ref{prop:StrongPointwiseApprUnit}, the Mackey convergence of a sequence of maps in $\intHom(V,W)$ implies the pointwise Mackey convergence when applied to every $v \in V$. In particular, Lemma~\ref{lem:fiberdirac} implies that for every $f \in C^\infty(Z)$ with $\pi(\Supp f)$ precompact, the sequence of fiber integrations of $fe_n$ Mackey converges as
\begin{equation*}
  \pi_*(f e_n) \xrightarrow{~n\to \infty~}
  f \circ 0_A
\end{equation*}
in $C_\mathrm{c}^\infty(X)$, that is, uniformly in all derivatives with uniformly compact support. 
\end{Remark}

\begin{Remark}
\label{rmk:approximateunitcompactsupport}
By applying Lemma~\ref{lem:fiberdirac} to the trivial vector bundle $X\times \{0\}\to X$, we conclude that $C_\mathrm{c}^\infty(X)$ has a strong approximate unit.
\end{Remark}

Let $\pi: A \to X$ be a vector bundle with a fiberwise density as in Lemma~\ref{lem:fiberdirac}. Let $\phi: Z \to X$ be a smooth map. Consider the pullback vector bundle:
\begin{equation}
\label{eq:ApprUnit01}
\begin{tikzcd}
\mathllap{\phi^*\! A :=~} Z \times_X A
\ar[r, "\pr_2"]
\ar[d, "\tilde{\pi} := \pr_1"']
&
A
\ar[d, "\pi"]
\\
Z
\ar[r, "\phi"']
&
X
\end{tikzcd}
\end{equation}
We will denote the zero section of $\phi^*\! A$ by
\begin{equation*}
  0_{\phi^*\! A} := (\id_Z, 0_A \circ \phi): 
  Z \longrightarrow Z \times_X A
  \,.
\end{equation*}
The $\tilde{\pi}$-fiber over $z \in Z$ is naturally diffeomorphic to the $\pi$-fiber over $\phi(z)$, which shows that the pullback bundle is naturally equipped with a smooth strictly positive fiberwise density. The fiber integration of the pullback will be denoted by $\tilde{\pi}_*$. Let $(e_n)_{n \in \bbN}$ be the sequence of functions of Lemma~\ref{lem:fiberdirac} such that $\lim_{n \to \infty} \pi_* \circ r_{e_n} = 0_A^*$. Let the pullback functions be denoted by $\tilde{e}_n := \pr_2^* e_n \in C^\infty(Z \times_X A)$.

\begin{Corollary}
\label{cor:fiberdiracPullback}
$\lim_{n\to \infty} \tilde{\pi}_*\circ r_{\tilde{e}_n} = 0_{\phi^*\! A}^*$.    
\end{Corollary}
\begin{proof}
If ${U_i}$ is a cover of $X$ with subordinate partition of unity $\{\chi_i: U_i \to \bbR\}$, then $\tilde{U}_i := \phi^{-1}(U_i)$ is a cover of $Z$ with a subordinate partition of unity $\tilde{\chi}_i := \phi^* \chi_i$. Now we repeat the arguments of Lemma~\ref{lem:fiberdirac}.
\end{proof}

\begin{proof}[Proof of Theorem~\ref{thm:ApproxUnitExists}]
As we have explained in Section~\ref{sec:ConvCompletion}, the convolution product on $A(G) \Cotimes A(G)$ is given by the pullback along the embedding $\bar{\imath}: \bar{G}_2 = G_1 \times_{G_0}^{s,s} G_1 \hookrightarrow G_1 \times G_1$ followed by the fiber integration along the multiplication $\bar{m} = \pr_1: \bar{G}_2 \to G_1$. 

The groupoid unit $1: G_0 \to G_1$ embeds $G_0$ as a closed submanifold in $G_1$. Since $s \circ 1 = \id_{G_0}$, all points on $1(G_0)$ are regular points of $s$. Therefore, we can find a tubular neighborhood $V \supset 1(G_0)$. By the usual construction of tubular neighborhoods, $V$ is $s$-vertically precompact, that is for every compact $K \subset G_1$, $s^{-1}(K) \cap V$ is precompact. By applying the groupoid inverse to $V$, we obtain the open neighborhood $V^{-1} \supset 1(G_0)$. Since the groupoid inverse intertwines the source and target maps, $V^{-1}$ is $t$-vertically precompact, that is, $t^{-1}(K) \cap V$ is precompact. The intersection $U := V \cap V^{-1}$ is both, $s$-vertically and $t$-vertically precompact. In other words, both $s|_U: U \to G_0$ and $t|_U:U \to G_0$ are proper. $U$ is still a tubular neighborhood of the identity bisection, so we have an isomorphism of smooth fiber bundles
\begin{equation*}
\begin{tikzcd}[column sep=0.5em, row sep=3ex]
U 
\ar[rr, "j", "\cong"']
\ar[dr, "s|_U"']
&
&
A
\ar[dl, "\pi"]
\\
& G_0 &
\end{tikzcd}
\end{equation*}
where $\pi: A \to G_0$ is a vector bundle. We can choose $A$ to be the Lie algebroid, but this will not be relevant for the proof.

Let $\lambda$ be the smooth strictly positive fiberwise density on the fibers of $s: G_1 \to G_0$ that defines a right Haar system on $G_1$. Its restriction $\lambda|_U$ to $U$ pulls back to a smooth strictly positive fiberwise density $\rho := j^* \lambda|_U$ on $A$. By Lemma~\ref{lem:fiberdirac}, we obtain functions $e'_n \in C_\mathrm{c}^\infty(A)$, such that $\lim_n \pi_* \circ r_{e'_n} = 0_A^*$. The pullbacks
\begin{equation*}
  e_n := j^* e'_n 
\end{equation*}
can be viewed as a smooth functions on $G_1$ with compact support contained in $U$. We will now show that $(e_n)_{n \in \bbN}$ is an approximate right unit of the convolution product.

Let $a \in C_\mathrm{c}^\infty(G_1)$. The convolution product with $e_n$ is given by
\begin{equation}
\label{eq:ApproxUnit02}
  a * e_n = \bar{m}_*(\bar{a}\bar{e}_n)
  \,,
\end{equation}
where the functions $\bar{a},\bar{e}_n \in C^\infty(\bar{G}_2)$ are given by
\begin{equation*}
  \bar{a}(g,h) = a(gh^{-1})
  \,,\qquad
  \bar{e}_n(g,h) = e_n(h)
  \,.
\end{equation*}
Since the support of $e_n$ is contained in $U$, the support of $\bar{e}_n$ is contained in
\begin{equation*}
  s^* U := G_1 \times_{G_0}^{s,s|_U} U
  \hookrightarrow
  \bar{G}_2
  \,.
\end{equation*}
It follows that the pointwise product $\bar{a}\bar{e}_n$ depends only on the restriction
\begin{equation*}
  \bar{f} := \bar{a}|_{s^* U}
\end{equation*}
of $\bar{a}$ to $(id_{G_1} \times i): G_1 \times_{G_0}^{s,s} U \to G_1 \times_{G_0}^{s,s} G_1$ the same domain, that is,
\begin{equation}
\label{eq:ApproxUnit03}
  \bar{a}\bar{e}_n = \bar{f} \bar{e}_n
  \,.
\end{equation}
Using the isomorphism
\begin{equation*}
  \bar{j} := \id \times_{G_0} j:
  G_1 \times_{G_0}^{s,s} U
  \xrightarrow{~\cong~}
  G_1 \times_{G_0}^{s,\pi} A
  \,,
\end{equation*}
we can identify the functions $\bar{f}$ and $\bar{e}_n$ with the functions 
\begin{equation*}
  f := \bar{f} \circ \bar{j}^{-1} 
  \,,\qquad
  \tilde{e}_n := \bar{e}_n \circ \bar{j}^{-1}
\end{equation*}
on $G_1 \times_{G_0} A$. The fiber integration of $\bar{f} \bar{e}_n$ can now be expressed as
\begin{equation}
\label{eq:ApproxUnit04}
  \bar{m}_*(\bar{f}\bar{e}_n)
  = \tilde{\pi}_* (f \tilde{e}_n)
  \,,
\end{equation}
where $\tilde{\pi} = \pr_1: G_1 \times_{G_0}^{s,\pi} A \to G_1$ is the projection of the pullback bundle. Putting Equations~\eqref{eq:ApproxUnit02}, \eqref{eq:ApproxUnit03}, and \eqref{eq:ApproxUnit04} together, we see that 
\begin{equation}
\label{eq:ApproxUnit05}
  a * e_n
  = \tilde{\pi}_* (f \tilde{e}_n)
  \,.
\end{equation}

Let $K := \Supp a$, which is compact by assumption. Since $\bar{f}(g,h) = f(gh^{-1})$, the support of $\bar{f}$ satisfies
\begin{equation*}
\begin{split}
  \bar{m} ( \Supp\bar{f})
  &\subset m(K \times_{G_0}^{s,t} U)
  \\
  &\subset
  m\bigl( K \times_{G_0}^{s,t} 
  \bigl(t^{-1} \bigl( s(K) \bigr)\cap U \bigr) \bigr)
  \,.   
\end{split}
\end{equation*}
Since $U$ is $t$-vertically compact, the right side is precompact, which implies that $\bar{m}(\Supp \bar{f}) = \tilde{\pi}(\Supp f)$ is precompact. We conclude that the map $a \mapsto a * e_n$ is given by the following composition of maps:
\begin{equation}
\label{eq:a_Star_e_nMap}
\begin{tikzcd}[column sep=large, row sep=2.7ex]
C_\mathrm{c}^\infty(G_1) \ar[d]
& 
a \ar[d,mapsto]
\\
C_{/G_1}^\infty(\bar{G}_2) \ar[d]
& 
\bar{a} \ar[d,mapsto]
\\
C_{/{G_1}}^\infty(s^*U) \ar[d,"\cong"',"(\bar{j}^{-1})^*" ]
& 
\bar{f} \mathrlap{{}:=\bar{a}|_{s^*U}} \ar[d,mapsto]
\\
C_{/{G_1}}^\infty(s^*A) \ar[d,"\tilde{\pi}_*\circ r_{\tilde{e}_n}"]
& f \mathrlap{{}:= \bar{f} \circ \bar{j}^{-1}} \ar[d]
\\
C_\mathrm{c}^\infty(G_1) 
& 
\tilde{\pi}_*(f\tilde{e}_n) \mathrlap{{} \stackrel{\eqref{eq:ApproxUnit05}}{=} a * e_n}
\end{tikzcd}
\end{equation}
By construction, $\tilde{e}_n = \bar{e}_n \circ \bar{j}^{-1} = e'_n \circ \pr_2$, so we can apply Corollary~\ref{cor:fiberdiracPullback} to the pullback of $A \to G_0$ by $s: G_1 \to G_0$. This shows that $\lim_{n\to \infty} \tilde{\pi}_* \circ r_{\tilde{e}_n} = 0_{s^* A}^*$ in $\intHom(C_{/G_1}^\infty(s^* A), C_\mathrm{c}^\infty(G_1))$. It follows from Proposition~\ref{prop:convergenceComposition} that the sequence of morphisms $a \mapsto a * e_n$ given by~\eqref{eq:a_Star_e_nMap} Mackey converges in $\intHom(C_\mathrm{c}^\infty(G_1), C_\mathrm{c}^\infty(G_1))$ to the morphism
\begin{equation*}
\begin{split}
  a &\longmapsto f\circ 0_{s^*\!A}
  \\
  &= \bar{f}\circ (\id_{G_1} \times_{G_0} j^{-1})\circ 0_{s^*\!A}
  \\
  &= \bar{f}\circ (\id_{G_1}, j^{-1}\circ 0_A\circ s)
  \\
  &= \bar{f}\circ (\id_{G_1}, 1\circ s)
  \\
  &= \bar{a}\circ (\id_{G_1}, 1\circ s)
  \\
  &= a
  \, ,
\end{split}
\end{equation*}
where in the last step we have used that $\bar{a}(g, 1_{s(g)}) = a(g1_{s(g)}^{-1}) = a(g)$. We conclude that $(e_n)_{n \in \bbN}$ is a strong approximate right unit in the sense of Definition~\ref{def:ApproxUnit}. 

The sequence $(\tilde{e}_n)_{n \in \bbN}$ in $A(G)$ defined by $\tilde{e}_n(g) := e_n(g^{-1})$ is a strong approximate left unit.
\end{proof}

\subsection{Proof of smoothness of convolution modules}

\begin{Proposition}
\label{prop:ConvBimodSmooth}
Let $P$ be a $G$-$H$ groupoid bibundle, not necessarily right principal. Then the completed bornological convolution bimodule $M(P)$ is smooth.  
\end{Proposition}

\begin{proof}
The completed convolution action is given in~\eqref{eq:ConvModSeq} by the pullback $j^*$ by the inclusion $j:P \times_{H_0} H_1 \to P \times H_1$ followed by the fiber integration $\beta_*$ of the $H$-action. It follows from Lemma~\ref{lem:PullbackSubfmld} and Proposition~\ref{prop:TestSubmfgPullback}, that $j^*$ is a split epimorphism. In Lemma~\ref{lem:ActionSubmersion}, we have shown that $\beta$ is a surjective submersion. It follows from Proposition~\ref{prop:TestFiberIntegr} that $\beta_*$ is a regular epimorphism. A fortiori, $j^*$ and $\beta_*$ are strong epimorphisms, so their composition is a strong epimorphism. Since $A(G)$ has a strong approximate right unit by Theorem~\ref{thm:ApproxUnitExists}, it follows from Proposition~\ref{prop:ApproxUnitSmoothMod} that $M(P)$ is smooth as a right $A(H)$-module. The proof that $M(P)$ is smooth as a left $A(G)$-module is analogous.
\end{proof}

\begin{Remark}
\sloppypar
Proposition~\ref{prop:ConvBimodSmooth} implies that the completed convolution action $A(G) \Cotimes M(P) \to M(P)$ is surjective. By \cite[Theorem~12]{Francis2022}, any $m\in M(P)$ may be written as a \emph{finite} sum $m = a_1\cdot m_1 + \dots + a_n\cdot m_n$, where $a_i\in A(G)$ and $m_i\in M(P)$. This proves that, already before completion, the convolution action on the \emph{algebraic} tensor product $A(G)\otimes M(P) \to M(P)$ is surjective. 
\end{Remark}

\begin{Corollary}
The complete bornological convolution algebra $A(G)$ of a Lie groupoid $G$ is self-induced.
\end{Corollary}
\begin{proof}
Apply Proposition~\ref{prop:ConvBimodSmooth} to the $G$-$G$ bibundle $P = G_1$.
\end{proof}

\subsection{Proof of functoriality}
\label{sec:functoriality}

The completed functoriality constraint $\hat{\tau}$ is obtained by applying the separation and completion functors to the serially commutative Diagram~\eqref{eq:taudiagram} in $\Born$, which yields the corresponding serially commutative diagram in $\cBorn$:
\begin{equation}
\label{eq:tauhatdiagram}
\begin{tikzcd}[column sep=3em]
M(P) \Cotimes
A(H) \Cotimes
M(Q)
\ar[r, shift left=1.2, "{\hat{\rho} \Cotimes  \id}"]
\ar[r, shift left=-1.2, "{\id \Cotimes \hat{\lambda}}"']
\ar[d, "\cong"]
&
M(P) \Cotimes 
M(Q)
\ar[r, "\Pi"]
\ar[d, "\cong"]
&
M(P) \Cotimes_{A(H)} M(Q)
\ar[dd, "\hat{\tau}_{P,Q}"]
\\
C_\mathrm{c}^\infty(P \times H_1 \times Q)
\ar[d, "i^*"]
&
C_\mathrm{c}^\infty(P \times Q)
\ar[d, "j^*"]
&
\\
C_\mathrm{c}^\infty(
P \times_{H_0} H_1 \times_{H_0} Q
)
\ar[r, shift left=1.2, "{(\beta \times \id_Q)_*}"]
\ar[r, shift left=-1.2, "{(\id_P \times \alpha)_*}"']
&
C_\mathrm{c}^\infty(
P \times_{H_0} Q
)
\ar[r, "\pi_*"']
&
C_\mathrm{c}^\infty(
P \circ_H Q
)
\,.
\end{tikzcd}
\end{equation}
The functoriality constraint is given explicitly by
\begin{equation}
\label{eq:tauhatformula}
    \hat{\tau}_{P,Q}(f\Cotimes_{A} g)([p,q])
    = 
    \int_{\mathclap{s^{-1}(l(q))}} f(p\cdot h^{-1})\, g(h \cdot q) 
    \,\di \lambda_{l(q)}(h) \,,
\end{equation}
for all $f\Cotimes_{A} g \in M(P) \Cotimes_{A(H)} M(Q)$ and $[p,q] \in P \circ_H Q$. Since both the separation and the completion functor are left adjoints, they preserve the coequalizers of the rows. Lemma~\ref{lem:PullbackSubfmld} and Proposition~\ref{prop:TestSubmfgPullback} show that $j^*$ is a split epimorphism; Proposition~\ref{prop:TestFiberIntegr} shows that $\pi_*$ is a regular epimorphism. It follows that $\hat{\tau}_{P,Q}$ is a strong epimorphism. To prove that $\hat{\tau}_{P,Q}$ is an isomorphism, we have to show that it is a monomorphism in the category $\cBorn$. This is equivalent to the underlying linear map being injective. 

To prove injectivity, we can work in the underlying category of vector spaces. For ease of notation, we identify
\begin{align*}
  M(P) \Cotimes A(H) \Cotimes M(Q)
  &=
  C_\mathrm{c}^\infty(P \times H_1 \times Q)
  \\
  M(P) \Cotimes M(Q)
  &=
  C_\mathrm{c}^\infty(P \times Q)
  \,.
\end{align*}
By using this identification and replacing the coequalizers in Diagram~\eqref{eq:tauhatdiagram} with the cokernels of the difference maps
\begin{align*}
  \phi 
  &= \hat{\rho} \Cotimes  \id - \id \Cotimes \hat{\lambda}
  \\
  \psi
  &= (\beta \times \id_Q)_* - (\id_P \times \alpha)_*
  \,,
\end{align*}
we obtain the commutative diagram
\begin{equation}
\label{eq:tauhatcoker}
\begin{tikzcd}
C_\mathrm{c}^\infty(P \times H_1 \times Q)
\ar[r, "\phi"]
\ar[d, "i^*"]
&
C_\mathrm{c}^\infty(P \times Q)
\ar[d, "j^*"]
\ar[r, "\Pi = \coker \phi"]
&[1em]
M(P) \Cotimes_{A(H)} M(Q)
\ar[d, "\hat{\tau}_{P,Q}"]
\\
C_\mathrm{c}^\infty(P \times_{H_0} H_1 \times_{H_0} Q)
\ar[r, "\psi"']
&
C_\mathrm{c}^\infty(P \times_{H_0} Q)
\ar[r, "\pi_* = \coker\psi"']
&
C_\mathrm{c}^\infty(P \circ_H Q)
\end{tikzcd}
\end{equation}
where each row is right exact.

\begin{Lemma}
\label{lem:exactnessKernel}
$\ker j^* \subset \mathrm{im}\, \phi$ in Diagram~\eqref{eq:tauhatcoker}.
\end{Lemma}
\begin{proof}
The relevant maps are given explicitly by the formulas
\begin{align*}
  (\pi_* \circ j^*)(f)([p,q])
  &=
  \int_{\mathclap{s^{-1}(l(p))}} f(p\cdot h^{-1}, h\cdot q) 
    \,\di \lambda_{l(p)}(h) 
  \\
  (\hat{\rho} \Cotimes \id)(\tilde{f})(p,q) 
  &=
  \int_{\mathclap{s^{-1}(r(p))}} \tilde{f}(p\cdot h^{-1}, h, q) 
    \,\di \lambda_{r(p)}(h)
  \\
  (\id \Cotimes \hat{\lambda})(\tilde{f})(p,q)
  &= 
  \int_{\mathclap{s^{-1}(l(q))}} \tilde{f}(p, h^{-1}, h \cdot q) 
    \,\di \lambda_{l(q)}(h)
  \,,
\end{align*}
for all $f \in C_\mathrm{c}^\infty(P \times Q)$, $\tilde{f} \in C_\mathrm{c}^\infty(P \times H_1 \times Q)$, $[p,q] \in P \circ_H Q$, and $(p,q) \in P \times Q$.

Let $f \in \ker j^* \subset C_\mathrm{c}^\infty(P \times Q)$. In the first step, we will construct an auxiliary function on $P$. For every $g \in C_\mathrm{c}^\infty(P)$ we define the function $l_* g: P \to \bbR$ by
\begin{equation*}
  (l_* g)(p) 
  := 
  \int_{\mathclap{s^{-1}(r(p))}} g(p \cdot h^{-1}) \,\di \lambda_{r(p)}(h)
  \,.
\end{equation*}
It follows from the right invariance of the Haar system that $(l_* g)(p \cdot g) = (l_* g)(p)$ for all composable $p$ and $g$. In other words, $l_* g$ descends to a smooth function on the orbit space $P/G \cong H_0$, which justifies the notation. Since $l$ is a surjective submersion, we can find a function $e_1 \in C_\mathrm{c}^\infty(P)$, such that
\begin{equation*}
  l_* e_1 \bigr|_{\pr_{\!P}(\Supp f)} = 1
  \,.
\end{equation*}
(We can first choose such a function locally on a tubular neighborhood of a local section of $l$, then add the local functions with a partition of unity on $\pr_P(\Supp f) \subset H_0$.)

On $P \times_{H_0} H_1 \times Q$ we have the smooth function
\begin{equation*}
  f_1(p,h,q) := e_1(p)\,f(p \cdot h, q) 
  \,.
\end{equation*}
It is the pullback of the compactly supported function $P \times_{H_0} P \times Q \to \bbR$, $(p,p',q) \mapsto e_1(p)\, f(p,q)$ along the map
\begin{equation*}
  P \times_{H_0} H_1 \times Q 
  \xrightarrow{(\pr_1, \beta) \times \id_Q}
  P \times_{H_0} P \times Q
  \,,
\end{equation*}
which is an isomorphism since the right $H$-action $\beta$ is principal. It follows that $f_1$ is compactly supported. By assumption, $f$ is in the kernel of $j^*$, so it vanishes on $P \times_{H_0} Q$. This implies that $f_1$ vanishes on $P \times_{H_0} H_1\times_{H_0} Q$.

Since $P \times_{H_0} H_1 \times Q \hookrightarrow P \times H_1 \times Q$ is a closed embedding, we can extend $f_1$ to a function $\tilde{f}_1 \in C_\mathrm{c}^\infty(P \times H_1 \times Q)$. This function satisfies
\begin{equation}
\label{eq:exactness01}
\begin{split}
  (\hat{\rho} \Cotimes \id)(\tilde{f})(p,q) 
  &=
  \int_{\mathclap{s^{-1}(r(p))}} 
  e_1(p\cdot h^{-1})\, f\bigl((p \cdot h) \cdot h^{-1}, q \bigr) 
    \,\di \lambda_{r(p)}(h)
  \\
  &=
  \int_{\mathclap{s^{-1}(r(p))}} 
  e_1(p\cdot h^{-1})\, f(p, q) \,\di \lambda_{r(p)}(h)
  \\
  &= (l_*e_1)(p)\, f(p,q)
  \\
  &= f(p,q)
  \,.
\end{split}    
\end{equation}
Let us denote
\begin{equation}
\label{eq:exactness02}
  f' := \bigl( \id \Cotimes \hat{\lambda} \bigr)\tilde{f}
  \,,
\end{equation}
which is a compactly supported smooth function on $P \times Q$. Since $\tilde{f}_1$ vanishes on $P \times_{H_0} H_1\times_{H_0} Q$, $f'$ vanishes on $P \times_{H_0} Q$.

We will now construct a second auxiliary function on $H_1$. For every $g \in C_\mathrm{c}^\infty(P)$ we define the smooth function $s_* g: H_0 \to \bbR$ by
\begin{equation*}
  (s_* g)(x) 
  := 
  \int_{\mathclap{s^{-1}(x)}} g(h) \,\di \lambda_x(h)
  \,.
\end{equation*}
Since $\Supp f'$ is compact, $(l \circ \pr_Q)(\Supp f') \subset H_0$ is compact. Since $s$ is a surjective submersion, we can find a function $e_2 \in C_\mathrm{c}^\infty(H_1)$, such that
\begin{equation*}
  s_* e_2 \bigr|_{(l \circ \pr_Q)(\Supp f')} = 1
  \,.
\end{equation*}

On $P \times H_1 \times_{H_0} Q$ we have the smooth function
\begin{equation*}
  f_2(p,h,q) := f'(p, h \cdot q) \, e_2(h) 
  \,.
\end{equation*}
It is the pullback of the compactly supported function $P \times H_1 \times_{H_0} Q \to \bbR$, $(p,h,q) \mapsto f'(p,q) \, e_2(h)$ along the map
\begin{equation*}
  P \times H_1\times_{H_0} Q 
  \xrightarrow{ \id_P \times (\pr_1, \alpha)}
  P \times H_1\times_{H_0} Q
  \,,
\end{equation*}
which is an isomorphism since $\alpha$ is a groupoid action. It follows that $f_2$ is compactly supported.

Since $f'$ vanishes on $P \times_{H_0} Q$, $f_2$ vanishes on $P \times_{H_0} H_1\times_{H_0} Q$. The submanifolds $P \times_{H_0} H_1\times Q$ and $P \times H_1\times_{H_0} Q$ of $P \times H_1\times Q$ intersect transversely in $P \times_{H_0} H_1\times_{H_0} Q$. Therefore, we can extend $f_2$ to a smooth function $\tilde{f}_2 \in C_\mathrm{c}^\infty(P \times H_1 \times Q)$ that vanishes on $P \times_{H_0} H_1\times Q$. 

Since $(\hat{\rho} \Cotimes \id)(\tilde{f}_2)$ depends only on the restriction $\tilde{f}_2|_{P\times_{H_0} H_1 \times Q}$, it follows that
\begin{equation}
\label{eq:exactness03}
  (\hat{\rho} \Cotimes \id)(\tilde{f}_2) (p,q) = 0
  \,.
\end{equation}
Moreover,
\begin{equation}
\label{eq:exactness04}
\begin{split}
  (\id \Cotimes \hat{\lambda})(\tilde{f}_2)(p,q)
  &= 
  \int_{\mathclap{s^{-1}(l(q))}} f_2(p, h^{-1}, h \cdot q) 
    \,\di \lambda_{l(q)}(h)
  \\
  &= 
  \int_{\mathclap{s^{-1}(l(q))}} 
    f'\bigl(p, h^{-1} \cdot (h \cdot q) \bigr) \, e_2(h) 
    \,\di \lambda_{l(q)}(h)
  \\
  &= 
  \int_{\mathclap{s^{-1}(l(q))}} 
    f'\bigl(p, q) \, (s_* e_2)\bigl(l(q)\bigr) 
    \,\di \lambda_{l(q)}(h)
  \\
  &= 
  f'(p,q)
  \,.
\end{split}
\end{equation}
Let $\tilde{f} = \tilde{f}_1 - \tilde{f}_2$. Using Equations~\eqref{eq:exactness01}, \eqref{eq:exactness02}, \eqref{eq:exactness03}, and \eqref{eq:exactness04}, we obtain
\begin{equation*}
\begin{split}
  \phi(\tilde{f})
  &= 
    (\hat{\rho} \Cotimes \id)(\tilde{f}_1) 
  - (\id \Cotimes \hat{\lambda})(\tilde{f}_1)
  - (\hat{\rho} \Cotimes \id)(\tilde{f}_2) 
  + (\id \Cotimes \hat{\lambda})(\tilde{f}_2)
  \\
  &=
  f - f' - 0 + f'
  \\
  &= f
  \,,
\end{split}
\end{equation*}
which shows that $f$ is in the image of $\phi$.
\end{proof}

Consider Diagram~\eqref{eq:tauhatcoker}. It follows from Lemma~\ref{lem:PullbackSubfmld} and Proposition~\ref{prop:TestSubmfgPullback} that $i^*$ is surjective. The four lemma states $\ker(\hat{\tau}_{P,Q}) = \Pi (\ker j^*)$. By Lemma~\ref{lem:exactnessKernel}, we have $\Pi(\ker j^*) \subset \Pi(\im \phi) = \{0\}$. This shows that $\hat{\tau}_{P,Q}$ is injective, so a monomorphism. We have already shown that $\hat{\tau}_{P,Q}$ is a strong epimorphism. We conclude that $\hat{\tau}_{P,Q}$ is an isomorphism. \qed

\subsection{Proof of coherence}
\label{sec:coherence}

Let us recall the data defining our $2$-functor.
The convolution functor assigns to a Lie groupoid $G$ the bornological convolution algebra $A(G)$.
The category of 1-morphisms $\GrpdMrt(G,H)$ consists of right principal $G$-$H$-bibundles and their biequivariant diffeomorphisms.
The assignments $P\mapsto M(P)$, $\phi\mapsto \phi_* = (\phi^{-1})^*$ defines a functor $\GrpdMrt(G,H)\to \AlgMrt(A(G),A(H))$, because the inverse and the pullback are functorial.

The functoriality constraint is the collection of morphisms $\hat{\tau}_{P,Q}$ of Diagram~\eqref{eq:tauhatdiagram} and Proposition~\ref{prop:CompletedAlgsMods}~(iv), where we have already proved its naturality in $(P,Q)$. In Section~\ref{sec:functoriality}, we have proved that it is an isomorphism.

The identity morphism in $\GrpdMrt(G,G)$ is $G$ viewed as $G$-$G$ bibundle. The identity morphism in $\AlgMrt(A(G),A(G))$ is $M(G)$, which is the algebra $A(G)$ viewed as $A(G)$-$A(G)$ bimodule. The unitality constraint $\eta_G: M(G) \to M(G)$ is simply the identity morphism.

We have to show that $\hat{\tau}$ is compatible with the associator of the composition of groupoid bibundles and the associator of composition of bimodules, such that Diagram~\eqref{diag:Coherence1} commutes. Using the subscript notation $A_H \equiv A(H)$, $M_P \equiv M(P)$, etc.\ to conserve space, the diagram has the form
\begin{equation}
\label{eq:tauCoherence1}
\begin{tikzcd}[column sep=3.4em]
   (M_P\Cotimes_{A_H} M_Q) \Cotimes_{A_K} M_R \ar[r,"\hat{\tau}_{P,Q} \Cotimes \id"]\ar[d,"\cong"']
   & 
   M_{P\circ_H Q} \Cotimes_{A(K)} M_R \ar[r,"\hat{\tau}_{P \circ_H Q, R}"] 
   & 
   M_{(P\circ_H Q)\circ_K R} \ar[d,"\cong"]
   \\
   M_P\Cotimes_{A_H} (M_Q \Cotimes_{A_K} M_R)\ar[r,"\id \Cotimes \hat{\tau}_{Q,R}"']
   &
   M_P \Cotimes_{A_H} M_{Q\circ_K R} \ar[r,"\hat{\tau}_{P, Q \circ_K R}"'] 
   & 
   M_{P\circ_H (Q \circ_K R)}
\end{tikzcd}
\end{equation}
where $P$ is a $G$-$H$ groupoid bibundle, $Q$ an $H$-$K$ bibundle, and $R$ a $K$-$L$ bibundle. 

The commutativity of this diagram can be proved by a direct computation using Equation~\eqref{eq:tauhatformula}. We will give a more conceptual argument using the defining Diagram~\eqref{eq:tauhatdiagram}. The right commutative square of~\eqref{eq:tauhatdiagram} shows that
\begin{equation}
\label{eq:tauCoherence2}
\begin{tikzcd}
M_P \Cotimes M_Q
\ar[dr, "\pi_* \circ j^* = T_{P,Q}"]
\ar[d, "\Pi"']
&
\\
M_P \Cotimes_{A_H} M_Q
\ar[r, "\hat{\tau}_{P,Q}"']
&
M_{P \circ_H Q}
\end{tikzcd}
\end{equation}
commutes. $\hat{\tau}_{P,Q}$ is uniquely determined by the $A_H$-tensorial map $\pi_* \circ j^*$. Consider the following commutative diagram:
\begin{equation}
\label{eq:tauCoherence3}
\begin{tikzcd}[column sep=4em, row sep=5ex]
M_P\Cotimes M_Q \Cotimes M_R 
\ar[r, "\id \Cotimes j_{Q,R}^*"]
\ar[d, "j_{P,Q}^* \Cotimes \id"']
&
M_P\Cotimes M_{Q\times_{K_0} R} 
\ar[r,"\id \Cotimes (\pi_{Q,R})_*"]
\ar[d, "j_{P,Q\times_{K_0} R}^*"] 
&
M_P\Cotimes M_{Q\circ_K R} 
\ar[d,"j_{P, Q\circ_K R}^*"] 
\\
M_{P\times_{H_0} Q}\Cotimes M_R 
\ar[r, "j_{P \times_{H_0} Q, R}^*"]
\ar[d,"(\pi_{P,Q})_*\Cotimes \id"']
& 
M_{P\times_{H_0}Q\times_{K_0}R} 
\ar[r, "(\id_P \times \pi_{Q,R})_*"]
\ar[d, "(\pi_{P,Q} \times \id_R)_*"]
& 
M_{P\times_{H_0} (Q\circ_K R)} 
\ar[d, "(\pi_{P, Q\circ_K R})_*"]
\\
M_{P\circ_H Q} \Cotimes M_R 
\ar[r, "j_{P \circ_H Q, R}^*"'] 
& 
M_{(P\circ_H Q)\times_{K_0} R} 
\ar[r, "{(\pi_{P \circ_H Q, R})_*}"']
& 
M_{P\circ_{H} Q \circ_K R}
\end{tikzcd}
\end{equation}
The outer square of~\eqref{eq:tauCoherence3} is the inner square in the following diagram: 
\begin{equation}
\label{eq:tauCoherence4}
\begin{tikzcd}[column sep=2ex, row sep=6ex]
&
M_P\Cotimes (M_Q \Cotimes M_R) 
\ar[dr, "\id \Cotimes T_{Q,R}"]
\ar[d, "\cong"] &&
\\
(M_P\Cotimes M_Q) \Cotimes M_R
\ar[ur, "\cong"]
\ar[dr, "T_{P,Q} \Cotimes \id"']
\ar[r, "\cong"]
&
M_P\Cotimes M_Q \Cotimes M_R 
\ar[r] \ar[d]
&  
M_P\Cotimes M_{Q\circ_K R} 
\ar[dr,"T_{P, Q\circ_K R}"]
\ar[d]
&
\\
&
M_{P\circ_H Q} \Cotimes M_R 
\ar[dr, "T_{P \circ_H Q, R}"'] 
\ar[r] 
&
M_{P\circ_{H} Q \circ_K R}
\ar[d, "\cong"] \ar[r, "\cong"]
&
M_{P\circ_{H} (Q \circ_K R)}
\\
&&
M_{(P\circ_{H} Q) \circ_K R}
\ar[ur, "\cong"']
&
\end{tikzcd}
\end{equation}
The outer parallelogram has the form of Diagram~\eqref{eq:tauCoherence1} with $\hat{\tau}$ replaced by $T$. In the same way as $\Pi$ projects $T_{P,Q}$ to $\hat{\tau}_{P,Q}$ as in Diagram~\eqref{eq:tauCoherence2}, $\Pi$ induces an epimorphism from the commutative Diagram~\eqref{eq:tauCoherence4} to the coherence Diagram~\eqref{eq:tauCoherence1}, which therefore commutes. This concludes the proof of $2$-functoriality. \qed

\subsection{Proof of projectivity of convolution modules}

We will now prove Proposition~\ref{prop:ConvModProjective}. Let $H$ be a Lie groupoid; let $X \stackrel{l}{\leftarrow }P \stackrel{r}{\rightarrow} H_0$ be submersions; let $\beta: P \times_{H_0} H_1 \to P$ be a right groupoid action that is proper and transitive on the $l$-fibers. We have seen in Equation~\eqref{eq:ConvModSeq}, that the completed convolution action is given as
\begin{equation*}
\begin{tikzcd}[column sep=2em]
M(P) \Cotimes A(H)=C_\mathrm{c}^\infty(P\times H_1) \ar[r,"i^*"]
& 
C_\mathrm{c}^\infty(P\times_{H_0} H_1) 
\ar[r,"\beta_*"]
&  
C_\mathrm{c}^\infty(P)=M(P) 
\end{tikzcd}
\end{equation*}
by a pullback followed by a pushforward. We will show that each of these maps admits a bounded right $A(H)$-linear section.

In the first step, we will construct an $A(H)$-linear section of $i^*$. By Corollary~\ref{cor:PullbackRetract}, we can find an open cover $\{U_i\}$ of $P$ together with maps $\zeta_i: U_i \times V_i \to U_i$, $V_i := r(U_i)$, satisfying $r\bigl(\zeta_i(p, y)\bigr) = y$ and $\zeta_i\bigl(p, r(p) \bigr) = p$. We can assume without loss of generality that the $U_i$ are precompact and, since we assume manifolds to be second countable, that the cover $\{U_i\}$ is locally finite. 
Let $\chi_i: U_i \to [0,1]$ be a subordinate partition of unity.

By Lemma~\ref{lem:PullbackSubfmld}, $P \times_{H_0} H_1 \subset P \times H_1$ is a closed embedded submanifold. Therefore, a subset of $P \times_{H_0} H_1$ is compact if and only if it is contained in the product $K_P \times K_{H_1}$ of compact subsets $K_P \subset P$ and $K_{H_1} \subset H_1$. Let $f \in C_\mathrm{c}^\infty(P \times_{H_0} H_1)$ with support contained in $K_P \times K_{H_1}$. Consider the smooth functions on $P \times H_1$ defined by
\begin{equation*}
  f^\zeta_i(p,h) 
  := \chi_i(p)\, f\bigl(\zeta_i(p,t(h)), h \bigr)
  \,.    
\end{equation*}
A necessary condition for $(p,h)$ to be in the support of $f^\zeta_i$ is that $\zeta_i(p, t(h)) \subset K_P$, since $\zeta_i$ takes values in $U_i$, the support can be non-empty only if $U_i \cap K_P$ is non-empty. Since the cover $\{U_i\}$ is locally finite and since $K_P$ is compact, the set of indices $I := \{i~|~ U_i \cap K_P \neq \emptyset\}$ is finite.

The support of $f^\zeta_i$ is contained in the compact set $\Supp \chi_i \times K_{H_1}$. Since $I$ is finite, we have a linear map
\begin{equation*}
\begin{aligned}
  \sigma_1:
  C_\mathrm{c}^\infty(P \times_{H_0} H_1)
  &\longrightarrow
  C_\mathrm{c}^\infty(P \times H_1)
  \\
  f &\longmapsto \sum_i f^\zeta_i
  \,.
\end{aligned}
\end{equation*}
Let $(p,h) \in P \times_{H_0} H_1$, that is, $r(p) = t(h)$. Then $\zeta_i(p, t(h)) = p$. It follows that
\begin{equation*}
  \sigma_1(f)(p,h)
  = \sum_i \chi_i(p)\, f(p, h)
  = f(p,h)
  \,,
\end{equation*}
which shows that $\sigma_1$ is a section of $i^*$. To show that $\sigma_1$ is bounded, we can use arguments analogous to the proof of Lemma~\ref{lem:fiberdirac}. Let $a \in A(H_1) = C_\mathrm{c}^\infty(H_1)$. Then
\begin{equation*}
\begin{split}
  \sigma_1(f \cdot a)(p,h)
  &= \sum_i \chi_i(p) (f \cdot a)\bigl(\zeta_i(p,t(h)), h)
  \\
  &= \sum_i \chi_i(p) \int_{\mathclap{s^{-1}(s(h))}}
  f\bigl( \zeta_i(p,t(hk^{-1})), hk^{-1}\bigr) \, a(k)
  \, \di\lambda_{s(k)}(k)
  \\
  &= \int_{\mathclap{s^{-1}(s(h))}}
  \sum\nolimits_i \chi_i(p)\,
  f\bigl( \zeta_i(p,t(hk^{-1})), hk^{-1}\bigr) \, a(k)
  \, \di\lambda_{s(k)}(k)
  \\
  &= \bigl(\sigma_1(f) \cdot a \bigr) (p,h) \,.
\end{split}    
\end{equation*}
We conclude that $\sigma_1$ is right $A(H)$-linear bounded section of $i^*$.

In the second step, we will construct an $A(H)$-linear section of $\beta_*$. Since $l: P \to G_0$ is a surjective submersion, we can find a smooth function $e: P \to \bbR$ that satisfies
\begin{equation}
\label{eq:ProjConv01}
  \int_{\mathclap{s^{-1}(r(p))}} e(p \cdot h^{-1}) \,\di \lambda_{r(p)}(h)
  = 
  1
\end{equation}
and is compactly supported in the $l$-direction, that is, for $K \subset G_0$ compact, $\Supp e \cap l^{-1}(K)$ is compact. (See the proof of Lemma~\ref{lem:exactnessKernel}.) 

Let $f \in C_\mathrm{c}^\infty(P)$ with support $\Supp f = K$. Consider the smooth function $\tilde{f}: P \times_{G_0}^{l,l} P \to \bbR$, defined by $\tilde{f}(p,p') := e(p) f(p')$. Its support satisfies
\begin{equation*}
  \Supp \tilde{f} \subset 
  \bigl( \Supp e \cap l^{-1}(l(K)) \bigr) \times K
  \,.
\end{equation*}
Since $K$ is compact, $l(K)$ is compact. Since $e$ is compact in the $l$-direction, the first summand is compact. It follows that $\Supp \tilde{f}$ is a closed subset of a compact set, so it is compact. The pull back of $\tilde{f}$ along the characteristic map~\eqref{eq:ActionCharMap} of the action will be denoted 
\begin{equation*}
\begin{aligned}
  f^e: P \times_{H_0} H_1 
  &\longrightarrow \bbR
  \\
  f^e(p,h) 
  &:= e(p)\, f(p \cdot h)
  \,.
\end{aligned}
\end{equation*}
By assumption on the action, \eqref{eq:ActionCharMap} is a proper map, so $f^e$ is compactly supported. We have
\begin{equation*}
\begin{split}
  (\beta_* f^e)(p)
  &= \int_{\mathclap{s^{-1}(r(p))}} 
    e(p \cdot h^{-1})\, f\bigl((p \cdot h^{-1}) \cdot h \bigr)
    \, \di \lambda_{r(p)}(h)
  \\
  &= \int_{\mathclap{s^{-1}(r(p))}} 
    e(p \cdot h^{-1})\, f(p) \, \di \lambda_{r(p)}(h)
  \\
  &= f(p)
  \,,
\end{split}  
\end{equation*}
where in the last step we have used~\eqref{eq:ProjConv01}. This shows that the linear map
\begin{equation*}
\begin{aligned}
  \sigma_2:
  C_\mathrm{c}^\infty(P)
  &\longrightarrow
  C_\mathrm{c}^\infty(P \times_{H_0} H_1)
  \\
  f &\longmapsto f^e
\end{aligned}
\end{equation*}
is a section of $\beta_*$. It follows from Proposition~\eqref{prop:TestFuncPullback} that $\sigma_2$ is bounded. Let $a \in A(H_1) = C_\mathrm{c}^\infty(H_1)$. Then
\begin{equation*}
\begin{split}
  \sigma_2(f \cdot a)(p,h)
  &= e(p)\, (f \cdot a)(p \cdot h)
  \\
  &= e(p) \int_{\mathclap{s^{-1}(r(p\cdot h))}}
  f\bigl((p \cdot h) \cdot k^{-1})\, a(k)\, \di\lambda_{r(p \cdot h)}(k)
  \\
  &= \int_{\mathclap{s^{-1}(s(h))}}
     e(p)\, f(p \cdot hk^{-1})\,
     a(k) \, \di\lambda_{s(h)}(k)
  \\
  &= \bigl(\sigma_2(f) \cdot a \bigr) (p,h) \,.
\end{split}    
\end{equation*}
It follows that $\sigma_2$ is an $A(H)$-linear bounded section of $\beta_*$. It follows from Proposition~\ref{prop:ModuleSectionProj} that the right $A(H)$-module $M(P)$ is projective, which concludes the proof of Proposition~\ref{prop:ConvModProjective}. \qed

\section{Examples and applications}
\label{sec:Applications}

\subsection{Basic examples}

The differentiable stack presented by the Lie groupoid $G_1 \rightrightarrows G_0$ is often denoted suggestively by $G_0 /\!/G_1$. Therefore, we also denote the convolution algebra by $A(G) = A(G_0/\!/G_1)$.

\begin{Example}[Terminal groupoid]
The convolution algebra of the terminal groupoid $1 = (* \rightrightarrows *)$ is $A(1) = \bbC$.
\end{Example}

\begin{Example}[Manifold]
Let $X$ be a smooth manifold viewed as Lie groupoid $X \rightrightarrows X$ with only identities as arrows. The convolution algebra is $C_\mathrm{c}^\infty(X)$ with the extension of pointwise multiplication (Example~\ref{ex:PointwiseMult}) to $C_\mathrm{c}^\infty(X) \Cotimes C_\mathrm{c}^\infty(X) \cong C_\mathrm{c}^\infty(X \times X)$. The algebra is unital if and only if $X$ is compact.
\end{Example}

\begin{Example}[Group as point stack]
\label{ex:GroupConv}
Let $G$ be a Lie group viewed as groupoid $G \rightrightarrows *$ with one object. Using the Haar measure on $G$, we obtain the convolution algebra $A(*/\!/G) = C_\mathrm{c}^\infty(G)$ with group convolution as product. The algebra is unital if and only if $G$ is discrete.
\end{Example}

\begin{Example}[Group bundle]
Example~\ref{ex:GroupConv} generalizes to group bundles, that is, a groupoid $G_1 \rightrightarrows G_0$ with equal source and target maps, $s=t$. Every fiber is a Lie group, the Haar system consists of a smooth family of Haar measures on the fibers, and the product is the fiberwise group convolution.
\end{Example}

\begin{Example}[Pair groupoid]
\label{ex:PairGrpd}
Let $X$ be a smooth manifold with a density $\mu$. The convolution algebra of the pair groupoid $X\times X \rightrightarrows X$ is $\calK(X) := C_\mathrm{c}^\infty(X\times X)$ with the convolution product of integral kernels,
\begin{equation*}
  (K_1 * K_2)(x,y) = \int_X K_1(x,z)\, K_2(z,y) \,\di\mu(z) 
  \,.
\end{equation*}
In other words, $\calK(X)$ is the algebra of linear integral operators on $C_\mathrm{c}^\infty(X)$. The bibundle $X \leftarrow X \rightarrow *$ with the left action of the pair groupoid on its base and right trivial action by the terminal groupoid $1$ is a Morita equivalence of Lie groupoids. Its convolution is the $\calK(X)$-$\bbC$ bimodule $C_\mathrm{c}^\infty(X)$ with the left action by integral operators, which is a Morita equivalence
\begin{equation*}
  \calK(X) \simeq \bbC 
  \,.   
\end{equation*}
\end{Example}

\begin{Example}[Product of groupoids]
\label{ex:ProdGrpd}
Let $G \times H$ be the product of Lie groupoids. Then 
\begin{equation*}
  A(G \times H) \cong A(G) \Cotimes A(H)
  \,,
\end{equation*}
where the right side denotes the tensor product of complete bornological algebras.
\end{Example}

\begin{Example}[Matrix algebra]
Let $X = \{1, \ldots, n\}$. Let $H$ be the product of the pair groupoid of $X$ and some Lie groupoid $G$. Then
\begin{equation*}
\begin{split}
  A(H) 
  &\cong \calK(X) \Cotimes A(G)
  \cong \Mat(n\times n, \bbC) \Cotimes A(G)
  \\
  &\cong \Mat\bigl( n \times n, A(G) \bigr)    
\end{split}
\end{equation*}
Since the pair groupoid is Morita equivalent to the terminal groupoid $1$, as explained in Example~\ref{ex:PairGrpd}, $H$ is Morita equivalent to  $G$. By applying the convolution functor, we recover the usual Morita equivalence
\begin{equation*}
  \Mat\bigl( n \times n, A(G) \bigr)
  \simeq A(G)
  \,.
\end{equation*}
\end{Example}

\begin{Example}[Group action]
\label{ex:GroupAction}
Let $G \times X \to X$, $(g,x) \mapsto g \cdot x$ the action of a Lie group on a manifold $X$; choose a Haar measure $\mu$ on $G$. The convolution algebra of the action groupoid $G\ltimes X \rightrightarrows X$ is $A(X/\!/G) \equiv A(G\ltimes X) = C_\mathrm{c}^\infty(G\times X)$ with the product
\begin{equation*}
  (a_1*a_2)(g,x)
  =
  \int_G a_1(gh^{-1}, h \cdot x)\, a_2(h,x) \,\di\mu(h) 
  \,.
\end{equation*}
\end{Example}

\begin{Example}[Principal bundle]
\label{ex:PrincBund}
When the $G$-action on $X$ is free and proper, the quotient $X/G$ exists in $\Mfld$ and $\pi: X \to X/G$ is a principal $G$-bundle. The space of smooth functions on $X/G$ can be identified with the space $C^\infty(X)^G$ of $G$-invariant smooth functions on $X$. It follows that
\begin{equation*}
  C_\mathrm{c}^\infty(X/G)
  \cong \{ f \in C^\infty(X)^G ~|~ \pi(\Supp f) \subset X/G \text{ compact} \}
  =: C_{\mathrm{bc}}^\infty(X)^G
  \,,
\end{equation*}
where ``bc'' stands for ``basically compact''. The bibundle $X \leftarrow X \rightarrow X/G$ with left $G \ltimes X$-action on its base and right trivial action of $X/G \rightrightarrows X/G$ is a Morita equivalence. Its convolution bimodule is a Morita equivalence $A(X /\!/ G) \simeq C_\mathrm{c}^\infty(X/G)$. 
\end{Example}

\begin{Example}[Gauge groupoid]
\label{ex:GaugeGrpd}
Let $P \to X$ be a right principal $G$-bundle. The gauge groupoid $\mathrm{Gauge}(P) = (P \times P)/G \rightrightarrows P/G = X$ is the quotient of the pair groupoid of $P$ by the diagonal action of $G$, which is free and proper. An arrow can be viewed as a $G$-equivariant diffeomorphism between two fibers of $P$. As we have seen in Example~\ref{ex:PrincBund}, the space of compactly supported smooth functions on the group quotient is 
\begin{equation*}
  A\bigl(\mathrm{Gauge}(P)\bigr) 
  \cong C_{\mathrm{bc}}^\infty(P \times P)^G
  \,.
\end{equation*}
Explicitly, an element of the algebra is a smooth function $a: P \times P \to \bbC$ such that $(\Supp a)/G$ is compact and $a(p \cdot g,p' \cdot g) = a(p, p')$. The product is given by the convolution product of the pair groupoid. The bibundle $X \leftarrow P \rightarrow *$ with the left action of the gauge groupoid on its base and the right $G$-action is a Morita equivalence. The convolution bimodule is a Morita equivalence
\begin{equation*}
  A\bigl(\mathrm{Gauge}(P)\bigr) 
  \simeq
  A(*/\!/G)
  \,,
\end{equation*}
where the right side is the group convolution algebra of Example~\ref{ex:GroupConv}.
\end{Example}

\begin{Example}[Transitive groupoid]
Every transitive groupoid is isomorphic to the gauge groupoid of the principal bundle $P = s^{-1}(x) \to G_0$ with gauge group the isotropy group $G(x,x) = t^{-1}(x) \cap s^{-1}(x)$. We conclude from Example~\ref{ex:GaugeGrpd}, that the convolution algebra is Morita equivalent to $A(G) \simeq A(*/\!/ G(x,x))$.
\end{Example}

\begin{Example}[Homotopy groupoid of a manifold]
\label{ex:pi1Grpd}
Let $X$ be a connected smooth manifold. The homotopy groupoid $\pi_1\mathrm{Paths}(X) \rightrightarrows X$ has smooth homotopy classes of paths as arrows with the endpoints as source and target. It is isomorphic to the gauge groupoid of the universal covering space $\tilde{X} \to X$ viewed as principal $\pi_1(X)$-bundle, where $\pi_1(X)$ is the fundamental group. The isomorphism depends on the choice of a basepoint in $X$. By Example~\ref{ex:GaugeGrpd}, the convolution algebra is the space
\begin{equation*}
  A\bigl( \pi_1\mathrm{Paths}(X)\bigr)
  \cong
  C_{\mathrm{bc}}^\infty(\tilde{X} \times \tilde{X})^{\pi_1(X)}
  \,,
\end{equation*}
with the convolution product of the pair groupoid. As in Example~\ref{ex:GaugeGrpd}, we have the Morita equivalence
\begin{equation*}
  A\bigl( \pi_1\mathrm{Paths}(X)\bigr)
  \cong
  A\bigl(\mathrm{Gauge}(\tilde{X})\bigr)
  \simeq
  A\bigl(*/\!/ \pi_1(X) \bigr)
  \,.
\end{equation*}
\end{Example}

\begin{Example}[Non-equivalent groupoids with isomorphic convolution algebras]
\label{ex:NonMorButIsomorphic}
Consider the following groupoids. The group $\bbZ_2 = 
\{1,-1\}$ as groupoid over a point and the discrete manifold $\{x_+,x_-\}$ as groupoid with two objects and no other arrows. The groupoids have different orbit spaces and different isotropy groups, so they are not Morita equivalent. As vector space $A(*/\!/\bbZ_2) = \bbC\delta_{1} \oplus \bbC\delta_{-1}$ with $\delta_{1}$ the unit and $\delta_{-1}\delta_{-1} = \delta_1$. The basis
\begin{equation*}
  a_+ := \tfrac{1}{2}(\delta_1 + \delta_{-1})
  \,,\qquad
  a_- := \tfrac{1}{2}(\delta_1 - \delta_{-1})
\end{equation*}
satisfies $a_+ a_+ = a_+$, $a_- a_- = a_-$, and $a_+ a_- = 0 = a_- a_+$. This shows that $A(*/\!/\bbZ_2)$ is isomorphic to $\bbC \times \bbC = A(\{x_+, x_-\})$.
\end{Example}

\subsection{Homomorphisms of Lie groupoids}

To every homomorphism of Lie groupoids $\phi: G \to H$, which consists of a pair of maps $\phi_0: G_0 \to H_0$ and $\phi_1: G_1 \to H_1$, we can associate a $G$-$H$ bibundle
\begin{equation*}
  P_\phi := G_0 \times_{H_0}^{\phi_0, t} H_1
  \,,  
\end{equation*}
with left bundle map $l(x,h) = x$, right bundle map $r(x,h) = s(h)$, left $G$-action $g \cdot (x,h) = (g\cdot x, \phi_1(g)h)$, and right $H$-action $(x,h) \cdot h' = (x, hh')$, whenever defined \cite{Blohmann2008}. $P_\phi$ is right principal.

The bibundles of two homomorphisms $\phi,\psi: G \to H$ are isomorphic if and only if they are naturally equivalent, that is, if there is a smooth map $\tau:G_0 \to G_1$ such that $\tau_{t(g)} \phi_1(g) = \psi_1(g) \tau_{s(g)}$ for all $g \in G_1$. 

\begin{Example}[Terminal homomorphism]
\label{ex:terminalGrpdHom}
The bibundle of the terminal homomorphism $G \to 1$ is $P_{G\to 1} = G_0$ with the left action of $G$ on its base. The convolution module is $C^\infty_\mathrm{c}(G_0)$ with the left $A(G)$-action given by
\begin{equation*}
  (a \cdot m)(x)
  = \int_{t^{-1}(x)} a(g^{-1})\, m(g \cdot x) \, \di \mu_{x}(g)
\end{equation*}
and the $\bbC$-multiplication of the vector space as right action by $A(1) \cong \bbC$. 
\end{Example}

\begin{Example}[Groupoid anchor]
The anchor map of a groupoid $(t,s): G \to G_0 \times G_0$, $g \mapsto (t(g), s(g))$ is a homomorphism to the pair groupoid of its base. The associated bimodule is $P_{(s,t)} = G_0$ with the identity as left and right bundle map, the left action of $G$ on its base and the right action by the pair groupoid on its base. The convolution module is $M(P_{(s,t)}) = C_\mathrm{c}^\infty(G_0)$ with the left $A(G)$-action of Example~\ref{ex:terminalGrpdHom} and the right $\calK(G_0)$-action of Example~\ref{ex:PairGrpd}.
\end{Example}

\begin{Example}[Diagonal action]
\label{ex:DiagonalHom}
The bibundle of the diagonal homomorphism $\Delta_G: G \to G \times G$ is given by
\begin{equation*}
  P_{\Delta_G} 
  \cong
  G_1 \times_{G_0}^{t,t} G_1
  \,,
\end{equation*}
with the left $G$-action $g \cdot (h_1, h_2) = (gh_1, gh_2)$ and the right $(G\times G)$-action $(h_1, h_2) \cdot (g_1, g_2) = (h_1 g_1, h_2 g_2)$. Let $P$ and $Q$ left $G$-bundles. Then
\begin{equation*}
  P_{\Delta_G} \circ_{G \times G} (Q \times R)
  \cong Q \times_{G_0}^{l,l} R
\end{equation*}
with the left diagonal $G$-action $g \cdot (q,r) = (g\cdot q, g\cdot r)$. Given two smooth left $A(G)$-modules $M$ and $M'$, the smooth left $A(G)$-module
\begin{equation*}
  M\bullet M' := 
  M(P_{\Delta_G}) \Cotimes_{A(G) \Cotimes A(G)} (M \Cotimes M')
\end{equation*}
equips the category of smooth left $A(G)$-modules with a weak symmetric monoidal structure. The unit of $\bullet$ is given by the module of the terminal homomorphism $G \to 1$ of Example~\ref{ex:terminalGrpdHom}.
\end{Example}

\begin{Example}[Points of a groupoid]
Let $x \in G_0$ be a point of the base of the groupoid $G$. The maps $x_0\colon * \to G_0$, $* \mapsto x$ and $x_1\colon * \to G_1$, $* \mapsto 1_x$ define a morphism of groupoids $x\colon 1 \to G$.
Its bibundle is $P_{x} = t^{-1}(x)$ with the right $G$-action given by groupoid multiplication. Its convolution bimodule is $M(P_x) = C_\mathrm{c}^\infty(t^{-1}(x))$ with the left $A(1) \cong \bbC$ action given by scalar multiplication and the right action given by restriction of groupoid convolution.

The composition $1 \xrightarrow{x} G \to 1$ is the identity. With Example~\ref{ex:terminalGrpdHom}, it follows from the functoriality of convolution that
\begin{equation*}
  C^\infty_\mathrm{c}(t^{-1}(x)) \Cotimes_{A(G)} C^\infty_\mathrm{c}(G_0)
  \cong \bbC
  \,.
\end{equation*}
This shows that the bimodule $M(P_x)$ is a weakly split monomorphism from the tensor unit $\bbC$ to $A(G)$. In this sense, points of $G_0$ give rise to $\bbC$-points in $A(G)$.
\end{Example}

\begin{Example}[Identity bisection]
The identity bisection $1\colon G_0 \hookrightarrow G_1$ is a homomorphism of groupoids. The associated bibundle is $P_1 = G_1$, with target and source map as left and right bundle maps. On the left, $G_0 \rightrightarrows G_0$ acts trivially. On the right, $G$ acts by groupoid multiplication. It follows that $A(G)$ is a module over the nonunital commutative algebra $C_\mathrm{c}^\infty(G_0)$.
\end{Example}

The bibundle $P_\phi$ of a homomorphism $\phi: G \to H$ is a Morita equivalence if and only if there is a homomorphism $\psi\colon H \to G$ such $\phi \circ \psi$ and $\psi \circ \phi$ are smoothly naturally equivalent to the identity. The following criterion is more useful. Let $G(x,x') := t^{-1}(x) \cap s^{-1}(x') \subset G_1$ denote the submanifold of arrows to $x$ from $x'$.

\begin{Proposition}
\label{prop:BibuHomMorEquiv}
$P_\phi$ is a Morita equivalence if and only if
\begin{itemize}

\item[(i)] $\phi_1$ restricts to a diffeomorphism $G(x, x') \to H(x,x')$ for all $x,x' \in G_0$;

\item[(ii)] $\phi_0(H_0)$ intersects every orbit of $H$;

\item[(iii)] $T\phi_0(T_x G_0) + Tt(T_h H_1) = T_{t(h)} H_0$ for all $(x,h) \in P_\phi$.

\end{itemize}
\end{Proposition}
\begin{proof}[Sketch of proof]
Condition~(i) means that $\phi$ is smoothly full and faithful, Condition~(ii), that $\phi$ is essentially surjective, and Condition~(iii) that $\phi_0(G_0)$ intersects the orbits of $H$ transversely.
\end{proof}

\begin{Example}[Morita equivalent subgroupoids]
\label{ex:MorSubgrpd}
Let $H$ be a Lie groupoid. Let $\phi_0\colon S \hookrightarrow H_0$ be an embedded submanifold that intersects all orbits transversely. That is, $\phi_0$ satisfies Conditions~(i) and (ii) of Proposition~\ref{prop:BibuHomMorEquiv}. Such a submanifold is called a \textdef{complete transversal}. Let $G$ be the full subgroupoid of $H$ over $S$. That is, $G_0 = S$ and $G(s,s') = H(s,s')$ for all $s,s' \in S$. Then the bibundle of the inclusion $\phi\colon G \to H$ is a Morita equivalence $G \simeq H$. It follows that $A(H) \simeq A(G)$.
\end{Example}

\begin{Example}
Let $G$ be a Lie groupoid and $U \subset G_0$ an open subset that intersects every orbit of $G$. Then the full Lie subgroupoid over $U$ is Morita equivalent to $G$.
\end{Example}

\begin{Example}[\v{C}ech groupoid of a cover]
Let $\calU = \{U_i\}_{i\in I}$ be an open cover of the manifold $X$. Let $C(\calU)$ denote its \v{C}ech groupoid $\bigsqcup_{i,j}U_{ij} \rightrightarrows \bigsqcup U_i$, where $U_{ij}:=U_i\cap U_j$. The convolution algebra is $A(C(\calU)) = \bigoplus_{i,j} C_\mathrm{c}^\infty(U_{ij})$ with the product of matrix multiplication induced from pointwise multiplication $C_\mathrm{c}^\infty(U_{ij}) \Cotimes C_\mathrm{c}^\infty(U_{jk}) \to C_\mathrm{c}^\infty(U_{ijk}) \hookrightarrow C_\mathrm{c}^\infty(U_{ik})$. The orbit space of the \v{C}ech groupoid is the coequalizer $\bigsqcup_{i,j}U_{ij} \rightrightarrows \bigsqcup U_i \to X$. Since $C(\calU)$ has trivial isotropy at all points, the bibundle of the natural epimorphism of groupoids $C(\calU)\to (X \rightrightarrows X)$ is a Morita equivalence. By applying the convolution functor, we obtain the Morita equivalence
\begin{equation*}
   A\bigl(C(\calU)\bigr) \simeq C_\mathrm{c}^\infty(X)
   \,.
\end{equation*}
\end{Example}

\subsection{Crossed products}
\label{sec:CrossedProd}

Let $G$ be a Lie group, $B$ a complete bornological algebra, and $\alpha: G \to \Aut(B)$ a smooth group homomorphism. The triple $(G,B,\alpha)$ will be called a \textdef{bornological dynamical system}. Let $\mu$ be a right invariant Haar measure on $G$. Then we can equip the complete bornological vector space 
\begin{equation*}
  B \rtimes_\alpha G 
  := C^\infty_\mathrm{c}(G, B)
\end{equation*}
with the associative product
\begin{equation*}
  (b \cdot b')(g) := \int_{h \in G}  \alpha_{h^{-1}} b(gh^{-1}) \,b'(h)
  \,,
\end{equation*}
The algebra $B \rtimes_\alpha G$ is called the \textdef{crossed product} of the dynamical system.

Let $G \times X \to X$ be a smooth action of a Lie group on a manifold $X$. Let $B := C^\infty_\mathrm{c}(X)$. Then 
\begin{equation*}
\begin{aligned}
  \alpha: G &\longrightarrow \Aut(B)
  \\
  (\alpha_g b)(x) 
  &= b(g^{-1} \cdot x)
\end{aligned}
\end{equation*}
is a smooth group homomorphism. The convolution algebra of the action groupoid is the vector space $A(G \ltimes X) \cong C^\infty_\mathrm{c}(G\times X) \cong C^\infty_\mathrm{c}(G, B)$. It is straightforward to check that the convolution product is given by the product of the crossed module, $b * b' = b \cdot b'$.  
This shows that the convolution algebra of the action groupoid is isomorphic to the crossed product,
\begin{equation*}
  A(G \ltimes X) \cong C^\infty_\mathrm{c}(X) \rtimes_\alpha G
  \,.
\end{equation*}

The linear span of pure tensors $a \Cotimes b \in C_\mathrm{c}^\infty(G) \Cotimes C_\mathrm{c}^\infty(X) \cong C_\mathrm{c}^\infty(G\times X)$ is a dense subspace. Assume that $X$ is compact, so $C_\mathrm{c}^\infty(X)$ contains the constant function $1$. Then $(a \Cotimes 1) * (a' \Cotimes 1) = (a * a') \Cotimes 1$, where $a * a'$ is the product of group convolution. 

Assume further that $G$ is discrete, so $C_\mathrm{c}^\infty(G)$ is spanned by the functions $\delta_g$ supported at $g \in G$. The convolution product is given by $\delta_g * \delta_h = \delta_{gh}$. In particular, $\delta_e = 1$ is the unit. Then $(1 \Cotimes b)(1 \Cotimes b') = 1 \Cotimes bb'$, where $bb'$ is the pointwise product of functions in $C_\mathrm{c}^\infty(X)$. Moreover, $(a \Cotimes 1) * (1 \Cotimes b) = a \Cotimes b$ and
\begin{equation*}
  [(1 \Cotimes b) * (a \Cotimes 1)](x,g)
  = a(g) \Cotimes \alpha_{g^{-1}} b
  \,.
\end{equation*}
In particular, we have
\begin{equation*}
  (\delta_g \Cotimes 1) * (1 \Cotimes b) * (\delta_{g^{-1}} \Cotimes 1) = 1 \Cotimes \alpha_g b
  \,.
\end{equation*}

\subsection{Bornological noncommutative tori}

We recall that the leaves of the Kronecker foliation of the 2-torus $\bbT^2 \cong S^1 \times S^1$ of irrational slope $\theta$ are the orbits of the action $\bbR \times \bbT^2 \to \bbT^2$ defined by
\begin{equation*}
  t \cdot ([r], [s]) 
  := ([r + t], [s + \theta t])
  \,,
\end{equation*}
where $[r], [s] \in \bbR/\bbZ \cong S^1$. The submanifold $\phi_0\colon S^1 \cong S^1 \hookrightarrow \{[0]\} \times S^1 \subset \bbT^2$ is a complete transversal to the orbits. The full subgroupoid over $\{[0]\} \times S^1$ is the groupoid of the action $\bbZ \times S^1 \to S^1$, $k \cdot [r] := [r + k\theta]$. It follows from Example~\ref{ex:MorSubgrpd} that the bibundle of the inclusion $\phi: \bbZ \times S^1 \hookrightarrow \bbR \times \bbT^2$,
\begin{equation*}
\begin{tikzcd}
    \bbZ \ltimes S^1
        \ar[d, shift left=-1.2, near start]
        \ar[d, shift left=1.2, near start]
        \ar[r, phantom, "\circlearrowright"]
    &[-1em] 
    P_\phi
        \ar[ld]
        \ar[rd]
    &[-1em] 
    \bbR\ltimes \mathbb{T}^2
        \ar[d,  shift left=-1.2, near start]
        \ar[d,  shift left=1.2, near start] 
        \ar[l, phantom, "\circlearrowleft"]
    \\
S^1 & & \mathbb{T}^2 
\end{tikzcd}
\end{equation*}
is a Morita equivalence. It is given explicitly by
\begin{equation*}
  P_\phi
  = 
  S^1 \times_{\bbT}^{\phi_0,t} (\bbR \ltimes \bbT)
  \cong S^1 \times \bbR
\end{equation*}
with the bundle maps
\begin{align*}
  l([s],t) 
  &= [s]
  \\
  r([s],t) 
  &= ( [-t], [s - \theta t] )
  \\ 
\intertext{and the left and right actions}
  (k, [s])\cdot ([s],t)
  &= ([s + \theta k], k+t)
  \\  
  ([s],t)\cdot \bigl(t', ([-t - t'], [s - \theta t - \theta t'])\bigr)
  &=
  ([s],t+t')
  \,,
\end{align*}
for all $k \in \bbZ$, $t, t' \in \bbR$, and $[s] \in S^1$. 

The convolution algebra
\begin{equation*}
  \calA_\theta := A(\bbZ \ltimes S^1)
\end{equation*}
will be called a \textdef{bornological noncommutative torus} or noncommutative torus for short. When the slopes of two Kronecker foliations are related by
\begin{equation*}
  \theta' = \frac{a\theta + b}{c\theta+d}
  \quad\text{for some}\quad
  \begin{pmatrix}
     a & b\\
     c & d 
  \end{pmatrix} 
  \in \mathrm{GL}_2(\bbZ)
  \,,
\end{equation*}
then noncommutative tori are Morita equivalent, $\calA_\theta \simeq \calA_{\theta'}$. This can be shown by choosing two different complete transversals in $\bbT^2$. We posit that the converse statement is also true, as it is the case in the setting of $C^*$-algebras \cite[Theorem~4]{Rieffel:1981}.

$\bbZ$ is a discrete group that acts on the compact manifold $S^1$. As explained in Section~\ref{sec:CrossedProd}, the convolution algebra of the action groupoid is the crossed product
\begin{equation*}
  \calA_\theta \cong C_\mathrm{c}^\infty(S^1) \rtimes_\alpha \bbZ
  \,,
\end{equation*}
where the action $\alpha: \bbZ \to \Aut(B)$, $B := C_\mathrm{c}^\infty(S^1)$, is given by $(\alpha_k b)([r]) = [r - \theta k]$. 

The convolution algebra of $\bbZ$ is the $*$-algebra generated by the characteristic function $\delta_1$. The unit is $1 = \delta_0$. The $*$-subalgebra of $C_\mathrm{c}^\infty(S^1)$ generated by the map $S^1 \to \bbC$, $[r] \mapsto e^{2i\pi r}$ is a Fourier basis, so it is dense in $C_\mathrm{c}^\infty(S^1)$. The unit is the constant function $1$. It follows that the $*$-subalgebra of $\calA_\theta$ generated by
\begin{equation*}
  u := \delta_1 \Cotimes 1
  \,,\qquad
  v := 1 \Cotimes e^{2\pi i r}
  \,,
\end{equation*}
is dense in $\calA_\theta$. The generators satisfy the commutation relations
\begin{equation*}
  v * u = e^{2\pi i \theta} u * v 
  \,.
\end{equation*}

\subsection{Ideals}

There is a number of possible definitions of ideals in bornological algebras. In the category $\AlgMrt(\cBorn)$, the appropriate notion for our purposes is the following. 

\begin{Definition}
An \textdef{ideal} of a complete bornological algebra $A \in \AlgMrt(\cBorn)$ is a strong subobject of $A$ in the category $\Mod(A,A)$ of smooth $A$-$A$ bimodules. 
\end{Definition}

Let us spell out this definition in conventional terms. The strong subobject in $\cBorn$ represented by a strong monomorphism $i: I' \hookrightarrow A$ can be identified with its image $I = i(I') \subset A$ equipped with the subspace bornology. Since $I'$ is complete, $I$ is a closed subspace. Denote the image of $A \Cotimes I \hookrightarrow A \Cotimes A \to A$ by $AI$ and analogously by $IA$ for the right $A$-multiplication. If $I' \hookrightarrow A$ is a subobject in $\Mod(A,A)$, then $AI \subset I$ and $IA \subset I$. Since $I' \in \Mod(A,A)$, $I$ is smooth as an $A$-$A$ bimodule. We conclude that an ideal is a closed subspace of $A$ that satisfies $AI \subset I$, $IA \subset I$, and is smooth as an $A$-$A$ bimodule.

\begin{Proposition}
\label{prop:IdealsMoritaInv}
The lattice of ideals is Morita invariant.
\end{Proposition}
\begin{proof}
A Morita equivalence bimodule $P\in \Mod(A,B)$ with inverse $Q\in \Mod(B,A)$ induces an equivalence of categories $\Mod(A,A) \to \Mod(B,B)$, $M\mapsto Q \Cotimes_A M \Cotimes_A P$. Strong monomorphisms are invariant under equivalences of categories, so the equivalence induces an isomorphism of the categories of strong subobjects.
\end{proof}

As usual, an algebra $A$ is called \textdef{simple} if $0$ and $A$ are its only ideals. It follows from Proposition~\ref{prop:IdealsMoritaInv} that the simplicity of algebras is Morita invariant.

\begin{Example}
Let $G$ be a Lie groupoid and $U \subset G_0$ an open union of $G$-orbits. Let $G(U,U)$ denote the full Lie groupoid over $U$. Then $G(U,U)_1 \subset G_1$ is open, so $A(G(U,U)) = C_\mathrm{c}^\infty(G(U,U)) \hookrightarrow C_\mathrm{c}^\infty(G_1) = A(G)$ is an ideal. If $U$ is not empty, the ideal is not zero. If $U \neq G_0$, the ideal is not $A$. Consider the action groupoid $\bbZ \ltimes S^1$, where $\bbZ$ acts by rotation by a rational angle. Then all orbits are closed, so $U = G_0 \setminus (\bbZ\cdot[s])$ is open. The corresponding ideal in $\calA_\theta$ is proper. This shows that the noncommutative torus $\calA_\theta$ for $\theta$ rational is not simple.
\end{Example}

\begin{Theorem}
\label{thm:irrationaltorussimple}
If $\theta$ is irrational, the bornological noncommutative torus $\calA_\theta$ is simple.
\end{Theorem}

\begin{Corollary}
If $\theta$ is irrational, $A(\bbR\ltimes \mathbb{T}^2)$ is simple.
\end{Corollary}

\begin{proof}[Proof of Theorem~\ref{thm:irrationaltorussimple}]
The proof is an adaptation of the proof of Theorem~VI.1.4 in \cite{CStarExamples} to the bornological setting. Every $a \in \calA_\theta$ is of the form $a = \sum_k \delta_k \Cotimes a_k$ where the sum is finite and $a_k \in C_\mathrm{c}^\infty(S^1)$. Since $\calA_\theta$ is complete, we can construct the element
\begin{equation*}
\begin{split}
  \Phi_1(a)
  &:= \lim_{n \to \infty} 
     \frac{1}{2n+1}\sum_{j=-n}^n u^j * a * u^{-j}
  \\
  &= \lim_{n \to \infty}
     \sum_k \Bigl(\delta_k \Cotimes 
     \frac{1}{2n+1}\sum_{j=-n}^n a_k(t-j\theta) \Bigr)
  \\
  &= \sum_k \delta_k \Cotimes 
     \lim_{n \to \infty} \frac{1}{2n+1}\sum_{j=-n}^n a_k(t-j\theta)
  \\
  &= \sum_k \delta_k \Cotimes {\textstyle\int_{S^1}} a_k
\end{split}
\end{equation*}
where we have first expanded $a$ as a finite sum over $k$, then commuted the limit with the finite sum and with the bounded linear map $c \mapsto \delta_k \Cotimes c$. The last equality follows from the uniform convergence of Riemann sums and from the fact that as $n\to \infty$ each orbit of the irrational action uniformly approaches an equidistribution on $S^1$ . Similarly, we construct the element
\begin{equation*}
\begin{split}
  \Phi_2(a)
  &:= \lim_{n\to \infty} \frac{1}{2n+1} \sum_{j=-n}^n v^j* a * v^{-j} 
  \\
  &= \lim_{n \to \infty} \sum_k \Bigl(  \frac{1}{2n+1}\sum_{j=-n}^n e^{2\pi i j k\theta} \delta_k
      \Cotimes a_k  \Bigr)
  \\
  &= \sum_k \Bigl( \lim_{n \to \infty} \frac{1}{2n+1}\sum_{j=-n}^n e^{2\pi i j k\theta} \delta_k \Bigr)
      \Cotimes a_k 
  \\
  &= \sum_k \Bigl(\lim_{n\to \infty} \frac{1}{2n+1}
     \frac{\sin(\pi (2n+1) k\theta)}{\sin(\pi k\theta)} 
     \delta_k\Bigr) \Cotimes a_k
  \\
  &= \delta_0 \Cotimes a_0
  \,,
\end{split}    
\end{equation*}
where we have first expanded $a$, then commuted the limit with the finite sum over $k$ and with the bounded linear map $c \mapsto c \Cotimes a_k$, and used that $\theta$ is irrational to obtain the formula in terms of the sine functions.

Let $I \subset \calA_\theta$ be a non-zero ideal, so it contains a non-zero element $a \in I$. Consider the positive element $b = a^* * a \in I$, where $a^*$ is the $*$-structure~\eqref{eq:StartStruct} on $A_\theta$. Explicitly,
\begin{equation*}
  b =  \sum_n\bigl(  
       \delta_n \Cotimes e^{2\pi i n} \sum_k \overline{a_k}\, a_{k+n}
       \bigr)
  \,,
\end{equation*}
which has the zero component $b_0 = \sum_k |a_k|^2$. Consider
\begin{equation*}
  c 
  :=(\Phi_1 \circ \Phi_2)(b) 
  = \Phi_1(\delta_0 \Cotimes b_0)
  = \delta_0 \Cotimes {\textstyle\int_{S^1}} b_0
  \,.
\end{equation*}
Since $I$ is a closed subspace, $c \in I$. Since $a$ is non-zero, one of the summands $|a_k|$ of $b_0$ is a non-zero positive function, so that $\nu := \int_{S^1} b_0$ is a positive real number. We conclude that $c$ has an inverse $c^{-1} = \delta_0 \Cotimes \nu^{-1}$. We conclude that $I = A$.
\end{proof}

\subsection{Stacky Lie groups and Morita Hopf monoids}
\label{sec:HopfMonoid}

A \textdef{stacky Lie group} is a weak 2-group in $\GrpdMrt$ \cite{Blohmann2008,Schommer-Priess2011}. Explicitly, a stacky Lie group is given by a Lie groupoid $G$ together with the $(G \times G)$-$G$ bibundle of multiplication $\hat{\mu}$, the $1$-$G$ bibundle of the inclusion of the neutral element $\hat{e}$, the $G$-$G$ bibundle of the inverse $\hat{\imath}$, subject to the relations and coherence conditions of a weak group object. The axioms of the group structure also involve the $G$-$1$ bibundle $P_{G\to 1}$ of the terminal morphism from Example~\ref{ex:terminalGrpdHom} and the $G$-$(G\times G)$ bibundle $P_{\Delta_G}$ of the diagonal map of Example~\ref{ex:DiagonalHom}. For example, the left inverse relation is expressed by
\begin{equation*}
  P_{\Delta_G} \circ_{G\times G} 
  (\hat{\imath} \times G) \circ_{G \times G} \hat{\mu}
  \cong
  P_{G\to 1} \circ_{1} \hat{e}
  \,.
\end{equation*}
As is the case for group objects in any category with finite products, $G$ with multiplication $\hat{\mu}$, unit $\hat{e}$, coproduct $P_{\Delta_G}$, counit $P_{G\to 1}$, and antipode $\hat{\imath}$ is a (weak) \textdef{Hopf monoid} \cite{AdamekHerrlich:1985} in $\GrpdMrt$.

\begin{Proposition}
\label{prop:HopfMonoid}
Let $(G, \hat{\mu}, \hat{e}, \hat{i})$ be a stacky Lie group. Then $A(G)$ together with the coproduct $\boldsymbol{\Delta} := M(\hat{\mu})$, counit $\boldsymbol{\epsilon} = M(\hat{e})$, product $\boldsymbol{m} := M(P_{\Delta_G})$, unit $\boldsymbol{\eta} := M(P_{G\to 1})$, and antipode $\boldsymbol{S} := M(\hat{\imath})$ is a Hopf monoid in $\AlgMrt(\cBorn)$. 
\end{Proposition}
\begin{proof}
By Theorem~\ref{thm:ConvFunc}, the bornological convolution functor is monoidal, so it takes Hopf monoids to Hopf monoids.
\end{proof}

In Proposition~\ref{prop:HopfMonoid} we have chosen to view convolution as contravariant functor, since from the viewpoint of noncommutative geometry, $A(G)$ plays the role of the algebra of functions on the stack presented by $G$. But since the axioms of Hopf monoids are self dual, there is no difference between Hopf monoids in $\AlgMrt(\cBorn)$ and in $\AlgMrt(\cBorn)^\op$.

\begin{Example}
The action groupoid $\bbZ \ltimes S^1$ of irrational rotation is a stacky Lie group. The stacky group structure has been described in \cite{BlohmannTangWeinstein:2008} in some detail. It follows from Proposition~\ref{prop:HopfMonoid} that $\calA_\theta = A(\bbZ \ltimes S^1)$ is a Hopf monoid in $\AlgMrt(\cBorn)$. This endows the category of representations of $\calA_\theta$ with additional structure, which will be a topic for future research. 
\end{Example}

\subsection{Star structure}


We recall that a \textdef{$*$-structure} on a complex algebra is a map $*: A \to A$ that is
\begin{itemize}

\item[(i)] conjugate linear, $(a + b)^* = a^* + b^*$ and $(\lambda a)^* =  \bar{\lambda} a^*$,

\item[(ii)] an antihomomorphism, $(ab)^* = b^* a^*$, and

\item[(iii)] an involution, $(a^*)^* = a$.

\end{itemize}
for all $a, b \in A$ and $\lambda \in \bbC$. An algebra together with a $*$-structure is called a \textdef{$*$-algebra}. When $A$ is a bornological algebra, we additionally require the map $*$ to be bounded. The map $*:A(G) \to A(G)$ defined by
\begin{equation}
\label{eq:StartStruct}
  a^*(g) := \overline{ a(g^{-1}) }
\end{equation}
is a $*$-structure on the convolution algebra $A(G)$. This shows that the convolution functor factors through the 2-category $*$-$\AlgMrt(\cBorn)$ where all algebras are in addition equipped with a $*$-structure.


Let $M$ be an $A$-$B$ a bimodule of complete bornological $*$-algebras. We can define an $B$-$A$ module $M^\dagger$ as the complex conjugate vector space 
\begin{equation*}
  M^\dagger := \bar{M}
\end{equation*}
with left $B$-action $b\cdot m := m \cdot b^*$ and right $A$-action $m \cdot a := a^* \cdot m$. The 2-functor
\begin{equation*}
  *\textnormal{-} \AlgMrt(\cBorn) \longrightarrow *\textnormal{-}\AlgMrt(\cBorn)^{1\textnormal{-}\op}
\end{equation*}
given by $A \mapsto A$, $M \mapsto M^\dagger$, $f \mapsto f$ is a bottom dagger structure in the sense of \cite{FerrerEtAl:2024}.

\subsection{Dualizability}

For every groupoid $G$, let $G^\op$ denote the opposite groupoid. Since every left $G$-action can be viewed as a right $G^\op$-action, every $G$-$H$ groupoid bibundle $P$ can be viewed as a $H^\op$-$G^\op$ bibundle. We can also move a right action to the left and vice versa. In this way, we can view $P$ as a left $(G \times H^\op)$ bundle or a right $(G^\op \times H)$ bibundle. In general, neither bibundle will be right principal. An interesting example is the identity bibundle $G$. Let us denote the corresponding left $(G \times G^\op)$ bundle by $\Ev_G$, the right $(G^\op \times G)$ bibundle by $\Coev_G$. These bibundles satisfy the ``snake'' relations up to 2-isomorphism. However, $\Ev_G$ and $\Coev_G$ are not right principal, so they cannot be composed with arbitrary morphisms in $\GrpdMrt$. Since $\Ev_G$, and $\Coev_G$ do not lie in $\GrpdMrt$, $G^\op$ is not quite the dual of $G$ in the sense of \cite[Definition~2.3.5]{Lurie:2009}.

\begin{Example}[Coarse moduli space]
$\Coev_G$ composed with the $(G^\op \times G)$-$1$ bibundle $P_{G^\op \to 1} \times P_{G \to 1}$, where $P_{G \to 1}$ is the bibundle of the terminal homomorphism of Example~\ref{ex:terminalGrpdHom} is the coarse moduli space $G_0/G_1$.    
\end{Example}

On the algebraic side, we can view an $A$-$B$-bimodule as a left $A \Cotimes B^\op$-module and as a right $A^\op \Cotimes B$-module. The identity bimodule $A$ can be viewed as left $A \Cotimes A^\op$-module, which we denote by $\Ev_A$, and as right $A^\op \Cotimes A$-module, which we denote by $\Coev_A$. These modules satisfy the snake relations up to 2-isomorphism. This shows that every self-induced complete bornological algebra is dualizable in $\AlgMrt(\cBorn)$ in the sense of \cite[Definition~2.3.5]{Lurie:2009}. In the purely algebraic setting, an algebra $A$ is \emph{fully} dualizable if it is finitely generated and separable (Definition~\ref{def:Separable}). We have already seen in Corollary~\ref{cor:ConvMoritaProjective} and Corollary~\ref{cor:quasiunitality} that bornological convolution algebras are left and right separable. Which complete bornological algebras are fully dualizable is a question for future research.

\subsection{Hochschild cohomology}
\label{sec:NoncommDiffGeo}

As recalled in the introduction, a $C^*$-algebra that is nuclear or does not admit a bounded trace has vanishing Hochschild cohomology. This includes many $C^*$-convolution algebras of Lie groupoids. In light of the Hochschild-Kostant-Rosenberg theorem, this means that the noncommutative geometry described by the $C^*$-algebra does not admit vector fields. The situation for bornological convolution algebras is better:

\begin{Proposition}[Proposition~2.1 in \cite{Posthuma23}]
Multiplicative vector fields on the Lie groupoid act as derivations on its bornological convolution algebra.
\end{Proposition}

The authors in \cite{Posthuma23} go on to define a cochain map between the deformation cohomology of Lie groupoids and the Hochschild cohomology of the convolution algebra. To define Hochschild or cyclic cohomology of topological algebras we want to impose continuity conditions on the cocycles. However, due to the pathologies of locally convex vector spaces (Remark~\ref{rmk:lcTVSpathology}), this only works for nice subcategories such as Fr\'{e}chet algebras with the projective tensor product (cf.~\cite{Connes1985}). As advocated in \cite{Meyer2007}, it may be more natural to study homological algebra in $\cBorn$, where we can define Hochschild cohomology in the usual way as the derived functor \cite{Aretz23}
\begin{equation*}
  \mathrm{HH}^*(A,M):= \RHom_A(A,M) \,.
\end{equation*}
Recall that $A$ is \textbf{H-unital} if its bar complex is acyclic. For H-unital algebras, it can be shown that the explicit models using the bar resolution compute the $\Ext$-groups of $M$. It was proved in \cite[Proposition~2]{CrainicMoerdijk:2001} that the convolution algebras of Lie groupoids are H-unital. This also follows from the left and right separability of the convolution algebras (Corollary~\ref{cor:quasiunitality}) together with \cite[Lemma~5.70]{Aretz23}.

\begin{Remark}
    H-unitality depends on the tensor product that is used.
    Recent work by Michael Francis \cite{Francis2025} shows that the Lie groupoid convolution algebras are H-unital also with respect to the algebraic tensor product.
\end{Remark}

In \cite{Neumaier2006} a general type of HKR theorem for proper étale Lie groupoids is obtained.
In \cite[Appendix~A]{OrbifoldCupProducts2011}, the authors develop a bornological Morita theory close to the present article and prove partial results on the Morita invariance of Hochschild cohomology in the proper étale case.
    

\appendix

\section{Background material}

\subsection{Monomorphisms and epimorphisms}
\label{sec:monosandEpis}

Recall that a morphism $f:X\to Y$ in a category $\calC$ is a \textdef{monomorphism} if $fg = fh$ implies $g = h$ for any two morphisms $g,h:Z \to X$. Dually, $f$ is an \textdef{epimorphism} if $gf = hf$ implies $g = h$. There are stronger notions of epimorphisms that we will need:

\begin{Definition}
\label{def:Epis}
A morphism $f:X \to Y$ is called
\begin{itemize}

\item[(i)] a \textdef{split epimorphism} if there is a morphism $g: Y \to X$ such that $fg = \id_Y$;

\item[(ii)] 
a \textdef{regular epimorphism} if there is a coequalizer diagram
\begin{equation*}
\begin{tikzcd}
  Z \ar[r,shift left=1.2]\ar[r,shift right =1.2]& X \ar[r,"f"] & Y 
  \,;
\end{tikzcd}
\end{equation*}

\item[(iii)]
a \textdef{strong epimorphism} if it has the right lifting property with respect to all monomorphisms $i: A \to B$, that is, there is a diagonal lift in any diagram of the form:
\begin{equation*}
   \begin{tikzcd}
      A \ar[r]\ar[d,"i"] & X\ar[d,"f"] \\
      B \ar[r]\ar[ru,dotted,"\exists"] & Y
   \end{tikzcd}
\end{equation*}

\end{itemize}
\end{Definition}

We have the following implications:
\begin{equation*}
  \text{split epi} \implies 
  \text{regular epi} \implies 
  \text{strong epi} \implies 
  \text{epi}
  \,.
\end{equation*}

Let us gather a few facts that we will use:
\begin{itemize}

\item[(i)] Strong epimorphisms are closed under composition.

\item[(ii)] If a morphism is both a monomorphism and a strong epimorphism, then it is an isomorphism.

\item[(iii)] Left adjoint functors preserve epimorphisms and regular epimorphisms.

\item[(iv)] If a functor $U: \calC \to \calD$ is faithful and has a left adjoint, then a morphism $f$ in $\calC$ is a monomorphism if and only $U(f)$ is a monomorphism in $\calD$.

\end{itemize}

\begin{Example}
The forgetful functor out of any category of bornological vector spaces is faithful and has a left adjoint. It follows that the monomorphisms are the injective bounded linear maps.
\end{Example}

\subsection{2-categories and 2-functors}
\label{sec:bicategories}

A \textdef{(weak) $2$-category} or \textdef{bicategory} $\calC$ consists of the following data:
\begin{enumerate}
    \item A set of objects $\mathrm{Obj}(\calC)$. We write $c\in \calC$ for $c\in \mathrm{Obj}(\calC)$.
    \item For any two $c_1,c_2\in \calC$ there is a category of $1$-morphisms $\calC(c_1,c_2)$.
        The objects of this category are called $1$-morphisms, denoted $f:c_1\to c_2$, and the morphisms are called $2$-morphisms, denoted $\alpha:f_1 \Rightarrow f_2$.
        \[
        \begin{tikzcd}[column sep=large]
            c_1 \ar[r,"f_1",""'{name=A}, bend left=30]\ar[r,"f_2"',""{name = B},bend right = 30]\ar[r,from=A,to=B, "\alpha",Rightarrow]& c_2
        \end{tikzcd}
        \]
        There are identity $1$-morphisms $\id_{c}\in \calC(c,c)$.
    \item Horizontal composition functors $\circ:\calC(c_2,c_3)\times \calC(c_1,c_2)\to \calC(c_1,c_3)$ and an associator natural isomorphism $\Ass:(f\circ g)\circ h \Rightarrow f\circ (g\circ h)$.
    There are left and right unitor natural isomorphism: $\Unitl: \id_{c_2}\circ f\Rightarrow f$ and $\Unitr:f\circ \id_{c_1}\Rightarrow f$.
    \item These data are subject to certain coherences. We refer to \cite[Chapter 2.1]{Johnson2021} for details.
\end{enumerate}
A \textdef{(weak) $2$-functor} or \textdef{bifunctor} $F:\calC\to \calD$ between bicategories $\calC,\calD$ consists of the following data:
\begin{enumerate}
    
    \item a map $F:\mathrm{Obj}(\calC)\to \mathrm{Obj}(\calD)$;
    
    \item a functor $F_{c_1,c_2}:\calC(c_1,c_2)\to \calD(F(c_1),F(c_2))$;
    
    \item \label{def:coherencetau}
    a natural isomorphism $\tau_{c_1,c_2,c_3}:F_{c_2,c_3}(f)\circ F_{c_1,c_2}(g) \to F_{c_1,c_3}(f\circ g)$, called the \textdef{functoriality constraint}, such that the diagram
    \begin{equation*}
        \begin{tikzcd}
            \calC(c_2,c_3)\times \calC(c_1,c_2) \ar[r,"\circ"]\ar[d,"F_{c_2,c_3}\times F_{c_1,c_2}"']& \calC(c_1,c_3) \ar[d,"F_{c_1,c_3}"]
            \\
            \calD\bigl( F(c_2),F(c_3)\bigr) \times \calD\bigl(F(c_1),F(c_2)\bigr)
            \ar[r,"\circ"']
            \ar[ru,shorten >=5ex,shorten <=5ex, Rightarrow, "\tau_{c_1,c_2,c_3}"']
            & 
            \calD\bigl( F(c_1),F(c_3) \bigr)
        \end{tikzcd}
    \end{equation*}
    commutes;
    
    \item \label{def:coherenceeta}
    an isomorphism $\eta_c:F(\id_c) \to \id_{F(c)}$, called \textdef{unitality constraint}, for all $c\in \mathrm{Obj}(\calC)$.

\end{enumerate}
These data are subject to the following coherence conditions:
\begin{enumerate}
    \item Compatibility with associativity. The following diagram commutes:
    \begin{equation}
    \label{diag:Coherence1}
        \begin{tikzcd}
            \bigl( F(f)\circ F(g) \bigr) \circ F(h) \ar[r,"\tau \circ \id"]\ar[d,"\Ass_\calC"']
            & 
            F(f\circ g) \circ F(h) \ar[r,"\tau"] & F\bigl((f\circ g)\circ h \bigr) \ar[d,"F(\Ass_\calD)"]\\
            F(f)\circ \bigl(F(g)\circ F(h) \bigr)
            \ar[r, "\id \circ \tau"']
            &
            F(f)\circ F(g\circ h) \ar[r,"\tau"'] 
            &
            F\bigl(f\circ (g\circ h) \bigr)
        \end{tikzcd}
    \end{equation}
    \item Compatibility with units. The diagram
    \begin{equation*}
        \begin{tikzcd}
            F(f)\circ F(\id_c) \ar[r,"\id\circ \eta"]\ar[d,"\tau"']& F(f)\circ \id_{F(c)} \ar[d,"\Unitr"]\\
            F(f\circ \id) \ar[r,"F(\Unitr)"']& F(f) 
        \end{tikzcd}
    \end{equation*}
    commutes, as well as the analogous diagram for the left unit constraint.
\end{enumerate}
We refer to \cite[Chapter 4.1]{Johnson2021}. In their terminology the above data defines a pseudofunctor.

\subsection{Pullback of a submersion}
\label{sec:Pullbacks}

\begin{Lemma}
\label{lem:PullbackSubfmld}
Let $f: X \to Y$ be a submersion and $g: Z \to Y$ a smooth map of manifolds. Then the pullback $X \times_Y Z$ is a closed embedded submanifold of $X \times Z$.    
\end{Lemma}
\begin{proof}
As topological space, $X \times_Y^{f,g} Z$ is the pullback of $(f,g): Z \to Y \times Y$ by the diagonal $\Delta = (\id,\id): Y \to Y \times Y$. $\Delta$ is a split monomorphism, so a fortiori a strong monomorphism. Strong monomorphisms are stable under pullbacks, so the inclusion $i: X \times_Y Z \hookrightarrow X \times Z$ is a strong monomorphism, that is, an embedding. Since $f(x) = g(z)$ is a closed condition on $X \times Z$, the subspace $X \times_Y Z = \{(x,z) \in X \times Z~|~ f(x) = g(z)\}$ is closed.

Let $(x_0, z_0) \in X \times_Y Z \subset X \times Z$. By the local form theorem of submersions, there is an open neighborhood $V \subset Y$ of $f(x_0) = g(z_0)$ and an open neighborhood $U$ of $x_0$ that is diffeomorphic to a product manifold $U \cong N \times V$, such that $f$ is given on this neighborhood by the projection to $V$. Let $W := g^{-1}(V)$, which is an open neighborhood of $z_0$. Then
\begin{equation*}
\begin{tikzcd}
U \times_V W 
\ar[r, "i", hook]
\ar[d, "\cong"']
&
U \times W
\ar[dd, "\cong"]
\\
N \times V \times_V W
\ar[d, "\cong"']
&
\\
N \times W 
\ar[r, "j"', hook]
&
N \times V \times W
\end{tikzcd}
\end{equation*}
where $j(n,w) = (n, g(w), w )$. Since $N \times W$ is the graph of a smooth map, it is a submanifold of $N \times V \times W$. We conclude that $U \times_V W \subset U \times W$ is a closed embedded submanifold.
\end{proof}

\begin{Corollary}
\label{cor:PullbackRetract}
Let $f: X \to Y$ be a submersion and $g: Z \to Y$ a smooth map of manifolds. Then every $x_0 \in X$ has an open neighborhood $U \subset X$ with a smooth map 
\begin{equation*}
  \zeta: U \times f(U) 
  \longrightarrow 
  U
  \,,
\end{equation*}
such that $f\bigl(\zeta(x, y)\bigr) = y$ and $\zeta\bigl( x, f(x) \bigr) = x$ for all $x \in U$, $y \in f(U)$.
\end{Corollary}

\begin{Remark}
\label{rmk:PullbackRetract2}
The properties of the map $\zeta$ of Corollary~\ref{cor:PullbackRetract} imply that the induced map
\begin{equation*}
\begin{aligned}
  U \times g^{-1}\bigl(f(U)\bigr) 
  &\longrightarrow 
  U \times_{f(U)} g^{-1}\bigl(f(U)\bigr)
  \\
  (x,z) 
  &\longmapsto
  \bigl( \zeta(x,g(z)), z \bigr)
\end{aligned}
\end{equation*}
is a local section of the inclusion $X \times_Y Z \to X \times Z$.    
\end{Remark}

\subsection{Proofs about bornology}
\label{sec:AppendixBornology}

We give proofs for three elementary statements about bornological vector spaces, because we could not find proofs in the literature.

\subsubsection{Proof of Proposition~\ref{prop:TensSepIsSep}}
\label{sec:TensSepIsSep}

By Proposition~\ref{prop:SepDivZero} $V \Botimes W$ is separated if and only if $\{0\}$ is its only bounded vector subspace. First, we prove the statement for the case that both $V$ and $W$ are finite dimensional. 

On a finite dimensional vector space, there is a unique separated convex vector space bornology, which is the norm bornology. Since on a finite dimensional vector space all norms are equivalent, this bornology does not depend on the chosen norm.
Also, for finite dimensional vector spaces, all (bi)linear maps are bounded and continuous.
A subset of a finite dimensional vector space is norm bounded if and only if it is contained in compact subset. If $A \subset V$ and $B \subset W$ are compact subsets, then $A \times B$ is compact, so that its image $i(A \times B)$ under the continuous map $i$ from Definition~\ref{def:TensorBorn} is compact, which further implies that the convex hull $\Conv(i(A \times B))$ is compact. It follows from Proposition~\ref{prop:TensorBorn} that if a subsets of $V \Botimes W$ is bounded, then it is contained in a compact subset. The only subspace contained in a compact subset is $\{0\}$, which shows that $V \Botimes W$ is separated.

Let now $V'$ and $W'$ separated convex bornological vector spaces, possibly infinite dimensional. Assume that $V' \Botimes W'$ is not separated. Then there is a bounded one dimensional subspace $U \subset V' \Botimes W'$. Let $u = \sum_{i =1}^k v_i \otimes w_i$ be a basis vector. Then $U$ is contained in the subspace $W \Botimes V$, where $W = \Span\{v_1, \ldots, v_k\} \subset W'$ and $V = \Span\{w_1, \ldots, w_k\} \subset V'$. Any bornological subspace of a separated convex bornological vector space is separated, which implies that $V$ and $W$ are separated. As we have already seen $V \Botimes W$ is separated, so that $U \subset V \Botimes W$ cannot be bounded, which is a contradiction.
\qed

\subsubsection{Proof of Proposition~\ref{prop:MappinBornInnerHom}}
\label{sec:MappinBornInnerHom}

The adjunction~\eqref{eq:TensorInnHomAdj} of the underlying vector spaces maps a morphism $f: U \Botimes V \to W$ in $\Born$ to the map $\tilde{f}: U \to \intBorn(V,W)$ defined by
\begin{equation*}
  \tilde{f}(u):
  V 
  \xrightarrow{~\cong~}
  \{u\} \times V
  \longhookrightarrow
  U \times V
  \xrightarrow{~i~}
  U \Botimes V
  \xrightarrow{~f~}
  W
  \,.
\end{equation*}
since all arrows represent bounded maps, $\tilde{f}(u)$ is bounded. Let $A \subset U$ and $F = \tilde{f}(A) \subset \Born(V,W)$. For every bounded $B \subset V$ the subset $F(B) \subset W$ defined in Proposition~\ref{prop:UnifBoundBorn} is the image of the map
\begin{equation*}
  A \times B
  \longhookrightarrow
  U \times V
  \xrightarrow{~i~}
  U \Botimes V
  \xrightarrow{~f~}
  W
  \,,
\end{equation*}
which is the bounded subset $f(A \otimes B)$.

Conversely, to every morphism $\tilde{f}: U \to \intBorn(V,W)$ we can associate the bilinear map
\begin{align*}
  f: U \times V
  &\longrightarrow W
  \\
  (u,v) &\longmapsto \tilde{f}(u)(v)
  \,.
\end{align*}
Since $\tilde{f}$ is bounded, it sends a bounded subset $A \subset U$ to a bounded set $F := \tilde{f}(A)$. By definition of the bornology on $\intBorn(V,W)$, this means that for every bounded $B \subset V$,
\begin{equation*}
  F(B) 
  = \tilde{f}(A)(B)
  = f(A \times B)
\end{equation*}
is bounded. This shows that $f$ is bounded. It follows that the induced morphism $f: U \Botimes V \to W$ given by the universal property of the tensor product is bounded.

We conclude that the natural bijections of the adjunction between the tensor product and the inner homs of vector spaces restricts to convex bornological vector spaces.
\qed

\subsubsection{Proof of Proposition~\ref{prop:InnHomComp}}
\label{sec:InnHomComp}

Since $V$ is separated, there is a collection $\{D_i\}_{i\in I}$ of bounded norming disks in $V$ such that every bounded set is contained in one of the disks. Such a collection is called a basis of (bounded completant) disks. Since $V'$ is complete, there is a basis $\{V'_j\}_{j \in J}$ of bounded completant disks in $V'$. For every map $\sigma: I \to J$ let 
\begin{equation*}
  F_\sigma := \{f \in \Born(V,W) ~|~ f(D_i) \subset D'_{\sigma(i)} 
  \text{~for all~} i \in I \}
  \subset \Born(V,V')
  \,.
\end{equation*}
We want to show that $\{F_\sigma\}_{\sigma \in \Set(I,J)}$ is a basis of bounded completant disks in $\intBorn(V,V')$.

We have $F_\sigma(D_i) \subset D'_{\sigma(i)}$ for all $D_i$, which shows that $F_\sigma$ is bounded in $\intBorn(V,V')$. Since $f$ is $\bbK$-linear, we have
\begin{equation*}
  (\bar{B}_1(\bbK) \cdot f)(D_i)
  = \bar{B}_1(\bbK) \cdot f(D_i)
  \,,
\end{equation*}
which shows that $\bar{B}_1(\bbK) \cdot F_\sigma \subset F_\sigma$. It follows from the linearity of $f$ and the convexity of $D'_i$, that $F_\sigma$ is convex. We conclude that $F_\sigma$ is a bounded disk.

Let $E \subset \Born(V,W)$ and $A \subset V$ be bounded. Then there is an $i \in I$ such that $A \subset D_i$ and a $j \in J$ such that $E(D_i) \subset D_j$. Assuming the axiom of choice, this defines a map $\sigma: I \to J$ such that $E \subset F_\sigma$. This shows that $\{ F_\sigma \}$ is a basis of bounded disks in $\intBorn(V,V')$. It remains to show that all $F_\sigma$ are completant.

Let us denote by $V_i := \Span D_i$ and $V'_j := \Span D'_j$ the normed subspaces. Let $f \in F_\sigma$. Then $f(D_i) \subset D_{\sigma(i)}$, so that $f(V_i) \subset V'_{\sigma(i)}$. Let us denote by
\begin{equation*}
  f|_{V_i}: V_i \longrightarrow V'_{\sigma(j)}  
\end{equation*}
the restriction of $f$. By Remark~\ref{rmk:NormingDiskClosed}, we can assume without loss of generality that $D_i$ is closed in the norm topology of $V_i := \Span D_i$. Then $D_i \subset V_i$ is the unit ball. If $f \in F_\sigma$, then $f(D_i) \subset D_{\sigma(i)}$. It follows that $f|_{V_i}$ is bounded in the operator norm by
\begin{equation*}
  \| f|_{V_i} \| \leq 1
  \,.
\end{equation*}
The semi-norm on $U_\sigma := \Span F_\sigma$ defined by the disk $F_\sigma$ is given by
\begin{equation*}
  \| f\| = \sup_{i \in I} \| f|_{V_i} \|
  \leq 1
  \,.
\end{equation*}
In particular, $\| f|_{V_i} \| \leq \| f \|$. Therefore, $\|f \| = 0$ implies $\| f|_{V_i} \| = 0$, so that $f|_{V_i} = 0$. Since the $V_i$ cover $V$, it follows that $f = 0$, which shows that the semi-norm on $U_\sigma$ is a norm.

To show completeness, assume that $n \mapsto f_n \in \Span F_\sigma$ is a Cauchy sequence. In particular, $f_n$ is norm bounded,
\begin{equation*}
  \| f_n \| \leq c
  \,,
\end{equation*}
for some real number $c < \infty$. Since $\| f_n|_{V_i} \| \leq \| f_n \|$, the sequence $n \mapsto f_n|_{V_i}$ is Cauchy and bounded by $c$ with respect to the operator norm. By assumption, $V'_{\sigma(i)}$ is complete, so that the operator norm on $\Born(V_i, V'_{\sigma(i)})$ is complete. We conclude that $f_n|_{V_i}$ converges to a linear map $g_i: V_i \to V'_{\sigma(i)}$ bounded by $c$. 

Let $v \in V_i \subset V_j$. The value of $g_i(v)$ is the limit of $n \mapsto f_n(v)$ in the bornivorous topology, which is independent of whether we choose the norm on $V_i$ or on $V_j$. It follows that $g_i(v) = g_j(v)$, so that we have a well-defined linear map $g: V \to V'$ such that $g_i = g|_{V_i}$. Moreover, $g$ is bounded by $c$ in $\Span F_\sigma$. We conclude that $\Span F_\sigma$ is norm complete.
\qed

\bibliographystyle{alpha}
\bibliography{Bornology}


\end{document}